\documentclass[12pt,final]{amsart}
\usepackage{amssymb,euscript,nicefrac}
\usepackage{ytableau}
\usepackage{amscd,txfonts}
\usepackage[notref,notcite]{showkeys}
\usepackage{latexsym,todonotes}
\usepackage{xr,xr-hyper}
\usepackage{epsfig,pxfonts}
\usepackage{amsfonts}
\usepackage{xypic}
\usepackage{bbm,ifpdf}
\ifpdf
  \usepackage{pdfsync,hyperref}
\fi
\usepackage[margin=3cm,footskip=25pt,headheight=21pt]{geometry}
\usepackage{fancyhdr,mathrsfs}
\newtheorem{theorem}{Theorem}[section]
\newtheorem{lemma}[theorem]{Lemma}
\newtheorem{proposition}[theorem]{Proposition}
\newtheorem{corollary}[theorem]{Corollary}

\newtheorem{conjecture}[theorem]{Conjecture}
\newtheorem{itheorem}{Theorem}

{
\theoremstyle{plain}
\newtheorem{definition}[theorem]{Definition}
\newtheorem{example}[theorem]{Example}

\newtheorem{remark}[theorem]{Remark}

}
\renewcommand{\theequation}{\arabic{section}.\arabic{equation}}
\numberwithin{equation}{section}
\newcounter{subeqn}
\renewcommand{\thesubeqn}{\theequation\alph{subeqn}}
\newcommand{\subeqn}{%
  \refstepcounter{subeqn}
  \tag{\thesubeqn}
}
\makeatletter
\@addtoreset{subeqn}{equation}
\newcommand{\newseq}{%
  \refstepcounter{equation}
}

\newcommand{\nc}{\newcommand}
\nc{\cat}{\mathcal{V}}
\nc{\func}{\EuScript{T}}
\nc{\res}{\operatorname{res}}

\renewcommand{\Re}{\operatorname{Re}}

\newcommand{\im}{\operatorname{im}}
\newcommand{\KLR}{\operatorname{KLR}}
\newcommand{\id}{\operatorname{id}}

\renewcommand{\dim}{\operatorname{dim}}

\newcommand{\Sym}{\operatorname{Sym}}
\nc{\bQ}{\mathbb{N}}

\newcommand{\Bi}{\mathbf{i}}
\newcommand{\Bj}{\mathbf{j}}
\nc{\pq}{\mathsf{q}}
\nc{\PQ}{\mathsf{Q}}
\nc{\PC}{\mathfrak{C}}

\newcommand{\Bs}{{\underline{\mathbf{s}}}}
\newcommand{\Bt}{{\underline{\mathbf{t}}}}

\nc{\slehat}{\mathfrak{\widehat{sl}}_e}
\nc{\sllhat}{\mathfrak{\widehat{sl}}_\ell}
\nc{\glehat}{\mathfrak{\widehat{gl}}_e}
\nc{\slnhat}{\mathfrak{\widehat{sl}}_n}
\nc{\glnhat}{\mathfrak{\widehat{gl}}_n}
\nc{\eE}{\EuScript{E}}
\newcommand{\arxiv}[1]{\href{http://arxiv.org/abs/#1}{\tt arXiv:\nolinkurl{#1}}}

\nc{\eF}{\EuScript{F}}
\nc{\fF}{\mathfrak{F}}
\nc{\fE}{\mathfrak{E}}

\newcommand{\K}{\mathbbm{k}}

\newcommand{\Z}{\mathbb{Z}}
\newcommand{\Q}{\mathbb{Q}}
\nc{\Qlb}{\mathbb{\bar \Q}_\ell}
\nc{\Fq}{\mathbb{F}_q}
\nc{\Fqb}{\mathbb{\bar F}_q}

\nc{\walg}{W}
\newcommand{\om}{\omega}

\newcommand{\R}{\mathbb{R}}

\newcommand{\C}{\mathbb{C}}

\nc{\KZ}{\mathsf{KZ}}

\newcommand{\la}{\leftarrow}

\newcommand{\sS}{\mathsf{S}} 
 
\newcommand{\sT}{\mathsf{T}}
\newcommand{\sss}{\mathsf{s}} 
\newcommand{\sk}{\mathsf{k}}

\nc{\Bv}{\mathbf{v}}
  \nc{\Bw}{\mathbf{w}}
  \nc{\Bb}{\mathbf{b}}
\nc{\tU}{\mathcal{U}}
\nc{\Bu}{\mathbf{u}}
 \nc{\Fl}{\mathscr{F}\!\ell}
\nc{\Tr}{\operatorname{Tr}}
\nc{\sheafK}{\EuScript{K}}
\nc{\bmu}{\boldsymbol{\mu}}
\nc{\bpi}{\boldsymbol{\pi}}
\nc{\dwalg}{\mathbb{W}}
\nc{\dalg}{\mathbb{T}}
\nc{\aalg}{\mathbb{A}}
\nc{\alm}{\mathscr{A}}
\nc{\bra}{\mathscr{B}}
\nc{\bO}{\mathbb{O}}

\nc{\Kos}{\EuScript{K}}
\nc{\tilt}{\EuScript{T}}

\renewcommand{\la}{\lambda}

\newcommand{\al}{\alpha}

\newcommand{\Hom}{\operatorname{Hom}}

\renewcommand{\(}{\left(}

\newcommand{\cP}{\mathcal{P}}

\newcommand{\cO}{\mathcal{O}}

\nc{\hCST}{\hC_{\sS,\sT}}
\newcommand{\CST}{C_{\sS,\sT}}

\newcommand{\Ext}{\operatorname{Ext}}

\newcommand{\cS}{\mathcal{S}}
\newcommand{\hB}{\EuScript{B}}
\newcommand{\hD}{\EuScript{D}}
\newcommand{\hC}{\EuScript{C}}
\newcommand{\gu}{\mathfrak{g}_U}

\newcommand{\bS}{\mathbb{S}}

\newcommand{\cQ}{\mathcal{Q}}

\newcommand{\excise}[1]{}

\newcommand{\End}{\operatorname{End}}

\newcommand{\fg}{\mathfrak{g}}
\newcommand{\mmod}{\operatorname{-mod}}
\newcommand{\dgmod}{\operatorname{-dg-mod}}
\newcommand{\umod}{\operatorname{-\underline{mod}}}

\newcommand{\ck}{\kappa}

\newcommand{\alg}{T}

\newcommand{\bla}{{\underline{\boldsymbol{\la}}}}

\newcommand{\Lotimes}{\overset{L}{\otimes}}
\newcommand{\thetitle}{Rouquier's conjecture and diagrammatic algebra}

\begin{document}

\renewcommand{\theitheorem}{\Alph{itheorem}}

\usetikzlibrary{decorations.pathreplacing,backgrounds,decorations.markings,calc,
shapes.geometric}
\tikzset{wei/.style={draw=red,double=red!40!white,double distance=1.5pt,thin}}

\noindent {\Large \bf 
\thetitle}
\bigskip\\
{\bf Ben Webster}\footnote{Supported by the NSF under Grant DMS-1151473.}\\
Department of Mathematics, University of Virginia, Charlottesville, VA
\bigskip\\
{\small
\begin{quote}
\noindent {\em Abstract.}
We prove a conjecture of Rouquier relating the decomposition numbers
in category $\cO$ for a cyclotomic rational Cherednik algebra to
Uglov's canonical basis of a higher level Fock space.  Independent proofs of this
conjecture have also recently been given by Rouquier, Shan, Varagnolo and
Vasserot and by Losev, using different methods.

Our approach is to develop two diagrammatic models for this category
$\cO$; while inspired by geometry, these are purely diagrammatic
algebras, which we believe are of some intrinsic interest.  In
particular, we can quite explicitly describe the representations of
the Hecke algebra that are hit by projectives under the $\KZ$-functor
from the Cherednik category $\cO$ in this case, with an explicit basis.

This algebra has a number of beautiful structures including
categorifications of many aspects of Fock space.  It can be understood
quite explicitly using a homogeneous cellular basis which generalizes
such a basis given by Hu and Mathas for cyclotomic KLR algebras.
Thus, we can transfer results proven in this diagrammatic formalism to  category
$\cO$ for a cyclotomic rational Cherednik algebra, including the
connection of decomposition numbers to canonical bases mentioned
above, and an action of the affine braid group by derived equivalences
between different blocks.
\end{quote}
}
\bigskip

\section{Introduction}
\label{sec:introduction}
One of the most powerful tools in the theory of category $\cO$ for a
semi-simple Lie algebra is to consider it not just as a
lonely category but as a module over the monoidal category of
projective functors.  This perspective was essential for a number of
significant advances in our understanding of category $\cO$; one
example is
the theory of Soergel bimodules \cite{Soe90,Soe92}.
In category $\cO$ for cyclotomic Cherednik algebras, defined in \cite{GGOR}, the
r\^ole of projective functors is played by the induction functors of
Bezrukavnikov and Etingof \cite{BEind}.

In this paper, we exploit the fact that these functors essentially
control the entire structure of the category, just as is the case for
category $\cO$.  In category $\cO$, all projectives were obtained by
acting on a single projective with translation functors.  In the
Cherednik case, this method of control is a bit more indirect.  In
brief, category $\cO$ for a cyclotomic Cherednik algebra is the unique
collection of highest weight categories with a deformation which are tied together by induction functors, and a particular
partial order on simples.  Theorem \ref{thm:Cherednik-unique}, based
on ideas from \cite{RSVV}, makes this statement precise.  These
induction functors can also be repackaged into a highest weight
categorical action of $\slehat$ (a notion defined by Losev
\cite{LoHWCI}).  Similar uniqueness theorems with other applications
in representation theory have been proven by the author jointly with
Brundan and Losev \cite{LoWe,BLW}.

This fact is mainly of interest because we can give two constructions
of categories which also satisfy these properties, and thus are
equivalent to the Cherednik category $\cO$.  As in Rouquier
\cite{RouqSchur}, we can associate a choice of parameters for the
Cherednik algebra of $ \Z/\ell\Z\wr S_n$ to a charge $\Bs =(s_1,\dots,
s_\ell)\in \Z^\ell$; we let $\cO^{\Bs}$ denote the sum of category
$\cO$ for these parameters over all $n$.  (In fact, we can work with
arbitrary parameters.  See Section \ref{sec:comparison-theorem} for
details.)

We also associate a graded finite dimensional algebra with two presentations
to the same data, as introduced in \cite{WebBKnote} under the name
{\bf WF Hecke algebras}\footnote{Here, ``WF'' stands for ``weighted
  framed.'' This is explained in greater detail in \cite{WebBKnote}.}. We'll first introduce a ``Hecke-like'' presentation
which makes the connection to the $\KZ$ functor straightforward but
which is not homogeneous, and
then a ``KLR-like'' presentation, which has the considerable advantage
of being graded.  We'll let $T^\Bs$ denote this algebra with its
induced grading.  

We'll describe these presentations in considerable detail in Sections
\ref{sec:pict-hecke-cher}--\ref{sec:basic}.  The graded presentation
is a generalization of the Khovanov-Lauda-Rouquier algebras
\cite{KLI,Rou2KM}, and a special case of a construction described by
the author in \cite{WebwKLR}.  In the terminology of that paper, it's
a {\bf reduced steadied quotient} of a {\bf weighted KLR
  algebra}. Both these presentations are purely
combinatorial/diagrammatic in description, though the formalism from
which they are constructed is heavily influenced by geometry.
\begin{itheorem}\label{main}
  There is an equivalence of categories between the category of
  finite-dimensional (ungraded) representations of $T^{\Bs}$ and the
  category $\cO^{\Bs}$.  

  In particular, the category of graded modules over $T^{\Bs}$ is a
  graded lift of $\cO^{\Bs}$ compatible with the graded lifts of the
  Hecke algebra defined by Brundan and Kleshchev \cite{BKKL}.
\end{itheorem}

We should emphasize to the reader: this is the first explicit
description of the category $\cO$ for Cherednik algebras we know in
the literature.  As part of its proof, we give an explicit description
of the modules given by the image of the Knizhnik-Zamolodchikov
functor, a question which has been unresolved since the original
definition of this functor in \cite{GGOR}.

Furthermore, this development is also of theoretical interest.  The algebra $T^{\Bs}$ has a large number of desirable properties which
are not easily seen from the Cherednik perspective:

\begin{itheorem}\label{consequences}\hfill
  \begin{enumerate}
  \item The algebra $T^\Bs$ is graded cellular; its basis vectors are
    indexed by pairs of generalizations of standard Young tableaux of
    the same shape.
  \item This equivalence gives an explicit description, including a
    basis and graded lift, of the image of projectives from
    $\cO^{\Bs}$ under the $\KZ$ functor.
\item If the charges $\Bs$ and $\Bs'$ are
  permutations of each other modulo $e$, then the
  derived categories $D^b(T^\Bs\mmod)$ and $D^b(T^{\Bs'}\mmod)$ are equivalent,
  and in fact there is a strong categorical action of the affine braid
  group lifting that of the affine Weyl group on charges.
  \item The graded Grothendieck group $K^0_q(T^\Bs)$ is canonically
    isomorphic to Uglov's $q$-Fock space attached to the same charges.
\item Under this isomorphism, the standard modules correspond to pure
  wedges, the projectives to Uglov's canonical basis, and the simples
  to its dual.
  \end{enumerate}
\end{itheorem}
The first four points of this theorem have purely algebraic proofs.
The last point requires some geometric input from a category of
perverse sheaves considered in \cite{WebwKLR}; this also resolves a long-standing conjecture of Rouquier, that
the multiplicities of standard modules in projectives (which coincide
by BGG reciprocity with the multiplicities of simples in standards)
are given by the coefficients of a canonical basis specialized at
$q=1$.  Note that we have constructed a $q$-analogue of these
multiplicities using a grading on the algebras in question, rather
than using depth in the Jantzen filtration on standards as in
\cite{RamTing,ShanJantzen}.  Theorem \ref{consequences}(3) was proven
using geometric techniques in \cite[5.1]{GoLo}, but we eventually
intend to show that our functors match theirs in forthcoming work \cite{Webqui}.

Independent proofs of Theorem \ref{consequences}(5) have recently
appeared in work of Rouquier,
Shan, Varagnolo and Vasserot \cite{RSVV} and of Losev \cite{LoVV}, using very
different methods from those contained here; both proofs proceed by
proving the ``categorical dimension conjecture'' of Vasserot and
Varagnolo  \cite[8.8]{MR2673739}.  Of
course, it would be very interesting in the future to unify these
proofs.  

The ``categorical dimension conjecture'' actually leads to a stronger result, since instead
of relating $\cO^{\Bs}$ to a diagrammatic category, it relates
it to a truncation of parabolic category $\cO$ for an affine Lie
algebra, which is known to be Koszul by \cite[2.16]{SVV};  its Koszul dual
is again a Cherednik category $\cO$, with data specified by level-rank
duality, as conjectured of Chuang and Miyachi \cite{CM} and proven by
Shan, Varagnolo and Vasserot \cite[B.5]{SVV}. 

In our context, the consequence of these results is that:
\begin{itheorem}
 For each weight $\mu$, the algebra $T^{\Bs}_\mu$ is standard Koszul,
  and its Koszul dual is Morita equivalent to another such algebra
  $T^{\Bs^!}_{\mu^!}$, with parameters related by rank-level duality.
\end{itheorem}
We give an independent geometric proof that these algebras are
Koszul in \cite{Webqui}.
Since the grading and radical filtrations on the standards of a
standard Koszul algebra coincide, this shows on abstract grounds that
$q$-analogues of decomposition numbers using the grading coincide with
those using the Jantzen filtration.  This observation is also a key
piece of evidence for the ``symplectic duality'' conjectures on the
author, Braden, Licata and Proudfoot.  We will develop the
consequences of this observation further in later works \cite{BLPWgco,Webqui}.

\section*{Acknowledgements}

This paper owes a special debt to Ivan Losev; his suggestions and
generous sharing of unpublished manuscripts was vital to its
development.  I also benefited a great deal from discussions with
Chris Bowman-Scargill, Liron Speyer, Ben Elias, Peng Shan, Eric
Vasserot, Catharina Stroppel, Peter Tingley, Stephen Griffeth, Iain
Gordon, Nick Proudfoot, Tom Braden and Tony Licata, amongst many
others.

\clearpage

\section{WF Hecke algebras}
\label{sec:pict-hecke-cher}

\subsection{Hecke and Cherednik algebras}
\label{sec:hecke-cher-algebr}

Consider the
rational Cherednik algebra $\mathsf{H}$ of $\Z/\ell\Z\wr S_d$ (ranging over all
values of $d$)  over the base field $\C$ for the parameters $k=m/e$ where
$(m,e)=1$ and $h_j= s_jk-
j/\ell$.  That is, let $S_0$ be the set of complex reflections in
$\Z/\ell\Z\wr S_d$ that
switch two coordinate subspaces and $S_1$ the set which fix the
coordinate subspaces.  For each such reflection, let $\al_s$ be a
linear function vanishing on $\ker(s-1)$, and $\al_s^\vee$ a vector
spanning $\im(s-1)$ such that $\langle \al_s^\vee,\al_s\rangle =2$.  
Let \[\omega_s(y,x)=\frac{\langle y,\al_s\rangle \langle \al_s^\vee,x\rangle }{\langle
  \al_s^\vee,\al_s\rangle }=\frac{\langle y,\al_s\rangle \langle \al_s^\vee,x\rangle }{2 }\]
The RCA is the quotient of the algebra $T(\C^d\oplus
(\C^d)^*)\# (\Z/\ell\Z\wr S_d)$ by the relations for $y,y'\in \C^d,x,x'\in (\C^d)^*$:
\[[x,x']=[y,y']=0\]\[ [y,x]=\langle y,x\rangle
+\sum_{s\in S_0}2k\omega_s(y,x)  s+\sum_{s\in S_1} k\omega_s(y,x) 
\sum_{j=0}^{\ell-1}\det(s)^{-j}(s_j-s_{j-1}-1/\ell+\delta_{j,0})s.\]

\begin{definition}
  Category $\cO$, which we denote $\cO^{\Bs}_d$ (leaving $k$ implicit),
  is the full subcategory of modules over $\mathsf{H}$ which are
  generated by a finite dimensional subspace invariant under
  $\Sym(\C^d)\#\C[\Z/\ell\Z\wr S_d]$ on which $\Sym(\C^d)$ acts
  nilpotently.  Let $\cO^{\Bs}\cong \oplus_{d}\cO^{\Bs}_d$.
\end{definition}
This category is closely tied to the  cyclotomic Hecke algebra
$H_d(q,Q_\bullet)$ by a functor
$\KZ\colon\cO^{\Bs}_d\to H_d(q,Q_\bullet)\mmod$, which is fully
faithful on projectives.
\begin{definition}
  The cyclotomic Hecke algebra $H_d(q,Q_\bullet)$ is the algebra over
  $\C$ generated by $X_1^{\pm 1},\dots, X_d^{\pm 1}$ and
  $T_1,\dots, T_{d-1}$ with relations 
  \[ (T_i+1)(T_i-q)=0 \qquad T_iT_{i\pm 1}T_i=T_{i\pm 1}T_iT_{i\pm 1}
  \qquad T_iT_j=T_jT_i \,\, (i\neq j\pm 1)\]
  \[X_iX_j=X_jX_i\qquad T_iX_iT_i=qX_{i+1} \qquad X_iT_j=T_iX_j
  \,\,(i\neq j,j+1) \]
  \[(X_1-Q_1)(X_1-Q_2)\cdots (X_1-Q_\ell)=0\] where $q=\exp(2\pi i k)$ and
 $Q_i=\exp(2\pi i k s_i)$. 
\end{definition} 
One fact we'll use extensively is that these categories and functors
deform nicely when our parameters are valued not in $\C$ but a
local ring with residue field $\C$.  Let $\mathscr{R}=\C[[h,z_1,\dots,
z_\ell]]$.
We can consider the Cherednik algebra over $\mathscr{R}$ with
parameters $\sk=k+\frac{h}{2\pi i}$ and $\sss_j= (ks_j-\frac{z_j}{2\pi i})/\sk$.  Let $\bO^{\Bs}_d$ be the deformed category
$\cO$ of the
Cherednik algebra over $\mathscr{R}$ for the complex reflection group $\Z/\ell \Z\wr S_d$
with the parameters as above (the 1-parameter deformations inside this one are
discussed by Losev \cite[\S 3.1]{Lotowards}).  Let $\bO^{\Bs}\cong \bigoplus_d\bO^{\Bs}_d$.
This category is also equipped with a Knizhnik-Zamolodchikov functor,
landing in modules over the Hecke algebra $H_d(\pq,\PQ_\bullet)$ for
$\pq=qe^h$ and $\PQ_i=Q_ie^{-z_i}$.  We'll let $H_d(\pq)$ denote the
usual affine Hecke algebra of rank $d$ with parameter $\pq$.  Fix an
integer $D$.  

\begin{theorem}\label{thm:Cherednik-unique}
  Assume $\bQ^\Bs_d$ are categories for each $d\leq D$
  which satisfy:
  \begin{enumerate}
  \item $\bQ^\Bs_0\cong \mathscr{R}\operatorname{-mod}$
  \item $\bQ^\Bs_d$ is a highest weight
    category over $\mathscr{R}$ in the sense of \cite{RouqSchur}; in
    particular, $\bQ^\Bs_d$ is $\mathscr{R}$-flat.
  \item $\bQ^\Bs_d$ is endowed with adjoint $\mathscr{R}$-linear
    induction and restriction functors
    \[\operatorname{ind}\colon \bQ^\Bs_{d-1}\to \bQ^\Bs_{d}\qquad
    \operatorname{res}\colon  \bQ^\Bs_d\to \bQ^\Bs_{d-1}\]
 which preserve the categories of projective modules for all $d\leq D$.  Furthermore,
 the powers $\operatorname{ind}^c$
 have compatible actions of the affine Hecke algebra $H_c(\pq)$.
  \item the $d$-fold restriction functor
    $K=\operatorname{res}^d\colon \bQ^\Bs_d\to H_d(\pq)\mmod$ lands in
    the subcategory $H_d(\pq,\PQ_\bullet)\mmod$, and is a quotient
    functor to this subcategory that becomes an equivalence of categories after base change to $R=\C((h,z_1,\dots,
z_\ell))$ and is $-1$-faithful after base change to $\C$.
\item The category of $\mathscr{R}$-flat objects in $\bQ^\Bs_d$ are
  endowed with a duality which interwines a duality on modules
  over the Hecke algebra under induction functors and $K$.
\item The order induced on simple representations of
  $H_d(\pq,\PQ_\bullet)\otimes_{\mathscr{R}}R$ by the highest
  weight structure on $\bQ^\Bs_d$ has a common refinement with that induced by
  $\bO^{\Bs}_d$.
\item If $q=-1$, then the image of $\bQ^\Bs_2$ under $K$ contains
  the permutation module $H_2(T+1)$.
  \end{enumerate}
In this case, there is an equivalence $\bQ^\Bs_d\cong \bO^\Bs_d$ for
all $d\leq D$ which
matches $K$ with the usual Knizhnik-Zamolodchikov functor $\KZ$.
\end{theorem}
\begin{proof}
  This is heavily based on \cite[2.20]{RSVV}, which we'll apply in
  this case with $R=\mathscr{R},B=H_d(\pq,\PQ_\bullet), F=\KZ$, $F'=K$.
 There are 4 conditions required by this lemma, which we consider in
 the order given there.  
 \begin{itemize}
 \item The order induced by the two covers must have a common
   refinement:  This is one of our assumptions.
\item The functor $\KZ$ is fully faithful on standard or costandard
  filtered objects in $\bO^\bS_d$: This is proven in \cite[5.37]{RSVV}.
\item The functor $K$ is fully faithful on 
  $(\bQ^\Bs_d)^{\Delta}$ and  $(\bQ^\Bs_d)^{\nabla}$: Using the
  duality, these two statements are equivalent.
  Thus, we need only establish that $K$ is $0$-faithful (i.e. faithful
  on standard filtered objects).  We already
  assume that $K\otimes_{\mathscr{R}}\C((h,z_1,\dots, z_\ell))$ is an equivalence
  and thus $0$-faithful.  The result then
  follows from \cite[2.18]{RSVV}.
  
\item The image $\KZ(P)$ of any
  projective $P$ in $\bO^\bS$ whose simple quotient $L$ has
  $\Ext^i(L,T)\neq 0$ for some tilting $T$ and $i=0$ or $1$ also lies in the image of $K$:
  By \cite[6.3]{RSVV}, these images are precisely the modules of the
  form $H_d\otimes_{H_1} M$ for $M$ in the image of projectives under $\KZ$, and if $q=-1$, also the modules $H_d(T_1+1)$.  

  By compatibility with induction functors, we only need to show that
$\KZ(P)$ of any projective object in $\bO^{\Bs}_1$ and $H_2(T+1)$ (if $q=-1$) lie in this image.
  The latter is an assumption, so we need only address the former.
  In $H_1$, we have $\ell$ different simple representations over the
  generic point which correspond to the eigenvalues $\PQ_i$.  We
  denote the corresponding standard module $\Delta_i$.  Let $H_1^u$ be the stable kernel of $X_1-u$, and consider $P_1^u=\operatorname{ind}(\mathscr{R},
    H_1^u)$.  Let $m_u$ be the number of indices $i$ such that $Q_i=u$.  The object $P_1^u$ is indecomposable (since $H_1^u$ is
    and $K$ is fully faithful on projectives), and thus has a unique
    standard quotient $\Delta_i$ for some $i$, the largest standard
    such that $Q_i=u$.  The kernel of this map is a module we'll call
    $P_2^u$; this has a standard filtration, and thus a map to some
    other standard $\Delta_j$.  If we let $P_j$ be the projective
    cover of $\Delta_j$, we have an induced map $P_j\to P_2^u$.  The induced map $K(P_j)\to
    K(P_2^u)=(X-\PQ_i)H_1^u$ must be surjective, since it induces a
    surjective map $K(P_j)\otimes \C\to K(\Delta_i)\otimes \C$.

    Consider $K(P_j)\otimes \C$.  This must be the kernel of
    $(X_1-u)^m$ for some $1\leq m\leq m_u$.  In fact, $m>m_u$ because otherwise, we would have $K(P_j)\otimes
    \C\cong K(P_i)\otimes \C$ (impossible since $K\otimes \C$
    is fully faithful on projectives).  Thus, $K(P_j)\otimes \C$ has
    dimension $\leq m_u-1$.  Since the dimension of $K( P_2^u)\otimes
    \C$ is $m_u-1$, we must have $K( P_2^u)\otimes
    \C\cong K( P_j^u)\otimes
    \C $, so by full faithfulness, $P_j\cong  P_2^u$.  

Applying this argument inductively, we find that $P_1^u$ has a
filtration by the different indecomposable projectives on which
$X_1-u$ is topologically nilpotent, with successive quotients being
the different standards.  In particular, the images of these
projectives are \[\mathscr{R}[X_1]/\prod_{\substack{Q_i=u\\ \Delta_i\leq
  \Delta_j}}(X_1-\PQ_i)\] which are the same as the images for $\KZ$.
 \end{itemize}
This establishes all the conditions and finishes the proof.
\end{proof}

Note that this theorem can be easily applied to show that whenever the
order induced on multipartitions of numbers $\leq D$ induced by 
the Cherednik algebra (which uses the $c$-function) coincides with
dominance order, then we can apply this theorem with $\bQ^\Bs_d$ the
category of modules over the rank $d$ cyclotomic $q$-Schur algebra for
the parameters $(\pq,\PQ_\bullet)$.  This will occur whenever $s_1\ll
s_2\ll\cdots \ll s_\ell$, though there is no choice of parameters
where this will work for all $D$ if $k\in \Q$.  

\subsection{Combinatorial preliminaries }
\label{sec:comb-prel-}

The combinatorics that underlie category $\cO$ for a Cherednik algebra are those of higher-level
Fock spaces and multipartitions.

We must introduce a small
generalization of the combinatorics that appear in twisted Fock spaces
(in the sense of Uglov \cite{Uglov}).  As we'll see later, this is
just rearranging deck chairs, but it quite convenient for us.
Fix scalars $(r_1,\dots, r_\ell)\in (\C/\Z)^\ell$, and $k\in \C$ with $\ck=\Re(k)$.  Consider the subset of $\C/\Z$ defined by 
\[U=\{r_i+km\pmod \Z\mid i=1,\dots, \ell\text{ and }m\in \Z\}\]
This set is finite if and only if $k\in \Q$, and connected if and only
if all $r_i$ lie in the same coset of the subgroup of the additive
group $\Z k$ in $\C/\Z$.  We endow this set with an oriented
graph structure by connecting $u \to u+k$ for every $u\in U$.  We
let $\gu$ be the Lie algebra whose Dynkin diagram is given by $U$ if
$k\notin\Z$.  If $k\in \Z$, then we let $\gu$ be the product over
$U$ of copies of 
$\mathfrak{\widehat{gl}}_1$, the Heisenberg algebra on infinitely
many variables, with the grading element $\partial$ adjoined. 
This is a product of either finitely many copies of $\slehat$ if
$k=a/e$ with $(a,e)=1$, or of $\mathfrak{sl}_\infty$ if $k$ is
irrational.  Throughout, we'll fix $e$ to be the denominator of $k$ if
$k\in \Q$, or $e=0$ if $k\notin \Q$.
\begin{remark}
  For purposes of the internal theory of WF Hecke algebras, we'll only
  care about the exponentials $\exp(2\pi ir_j)=Q_j$ and $\exp(2\pi i k)=q$.  Thus,
  we could just as easily define $U$ to be the subset of $\C^\times$
  of the form $\{Q_iq^m\}$ for $m\in \Z$.  This definition easily
  translates to other fields, and applies equally well there.
  However, it's only over $\C$ that we can make sense of the
  connection to Cherednik algebras, so we will focus on this case.
\end{remark}

\subsubsection{Weightings}
\label{sec:weightings}

\begin{definition}
  An {\bf $\ell$-multipartition} of $n$ is an $\ell$-tuple of
  partitions with $n$ total boxes.  The individual partitions of this
  $\ell$-tuple are called its components.  Recall that the {\bf
    diagram} of a multipartition $(\xi^{(1)},\dots, \xi^{(k)})$ is the
  set of 3-tuples $(a,b,m)$ of natural numbers satisfying
  $k\geq m\geq 1, b\geq 1, \xi_b^{(m)}\geq a\geq 1$.  We call each of
  these 3 tuples a {\bf box}.

  A {\bf charge} on a multipartition is a choice of integers
  $s_1,\dots, s_\ell$; the charged content of a box $(a,b,m)$ is
  $s_m+b-a$.  
\end{definition}

For example, the multipartition $((2,2),(3,1))$ with charge
$s_1=3,s_2=-4$, when drawn in French notation (i.e. using $a$ and $b$
as the usual $x$ and $y$ coordinate) will appear as:
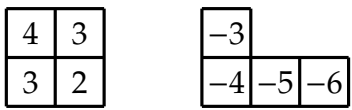
\[\tikz[very thick,scale=1.3]{\draw (0,0) -- (0,1);\draw (.5,0) -- (.5,1);\draw
  (1,0) -- (1,1);\draw (0,0) -- (1,0);\draw (0,.5) -- (1,.5);\draw
  (0,1) -- (1,1);
\draw (2,0) -- (2,1);\draw (2.5,0) -- (2.5,1);\draw
  (3,0) -- (3,.5);\draw (3.5,0) -- (3.5,.5);\draw (2,0) --
  (3.5,0);\draw (2,.5) -- (3.5,.5);
\draw (2,1) -- (2.5,1);
\node at (.25,.25) {$3$};
\node at (.75,.25) {$2$};
\node at (.25,.75) {$4$};
\node at (.75,.75) {$3$};
\node at (2.25,.25) {$-4$};
\node at (2.75,.25) {$-5$};
\node at (2.25,.75) {$-3$};
\node at (3.25,.25) {$-6$};
}\]
where each box is filled with its charged content.

Usually in the theory of twisted Fock spaces, one has a basis indexed
by $\ell$-multipartitions, and the structure of this space (especially its
$\gu$-module structure) depends on choice of charge.

These charges contribute to the structure of the Fock space and its
$\slehat$ action in two different ways: the
order induced on boxes by the charged content, and value of the
charged content $\pmod e$.  We wish to separate these functions of the
charge, and generalize to the case where $\slehat$ is replaced by $\gu$.
\begin{definition} A {\bf weighting} of an $\ell$-multipartition is an
  ordered $\ell$-tuple $(r_1,\dots, r_\ell)\in (\C/\Z)^\ell$ 
  and an ordered
  $\ell$-tuple $(\vartheta_1,\dots,\vartheta_\ell)\in \R^\ell$ with $\vartheta_i\neq \vartheta_j$ (with no assumption of congruence
  between the two).
\end{definition} 
The quantities $r_*$ carry the information which corresponds to the residue class
$\pmod e$, and quantities $\vartheta_*$ carry the information of the induced order on boxes. 

Given an arbitrary weighting, we
associate a {\bf residue} in $U$ to each box of the diagram of a
multi-partition:
the box
$(a,b,m)$ receives $r_m+k(b-a)$; note that all elements of $U$ occur
for a box of some multipartition.  We
will often match these residues with their corresponding simple roots
of $\gu$. We let $\res(\xi/\eta)$ for a skew multi-partition
$\xi/\eta$ denote the sum of the roots corresponding to each box in
its diagram.  

In essence, if the residues $r_i$ and $r_j$ do not differ by an
integer multiple of $k$, the corresponding partitions will not
interact; this is analogous to a result of Dipper and Mathas
\cite[1.1]{DMmorita} for Ariki-Koike algebras.  Thus, let us
concentrate on the case where the graph $U$ is connected.  

\begin{definition}\label{def:Uglov}
  The {\bf Uglov weighting} $\vartheta_{\Bs}^\pm$ attached to an $\ell$-tuple $(s_1,\dots,
  s_\ell)$ of integers
  (its {\bf charge}), is that where $k=\ck=\pm \nicefrac{1}{e}$ if
  $e>0$ and $k$ is an arbitrary positive irrational real number if $e=0$.  
  \begin{itemize}
  \item the residue $r_m$ is given by the reduction of $ks_m\pmod \Z$.
  \item the weights of the partitions are given by $\vartheta_j=\ck
    s_j-je\ck /\ell$.
\end{itemize}
\end{definition}
The choice of $k=\pm \nicefrac{1}{e}$ is less significant than it
might first appear; nothing about the combinatorics we consider later
will change if $k=\pm \nicefrac{a}{e}$ for any positive integer $a$
coprime to $e$. 
In general, our combinatorics will reduce to familiar notions for
those who work with charged multipartitions and twisted higher level
Fock spaces in the Uglov case.  In particular, the induced order on
boxes is the same as that coming from charged content (using the
component as a tie-breaker).   

There is a symmetry of this definition: sending $k\to -k$
and $s\mapsto s^\star=(-s_\ell,\dots, -s_1)$ results in the same
weighting up to shift, if we reindex $i\mapsto \ell-i+1$, and send
$r_i\mapsto -r_{\ell-i+1}$.  

Actually, for any weighting with $U$ connected, there is
an Uglov weighting which can replace it. Thus,
we lose no generality by only considering Uglov weightings.
\begin{definition}\label{def:Uglovation}
  For an arbitrary weighting with $U$ connected, we define its {\bf
    Uglovation} to be the Uglov weighting associated to $s_1,\dots,
  s_\ell$ constructed as
  follows:  
  \begin{itemize}
  \item By assumption, since $U$ is connected, $r_j-r_1$ is an integer multiple of $k$.  If
    $e=0$, we let
    $h_j$ be the unique integer such that $r_j-r_1=kh_j$, and if
    $e>0$, let $h_j$ be the smallest such non-negative integer.  In
    particular $h_1=0$.  
\item Reindexing values except for the first, we can assume
  $\vartheta_j/\ck-h_j$ are cyclicly ordered $\!\!\pmod e$.
\item We let $s_1=0$ by convention.  We let $s_j$ be the unique
  integer such that $s_j\equiv h_j\pmod e$ and $0\leq
  \vartheta_j/\ck-\vartheta_1/\ck-s_j\leq e$.   That is,
  \[s_j=h_j+e\big\lfloor ( \vartheta_j/e\ck-\vartheta_1/e\ck)\big\rfloor.\]
  \end{itemize}
\end{definition}
We will show that the algebra $T^\vartheta$ which we'll attach to a
weighting is Morita equivalent to that for an Uglov weighting in
Corollary \ref{cor:Uglovation-Morita}.

\subsubsection{Dominance order}
\label{sec:dominance-order}

We will imagine our multipartition diagram drawn in ``Russian
notation'' with rows tilted northeast, and columns northwest if $\ck>0$
(and {\it vice versa} if $\ck<0$), with the
bottom corner placed at $\vartheta_m$, and the boxes having diagonal of length
$2\ck$; see Figure \ref{fig:partition}. For a box at $(i,j,m)$ in the diagram of
$\nu$, its {\bf $x$-coordinate} is $\vartheta_m+\ck(j-i)$, that is, the $x$-coordinate of the
center of the box when partitions are drawn as we have specified.
This coincides with the $\Bs$-shifted content as in \cite{GoLo} if we
choose $\ck=1$ and $s_i=\vartheta_m$.

\begin{definition}\label{def:dominance-order}
  The {\bf weighted
    dominance order} on multipartitions for a fixed weighting is the
  partial order where $\nu \geq \nu'$ if for each real
  number $a$, the number of boxes in $\nu$ with a fixed residue with $x$-coordinate less
  than $a$ is greater than or equal to the same number in $\nu'$.  
\end{definition}

On
  single partitions, this is a coarsening of the usual dominance order, but
  for multipartitions, it depends in a subtle way on the weighting.
  What is usually called dominance order on multipartitions arises   as a refinement on
  multipartitions of $n$
when $\ck>0$ and $ \vartheta_i-\vartheta_{i-1}>n\ck$ for all $i$.  

In order to clarify the relationship between our combinatorics and
that for rational Cherednik algebras, it will be useful to refine
this order using a numerical function;  we let the {\bf weighted
  $c$-function} be the function that assigns to a multipartition the
sum of minus the
$x$-coordinates of its boxes.  The obvious order by
$c$-function is thus a refinement of weighted dominance order.  In the
theory of rational Cherednik algebras, there is also a $c$-function,
which we denote $c_{\operatorname{GL}}$, since we use the conventions
of \cite[\S 2.3.5]{GoLo} except that their $\la_i$
  is our $\la_{i-1}$.

\begin{proposition}\label{prop:c-functions}
  If we let $\vartheta_i=s_i\ck- i/\ell$ then our
  $c$-function is related to the usual $c$-function by 
\[\ell c=
c_{\operatorname{GL}}+\frac{\ell n}{2}-n+\kappa n(s_1+\cdots
  +s_\ell).\]  
Since these functions differ by an orientation preserving affine
transformation, they induce the same $c$-function order.
\end{proposition}
When $\kappa=1/e$
  and the numbers $s_i$ are integers, this recovers our usual Uglov
  weighting for the charge $\Bs$.  Thus, in this case, the usual $c$-function order is a refinement
  of $\vartheta_{\Bs}^+$-weighted dominance order.  
\begin{proof}
  This follows instantly from the formula \cite[(2.3.8)]{GoLo}
  (accounting for the difference in convention).  We need only to show that our $c$-function agrees with 
  \[-\sum_{r=1}^{\ell} (r-1)|\xi^{(r)}| +\sum_{A\in \xi}\kappa
  \ell\res^{\Bs}(A) = \sum_{(i,j,r)\in \la} -r+1+ (s_r+j-i)\cdot
  -\ck=\sum_{(i,j,r)\in \la}  -(\ck s_r+r-1)-\ck(j-i)\]
This is, indeed, the sum of minus the $x$-coordinates of the boxes
under the Uglov weighting.  
\end{proof}

\subsubsection{Loadings for multipartitions and $\mathbf{i}$-tableaux}
\label{sec:load-mult}

Fix a multipartition $\xi$, and give its diagram a very subtle
tilt to the right.  
We create a subset by projecting the top corner of each box to
the real number line, and weighting that point with the residue of the
box.  More precisely: 
\begin{definition}
  We let 
\[D_\xi:=\{\vartheta_k+(i+j)\epsilon+\ck (j-i)\mid (i,j,k)\text{ a box in the diagram of }\nu\}\]
\end{definition}
Obviously, this set depends on $\epsilon$, but for $0<\epsilon$
sufficiently small, its equivalence class will not change.  This
equivalence class will be independent of $\epsilon$
as long as $0<\epsilon
<|\vartheta_i-\vartheta_j+q\ck|/|\xi|$ for integers $q$ with $|q|\leq
|\xi|$, so we exclude $\epsilon$ from the notation.

We can upgrade this set to a loading--that is, to a map $D_\xi\to U$.
In \cite{WebwKLR}, we would think of this as a map from $\R\to
U\cap\{0\}$ that extends the map on $D_\xi$ by 0 on all other points.
  The loading $\Bi_\xi$ sends  $\vartheta_k+(i+j)\epsilon+\ck (j-i)$ to the
  simple root $\al_{m}$ if there is a box $(i,j,k)$ in the diagram of
  $\nu$ with residue $m=r_k+k(j-i)$, and 0 otherwise.

\nc{\lo}{.9951}
\nc{\sm}{.0696} \nc{\ong}{.0696} \nc{\ma}{.0048}
\begin{figure}
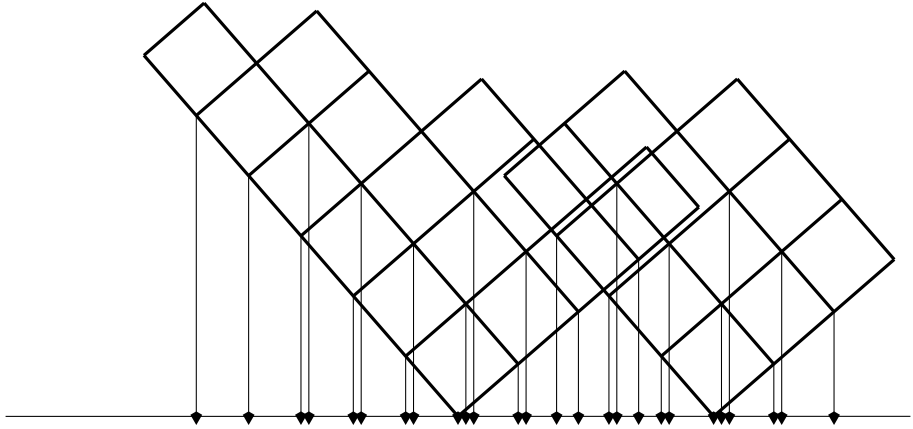
\centering
  
  \tikz[very thick,scale=3,rotate=-4]{ \coordinate (a) at (0,0);
    \coordinate (b) at (1.13,1.13*.07); \draw (a) -- +(-1.5,1.5);
    \draw (a) ++(.5,.5) -- +(-1.25,1.25); \draw (a) ++(.25,.25) --
    +(-1.5,1.5); \draw (a) ++(.75,.75) -- +(-.75,.75); \draw (a)
    ++(1,1) -- +(-.25,.25); \draw (a) -- +(1,1); \draw (a) ++
    (-.25,.25)-- +(1,1); \draw (a) ++ (-.5,.5)-- +(.75,.75); \draw (a)
    ++ (-.75,.75)-- +(.75,.75); \draw (a) ++ (-1,1)-- +(.5,.5); \draw
    (a) ++ (-1.25,1.25)-- +(.5,.5); \draw (a) ++ (-1.5,1.5)--
    +(.25,.25); \draw[thin] (a) + (-2,-.14) -- +(2,.14);\node at (a)
    [kite, fill, inner sep=1.2pt]{}; \draw[thin] (a) +(.25,.25) --
    +(.25*\lo+.25*\sm,.25*\ong+.25*\ma) node[at end,kite, fill,
    inner sep=1.2pt]{}; \draw[thin] (a) +(.25,.75) --
    +(.25*\lo+.75*\sm,.25*\ong+.75*\ma) node[at end,kite, fill,
    inner sep=1.2pt]{}; \draw[thin] (a) +(.5,.5) --
    +(.5*\lo+.5*\sm,.5*\ong+.5*\ma) node[at end,kite, fill, inner
    sep=1.2pt]{}; \draw[thin] (a) +(0,.5) -- +(.5*\sm,.5*\ma) node[at
    end,kite, fill, inner sep=1.2pt]{}; \draw[thin] (a) +(0,1) --
    +(\sm,\ma) node[at end,kite, fill, inner sep=1.2pt]{};
    \draw[thin] (a) +(.75,.75) --+(.75*\lo+.75*\sm,.75*\ong+.75*\ma)
    node[at end,kite, fill, inner sep=1.2pt]{}; \draw[thin] (a)
    +(-.25,.25) --+(-.25*\lo+.25*\sm,-.25*\ong+.25*\ma) node[at
    end,kite, fill, inner sep=1.2pt]{}; \draw[thin] (a) +(-.25,.75)
    --+(-.25*\lo+.75*\sm,-.25*\ong+.75*\ma) node[at end,kite, fill,
    inner sep=1.2pt]{}; \draw[thin] (a) +(-.5,.5)
    --+(-.5*\lo+.5*\sm,-.5*\ong+.5*\ma) node[at end,kite, fill,
    inner sep=1.2pt]{}; \draw[thin] (a) +(-.5,1)
    --+(-.5*\lo+\sm,-.5*\ong+\ma) node[at end,kite, fill, inner
    sep=1.2pt]{}; \draw[thin] (a) +(-.75,1.25)
    --+(-.75*\lo+1.25*\sm,-.75*\ong+1.25*\ma) node[at end,kite,
    fill, inner sep=1.2pt]{}; \draw[thin] (a) +(-.75,.75)
    --+(-.75*\lo+.75*\sm,-.75*\ong+.75*\ma) node[at end,kite, fill,
    inner sep=1.2pt]{}; \draw[thin] (a) +(-1,1)
    --+(-1*\lo+\sm,-1*\ong+\ma) node[at end,kite, fill, inner
    sep=1.2pt]{}; \draw[thin] (a) +(-1.25,1.25)
    --+(-1.25*\lo+1.25*\sm,-1.25*\ong+1.25*\ma) node[at end,kite,
    fill, inner sep=1.2pt]{}; \draw (b) -- +(-1,1); \draw (b) ++(.5,.5)
    -- +(-1,1); \draw (b) ++(.25,.25) -- +(-1,1); \draw (b)
    ++(.75,.75) -- +(-.75,.75); \draw (b) -- +(.75,.75); \draw (b) ++
    (-.25,.25)-- +(.75,.75); \draw (b) ++ (-.5,.5)-- +(.75,.75); \draw
    (b) ++ (-.75,.75)-- +(.75,.75); \draw (b) ++ (-1,1)-- +(.5,.5);
    \node at (b) [kite, fill, inner sep=1.2pt]{}; \draw[thin] (b)
    +(.25,.25) -- +(.25*\lo+.25*\sm,.25*\ong+.25*\ma) node[at
    end,kite, fill, inner sep=1.2pt]{}; \draw[thin] (b) +(.25,.75) --
    +(.25*\lo+.75*\sm,.25*\ong+.75*\ma) node[at end,kite, fill,
    inner sep=1.2pt]{}; \draw[thin] (b) +(.5,.5) --
    +(.5*\lo+.5*\sm,.5*\ong+.5*\ma) node[at end,kite, fill, inner
    sep=1.2pt]{}; \draw[thin] (b) +(0,.5) -- +(.5*\sm,.5*\ma) node[at
    end,kite, fill, inner sep=1.2pt]{}; \draw[thin] (b) +(0,1) --
    +(\sm,\ma) node[at end,kite, fill, inner sep=1.2pt]{};
    \draw[thin] (b) +(-.25,.25)
    --+(-.25*\lo+.25*\sm,-.25*\ong+.25*\ma) node[at end,kite, fill,
    inner sep=1.2pt]{}; \draw[thin] (b) +(-.25,.75)
    --+(-.25*\lo+.75*\sm,-.25*\ong+.75*\ma) node[at end,kite, fill,
    inner sep=1.2pt]{}; \draw[thin] (b) +(-.5,.5)
    --+(-.5*\lo+.5*\sm,-.5*\ong+.5*\ma) node[at end,kite, fill,
    inner sep=1.2pt]{};\draw[thin] (b) +(-.5,1)
    --+(-.5*\lo+\sm,-.5*\ong+\ma) node[at end,kite, fill, inner
    sep=1.2pt]{}; \draw[thin] (b) +(-.75,.75)
    --+(-.75*\lo+.75*\sm,-.75*\ong+.75*\ma) node[at end,kite, fill,
    inner sep=1.2pt]{}; }
  \caption{The set $D_\xi$ attached to the multipartition $\xi=(6,5,3,1); (4,4,3)$}
\label{fig:partition}
\end{figure}
\begin{definition}
  Given a subset $D\subset \R$, let a {\bf $D$-tableau} be a filling of
  the diagram of a multi-partition with the elements of $D$ such that
  \begin{itemize}
  \item each $d\in D$ occurs exactly once, and
  \item the entry in $(1,1,m)$ is greater than than $\vartheta_m$,
  \item the entry in $(i,j,m)$ is greater than that in $(i-1,j,m)$
    minus $\ck$ and greater than that in $(i,j-1,m)$ plus $\ck$.
  \end{itemize}
  If the differences between each pair of real numbers which occurs is
  greater than $\ck$, this is just the notion of a standard tableau on a
  charged multipartition.  Also, note that transposing each partition
  gives a tableau when $\ck$ is replaced by $-\ck$.

 If we upgrade the set $D$ to a loading $\Bi\colon D\to U$, an {\bf
   $\Bi$-tableau} is a $D$-tableau such that 
$\Bi$ is the function that sends each element of $D$ to the
    residue of the box it occurs in.
\end{definition}
Note that this condition is the same as saying that if we add $\ck(i-j)$
to the entry in box $(i,j,k)$, we obtain a standard tableau on each
component of the multipartition.   This is a less useful observation
than you might think, since in general, we will think about the set of
$D$-tableaux for $D$ fixed, and the addition described above will
result in different numbers used in the filling for different
tableaux.
\begin{example}
  For example, if $\ell=1$ and $D=\{a,b\}$ for two real numbers $a,b$,
  then the set of $D$-tableaux will look as expected if $|a-b|>|\ck|$:
  assuming $\vartheta_1<a<b$, we'll have
  tableaux
  
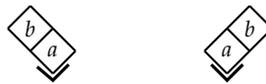
\[\tikz[baseline=5pt,scale=.3,thick]{\draw (0,0) --(-2,2); \draw
    (0,0) --(1,1); \draw (-1,3) --(-2,2); \draw (-1,3) --(1,1); \draw
    (-1,1) -- (0,2); \draw[very thick] (-.7,.3) -- (0,-.3) -- (.7,.4);
    \node[scale=.7] at (0,1) {$a$}; \node[scale=.7] at (-1,2)
    {$b$}; } \qquad \qquad \tikz[baseline=5pt,scale=.3,thick]{\draw
    (0,0) --(2,2); \draw (0,0) --(-1,1); \draw (1,3) --(2,2); \draw
    (1,3) --(-1,1); \draw (1,1) -- (0,2); \draw[very thick] (-.7,.4)
    -- (0,-.3) -- (.7,.4); \node[scale=.7] at (0,1)
    {$a$}; \node[scale=.7] at (1,2) {$b$};}\]
  On the other hand, if $ |a-b|<|\ck|$, and $a,b >\vartheta_1$, then
  then we have 2 $D$-tableaux of one shape:
  
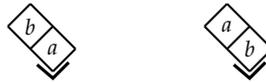
\[\tikz[baseline=5pt,scale=.3,thick]{\draw (0,0) --(-2,2); \draw
    (0,0) --(1,1); \draw (-1,3) --(-2,2); \draw (-1,3) --(1,1); \draw
    (-1,1) -- (0,2); \draw[very thick] (-.7,.3) -- (0,-.3) -- (.7,.4);
    \node[scale=.7] at (0,1) {$a$}; \node[scale=.7] at (-1,2)
    {$b$}; } \qquad \qquad \tikz[baseline=5pt,scale=.3,thick]{\draw
    (0,0) --(-2,2); \draw (0,0) --(1,1); \draw (-1,3) --(-2,2); \draw
    (-1,3) --(1,1); \draw (-1,1) -- (0,2); \draw[very thick] (-.7,.3)
    -- (0,-.3) -- (.7,.4); \node[scale=.7] at (0,1)
    {$b$}; \node[scale=.7] at (-1,2)
    {$a$}; } \]
  and none of the other\footnote{Note that switching the sign of $\ck$
    both reverses the conditions on rows and columns for a
    $D$-tableau, and our convention for drawing partitions in Russian
    notation.  Thus, this observation is true for either sign.}.
\end{example}

The {\bf Russian reading word} of an $\Bi$-tableau of shape $\eta$ is
the word obtained by reading the boxes of the tableau in order of the
$x$-coordinate, reading up columns, that is, in the order of the
loading $\Bi_\eta$, reading left to right.

For a usual standard tableau of shape $\eta$, the boxes where entries are
below a fixed value form a new partition diagram.  However, for a
$\Bi$-tableau, this is not the case; that said, one can make sense of a
particular box being addable or removable relative to a value $h$.
\begin{definition}
  For a fixed box $(i,j,m)$ whose entry is not $h$, we have a
  subdiagram of $\eta$ given by
  the boxes $(i',j',m)$ with entries $> h+(j'-i'-j+i)\kappa$.  We say
  that $b=(i,j,m)$ is {\bf addable} (resp. {\bf removable}) {\bf relative to
  $h$} if
\begin{itemize}
\item it is addable (resp. removable) for this subdiagram, and 
\item if $b=(1,1,m)$  then we have $\vartheta_m<h$.
\end{itemize}
\end{definition}
That is, a box is addable (resp. removable) relative to $h$ if it
existing entry (if it is has one) is $>h$ (resp. $<h$) and making its
entry $h$ would not disturb the tableau conditions.  
 
Note that the subdiagram we consider depends on $h$ and $i-j$, and
that it is only relevant whether the adjacent squares $(i,j\pm 1,m)$
and $(i\pm 1,j,m)$ are in this subdiagram.  Note that:
\begin{itemize}
\item The box $(1,1,m)$ is addable if  $\vartheta_m<h$  and
  furthermore if it is in the diagram, then the 
  entry is $>h$.
\item The box $(i,j,m)$ in the diagram of $\eta$ is  addable relative to $h$ if $h$ is
  less than the entry in $(i,j,m)$, $h+\ck$ is greater than the entry in
  $(i-1,j,m)$ and $h-\ck$ is greater than the entry in $(i,j-1,m)$.
 \item  If $(i,j,m)$ is not in the  diagram of $\eta$, it is
    addable relative to $h$ if it is addable for the whole diagram
and $h+\ck$ is greater than the entry in
  $(i-1,j,m)$ and $h-\ck$ is greater than the entry in $(i,j-1,m)$.
\item A box $(i,j,m)$ in the diagram of $\eta$ is removable relative
  to $h$ if $h$ is greater than the entry in $(i,j,m)$, $h+\ck$ is
  less than the entry in $(i,j+1,m)$ and $h-\ck$ is greater than the
  entry in $(i+1,j,m)$.
\end{itemize}

\begin{figure}
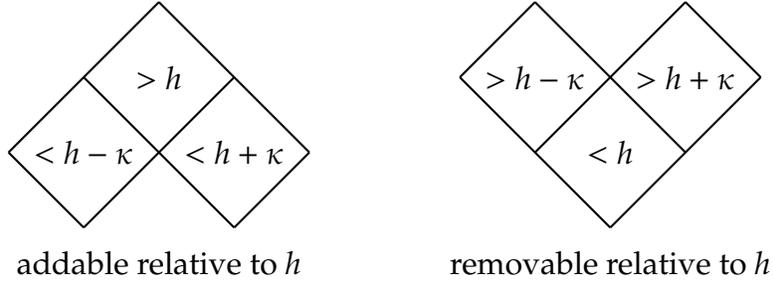

  \centering
  \tikz { \node[label=below:{addable relative to $h$}] at (-3,0){
      \tikz[thick,scale=2]{\draw (-1,0) -- (0,1); \draw (1,0) --
        (0,1); \draw (-1,0) -- (-.5,-.5); \draw (1,0) -- (.5,-.5);
        \draw (.5,.5) -- (-.5,-.5); \draw (-.5,.5) -- (.5,-.5); \node
        at (0,.5){$>h$}; \node at (.5,0){$<h+\ck$}; \node at
        (-.5,0){$<h-\ck$}; } }; \node[label=below:{removable relative to
      $h$}] at (3,0){ \tikz[thick,scale=-2]{\draw (-1,0) -- (0,1);
        \draw (1,0) -- (0,1); \draw (-1,0) -- (-.5,-.5); \draw (1,0)
        -- (.5,-.5); \draw (.5,.5) -- (-.5,-.5); \draw (-.5,.5) --
        (.5,-.5); \node at (0,.5){$<h$}; \node at (.5,0){$>h-\ck$};
        \node at (-.5,0){$>h+\ck$}; } }; }
  
  \caption{Relatively addable and removable boxes}
\label{fig:add-remove}
\end{figure}

We say that a box $(i',j',m')$ is right of $(i,j,m)$ if the associated
$x$ coordinate in $D_\xi$ is greater, that is if \[\vartheta_{m'}+(i'+j')\epsilon+\ck (j'-i')>\vartheta_m+(i+j)\epsilon+\ck (j-i).\]
\begin{definition}
  The {\bf degree} of a box $b$ in an $\Bi$-tableau with entry $h$ is the
  number of boxes of the same residue as and to the right of $b$ which are
  addable relative to the entry $h$ minus the number removable
  relative to $h$.

  The {\bf degree} of an $\Bi$-tableau is the sum of the degrees of the boxes.
\end{definition}
Again, we wish to emphasize that this does not count elements which
are addable or removable with respect to a fixed diagram; instead for
each box $(i',j',m')$ right of our fixed one, we compute a separate
subdiagram with depends on $i'-j'$ and on $h$, and check whether it is
addable or removable in this diagram.

\subsection{WF Hecke algebras defined}
\label{sec:pict-cher-algebr}

We will apply this combinatorics to define a diagrammatic version
of the category $\cO_m$.  As in \cite{WebBKnote}, let $S$ be a local
complete $\K$-algebra and let $\pq,\PQ_1,\dots,\PQ_\ell\in S$ be units
with $q,Q_1,\dots, Q_\ell$ their images in $\K$.  
\begin{definition}
  We let a {\bf type WF Hecke diagram} be a collection of curves in
  $\R\times [0,1]$ with each curve mapping diffeomorphically to
  $[0,1]$ via the projection to the $y$-axis.  Each curve is allowed
  to carry any number of squares or the formal inverse of a square.  We draw:
  \begin{itemize}
  \item a dashed line $\ck$ units to the right of each strand,
    which we call a {\bf ghost},
\item red
lines at $x=\vartheta_i$ each of which carries a label
$\PQ_{i}\in S$. 
  \end{itemize} 
We now require that there are no
  triple points or tangencies involving any combination of strands,
  ghosts or red lines and no squares lie on crossings.  We consider these diagrams equivalent if they
  are related by an isotopy that avoids these tangencies, double
  points and squares on crossings.
\end{definition}
In examples, we'll usually draw these with the number $Q_i$ written at
the bottom of the strand, leaving the lift $\PQ_i$ implicit.

Note that at any fixed value of $y$, the positions of the various
strands in this horizontal slice give a finite subset $D$ of $\R$.  If
this slice is chosen generically, in particular avoiding any crossings, we'll have that  we have
$\vartheta_i-d\neq m\kappa$ and $d'-d\neq m\kappa$ for 
for all $d,d'\in D,m\in \Z$.  We'll call such a subset {\bf generic}.

We can now define the object of primary interest in this section.  
\begin{definition}
  The {\bf type WF Hecke algebra}  $\PC^\vartheta$ is the $S$-algebra generated by
  WF Hecke diagrams modulo the local relations
\newseq
\begin{equation*}\subeqn\label{nilHecke-2}
    \begin{tikzpicture}[scale=.7,baseline,green!50!black]
      \draw[very thick](-3,0) +(-1,-1) -- +(1,1); \draw[very thick](-3,0) +(1,-1) --
      node[pos=.8,fill=green!50!black,inner sep=3pt]{} +(-1,1) ;
      \node[black] at (-1.5,0){$-$}; \draw[very thick](0,0) +(-1,-1) -- +(1,1); \draw[very thick](0,0) +(1,-1) --  node[pos=.2,fill=green!50!black,inner sep=3pt]{}
      +(-1,1); 
    \end{tikzpicture}\hspace{4mm}=\hspace{4mm}
    \begin{tikzpicture}[scale=.7,baseline,green!50!black]
      \draw[very thick](-3,0) +(-1,-1) --  node[pos=.2,fill=green!50!black,inner sep=3pt]{}+(1,1); \draw[very thick](-3,0) +(1,-1) -- +(-1,1); 
      \node[black] at (-1.5,0){$-$}; \draw[very thick](0,0) +(-1,-1) --
      node[pos=.8,fill=green!50!black,inner sep=3pt]{} +(1,1); \draw[very thick](0,0) +(1,-1) -- +(-1,1)
     ; \node[black] at (2,0){$=$}; \draw[very
      thick](4,0) +(-1,-1) -- +(-1,1); \draw[very
      thick](4,0) +(0,-1) -- +(0,1); 
    \end{tikzpicture}
  \end{equation*}
  \begin{equation*}\subeqn\label{NilHecke3}
    \begin{tikzpicture}[very thick,scale=.9,yscale=.8,baseline,green!50!black]
      \draw(-2.8,0) +(0,-1) .. controls (-1.2,0) ..  +(0,1); \draw (-1.2,0) +(0,-1) .. controls
      (-2.8,0) ..  +(0,1) ; 
    \end{tikzpicture}\hspace{4mm}
= 0\qquad \qquad 
    \begin{tikzpicture}[very thick,scale=.9,yscale=.8,baseline,green!50!black]
      \draw (-3,0) +(1,-1) -- +(-1,1); \draw
      (-3,0) +(-1,-1) -- +(1,1) ; \draw
      (-3,0) +(0,-1) .. controls (-4,0) ..  +(0,1); 
    \end{tikzpicture}\hspace{4mm}=\hspace{4mm}
\begin{tikzpicture}[very thick,scale=.9,yscale=.8,baseline,green!50!black]
\draw (1,0) +(1,-1) -- +(-1,1)
     ; \draw (1,0) +(-1,-1) -- +(1,1)
      ; \draw (1,0) +(0,-1) .. controls
      (2,0) ..  +(0,1); 
    \end{tikzpicture}\hspace{4mm}
  \end{equation*}
\[ \subeqn\label{qHghost-bigon1}
\begin{tikzpicture}[very thick,xscale=1.3,yscale=.8,baseline=25pt,green!50!black]
 \draw (1,0) to[in=-90,out=90]  (1.5,1) to[in=-90,out=90] (1,2)
;
  \draw[dashed] (1.5,0) to[in=-90,out=90] (1,1) to[in=-90,out=90] (1.5,2);
  \draw (2.5,0) to[in=-90,out=90]  (2,1) to[in=-90,out=90] (2.5,2);
\node[black] at (3,1) {=};
  \draw (3.7,0) --node[midway,fill,inner sep=3pt]{} (3.7,2) 
 ;
  \draw[dashed] (4.2,0) to (4.2,2);
  \draw (5.2,0) -- (5.2,2);
\node[black] at (5.6,1) {$-\pq$};
  \draw (6.2,0) -- (6.2,2);
  \draw[dashed] (6.7,0)-- (6.7,2);
  \draw (7.7,0) -- node[midway,fill,inner sep=3pt]{} (7.7,2);
\end{tikzpicture}
\]
\[ \subeqn\label{qHghost-bigon2}
\begin{tikzpicture}[very thick,xscale=1.3,yscale=.8,baseline=25pt,green!50!black]
 \draw[dashed]  (1,0) to[in=-90,out=90]  (1.5,1) to[in=-90,out=90] (1,2)
;
  \draw(1.5,0) to[in=-90,out=90] (1,1) to[in=-90,out=90] (1.5,2);
  \draw (2,0) to[in=-90,out=90]  (2.5,1) to[in=-90,out=90] (2,2);
\node[black] at (3,1) {=};
  \draw[dashed] (3.7,0) --(3.7,2) 
 ;
  \draw (4.2,0) to node[midway,fill,inner sep=3pt]{}  (4.2,2);
  \draw (4.7,0) -- (4.7,2);
\node[black] at (5.6,1) {$-\pq$};
  \draw[dashed] (6.2,0) -- (6.2,2);
  \draw (6.7,0)-- (6.7,2);
  \draw (7.2,0) -- node[midway,fill,inner sep=3pt]{} (7.2,2);
\end{tikzpicture}
\]
\begin{equation*}\label{eq:triple-point-1}\subeqn
    \begin{tikzpicture}[very thick,xscale=1.5,yscale=.8,baseline,green!50!black]
      \draw[dashed] (-3,0) +(.4,-1) -- +(-.4,1);
 \draw[dashed]      (-3,0) +(-.4,-1) -- +(.4,1); 
    \draw (-2,0) +(.4,-1) -- +(-.4,1); \draw
      (-2,0) +(-.4,-1) -- +(.4,1); 
 \draw (-3,0) +(0,-1) .. controls (-3.5,0) ..  +(0,1);\node at (-1,0) {=};  \draw[dashed] (0,0) +(.4,-1) -- +(-.4,1);
 \draw[dashed]      (0,0) +(-.4,-1) -- +(.4,1); 
    \draw (1,0) +(.4,-1) -- +(-.4,1); \draw
      (1,0) +(-.4,-1) -- +(.4,1); 
 \draw (0,0) +(0,-1) .. controls (.5,0) ..  +(0,1);
\node[black] at (2.1,0) {$-\pq$};
     \draw (4,0)
      +(.4,-1) -- +(.4,1); \draw (4,0)
      +(-.4,-1) -- +(-.4,1); 
 \draw[dashed] (3,0)
      +(.4,-1) -- +(.4,1); \draw[dashed] (3,0)
      +(-.4,-1) -- +(-.4,1); 
\draw (3,0)
      +(0,-1) -- +(0,1);
    \end{tikzpicture}
  \end{equation*}
\begin{equation*}\label{eq:triple-point-2}\subeqn
    \begin{tikzpicture}[very thick,xscale=1.5,yscale=.8,baseline,green!50!black]
      \draw (-3,0) +(.4,-1) -- +(-.4,1);
 \draw     (-3,0) +(-.4,-1) -- +(.4,1); 
\draw (-2,0) +(0,-1) .. controls (-2.5,0) ..  +(0,1);
 \draw[dashed] (-3,0) +(0,-1) .. controls (-3.5,0) ..  +(0,1);\node[black] at (-1,0) {=};  \draw (0,0) +(.4,-1) -- +(-.4,1);
 \draw   (0,0) +(-.4,-1) -- +(.4,1); 
    \draw[dashed] (0,0) +(0,-1) .. controls (.5,0) ..  +(0,1);
 \draw (1,0) +(0,-1) .. controls (1.5,0) ..  +(0,1);
\node[black] at (2,0)
      {$+$};   
 \draw (3,0)
      +(.4,-1) -- +(.4,1); \draw (3,0)
      +(-.4,-1) -- +(-.4,1); 
 \draw[dashed] (3,0)
      +(0,-1) -- +(0,1); \draw (4,0)
      +(0,-1) -- +(0,1); 
    \end{tikzpicture}
  \end{equation*}
 \begin{equation*}\label{qHcost}\subeqn
  \begin{tikzpicture}[very thick,baseline,xscale=1.5,yscale=.8]
    \draw [wei]  (-1.8,0)  +(0,-1) -- node[ below, at start]{$\PQ_i$}+(0,1);
       \draw[green!50!black](-1.2,0)  +(0,-1) .. controls (-2.8,0) ..  +(0,1);
           \node at (-.5,0) {=};
    \draw [green!50!black] (1.5,0)  +(0,-1) -- node[midway,
       fill=green!50!black,inner sep=2.5pt]{}+(0,1);
       \draw[wei] (.8,0)  +(0,-1) -- node[ below, at start]{$\PQ_i$}+(0,1);
 \node at (2.8,0) {$-\PQ_i$};
        \draw[green!50!black] (4.6,0)  +(0,-1) -- +(0,1);
       \draw [wei] (3.9,0)  +(0,-1) --node[ below, at start,black]{$\PQ_i$} +(0,1);
  \end{tikzpicture}
\end{equation*}
 \excise{ \begin{equation*}\label{red-triple}\subeqn
  \begin{tikzpicture}[very thick,baseline=-2pt,scale=.7]
    \draw [wei]  (0,-1) -- (0,1);
       \draw[green!50!black](.5,-1) to[out=90,in=-30] (-.5,1);
       \draw[green!50!black](-.5,-1) to[out=30,in=-90] (.5,1);
  \end{tikzpicture}- \begin{tikzpicture}[very thick,baseline=-2pt,scale=.7]
    \draw [wei]  (0,-1) -- (0,1);
       \draw[green!50!black](.5,-1) to[out=150,in=-90] (-.5,1);
       \draw[green!50!black](-.5,-1) to[out=90,in=-150] (.5,1);
  \end{tikzpicture}
=-\PQ_\bullet \Bigg(\,\begin{tikzpicture}[very thick,baseline=-2pt,scale=.7]
    \draw [wei]  (0,-1) -- (0,1);
       \draw[green!50!black](.5,-1) to[out=90,in=-90] (.5,1);
       \draw[green!50!black](-.5,-1) to[out=90,in=-90] node[midway,
       fill=green!50!black,inner sep=2.5pt]{} (-.5,1);
  \end{tikzpicture}\,-\pq \,\begin{tikzpicture}[very thick,baseline=-2pt,scale=.7]
    \draw [wei]  (0,-1) -- (0,1);
       \draw[green!50!black](.5,-1) to[out=90,in=-90] node[midway,
       fill=green!50!black,inner sep=2.5pt]{} (.5,1);
       \draw[green!50!black](-.5,-1) to[out=90,in=-90] (-.5,1);
  \end{tikzpicture}\,\Bigg)
\end{equation*}}
\begin{equation*}\label{smart-red-triple}\subeqn
    \begin{tikzpicture}[very thick,xscale=1.5,yscale=.8,baseline,green!50!black]
      \draw (-2,0) +(.4,-1) -- +(-.4,1);
 \draw     (-2,0) +(-.4,-1) -- +(.4,1); 
 \draw[wei] (-2,0) +(0,-1) .. controls (-2.5,0) ..   node[ below, at start,black]{$\PQ_i$}+(0,1);
\node[black] at (-1,0) {=};  \draw (0,0) +(.4,-1) -- +(-.4,1);
 \draw   (0,0) +(-.4,-1) -- +(.4,1); 
    \draw[wei] (0,0) +(0,-1) .. controls (.5,0) .. node[ below, at start,black]{$\PQ_i$} +(0,1);
 \node[black] at (1,0)
      {$+$};   
 \draw (2,0)
      +(.4,-1) -- +(.4,1); \draw (2,0)
      +(-.4,-1) -- +(-.4,1); 
 \draw[wei] (2,0)
      +(0,-1) -- +(0,1); 
    \end{tikzpicture}
  \end{equation*}
 \begin{equation*}\label{dumb-red-triple}\subeqn
   \begin{tikzpicture}[very thick,baseline=-2pt,yscale=.8,]
    \draw [wei]  (0,-1) --node[ below, at start,black]{$\PQ_i$} (0,1);
       \draw[green!50!black,dashed](.5,-1) to[out=90,in=-30] (-.5,1);
       \draw[green!50!black](-.5,-1) to[out=30,in=-90] (.5,1);
  \end{tikzpicture}= \begin{tikzpicture}[very thick,baseline=-2pt,yscale=.8,]
    \draw [wei]  (0,-1) -- node[ below, at start,black]{$\PQ_i$} (0,1);
       \draw[green!50!black,dashed](.5,-1) to[out=150,in=-90] (-.5,1);
       \draw[green!50!black](-.5,-1) to[out=90,in=-150] (.5,1);
  \end{tikzpicture}\qquad \qquad   \begin{tikzpicture}[very thick,baseline=-2pt,yscale=.8,]
    \draw [wei]  (0,-1) -- node[ below, at start,black]{$\PQ_i$} (0,1);
       \draw[green!50!black](.5,-1) to[out=90,in=-30] (-.5,1);
       \draw[green!50!black,dashed](-.5,-1) to[out=30,in=-90] (.5,1);
  \end{tikzpicture}= \begin{tikzpicture}[very thick,baseline=-2pt,yscale=.8]
    \draw [wei]  (0,-1) -- node[ below, at start,black]{$\PQ_i$} (0,1);
       \draw[green!50!black](.5,-1) to[out=150,in=-90] (-.5,1);
       \draw[green!50!black,dashed](-.5,-1) to[out=90,in=-150] (.5,1);
  \end{tikzpicture}\qquad \qquad   \begin{tikzpicture}[very thick,yscale=.8,baseline=-2pt]
    \draw [wei]  (0,-1) -- node[ below, at start,black]{$\PQ_i$} (0,1);
       \draw[green!50!black,dashed](.5,-1) to[out=90,in=-30] (-.5,1);
       \draw[green!50!black,dashed](-.5,-1) to[out=30,in=-90] (.5,1);
  \end{tikzpicture}= \begin{tikzpicture}[very thick,yscale=.8,baseline=-2pt]
    \draw [wei]  (0,-1) --node[ below, at start,black]{$\PQ_i$} (0,1);
       \draw[green!50!black,dashed](.5,-1) to[out=150,in=-90] (-.5,1);
       \draw[green!50!black,dashed](-.5,-1) to[out=90,in=-150] (.5,1);
  \end{tikzpicture}
\end{equation*}
and the non-local relation that a idempotent is 0 if the strands can be
divided into two groups with a gap $>|\kappa|$ between them
and all red strands in the right hand group.  

Some care must be used when understanding what it means to apply these
relations locally.  In each case, the LHS and RHS have a dominant term
which are related to each other via an isotopy through a disallowed
diagram with a tangency, triple point or a square on a crossing.  You can only apply the
relations if this isotopy avoids tangencies, triple points and squares on crossings
everywhere else in the diagram;  in particular, all other strands must
avoid the region where the relation is applied.  One can always choose isotopy
representatives sufficiently generic for this to hold.
\end{definition}

\subsection{A Morita equivalence}
\label{sec:morita-equivalence}

One must be slightly careful in the definition of these algebras,
since as described they have $\aleph_1$ many idempotents. We'll
usually fix a finite collection $\mathscr{D}$ of subsets of $\R$ and consider the subalgebra
$\PC^\vartheta_{\mathscr{D}}$ where the green
strands at the top
and bottom of every diagram is equal to one of the sets in
$\mathscr{D}$.  This subalgebra is finite dimensional. In fact, we'll
describe a basis of it in Lemma \ref{lem:hbasis}. Recall that for each $\ell$-multipartition $\xi$, we have a
subset $D_\xi$ defined as in Figure \ref{fig:partition}.  Let
$\mathscr{D}^\circ_m$ be the collection of these for all
$\ell$-multipartitions of size $m$.  

\begin{lemma}\label{lem:Morita}
  For all collections  $\mathscr{D}$ of $m$-element subsets containing
  $\mathscr{D}^\circ_m$, the inclusion
  $\PC^\vartheta_{\mathscr{D}^\circ_m}\to \PC^\vartheta_{\mathscr{D}}$
  induces a Morita equivalence.
\end{lemma}
\begin{proof}
Throughout this proof, we write ``diagram'' to mean a WF Hecke diagram.
  Let $e_{\circ}$ be the idempotent given by the sum of straight-line
  diagrams for the subsets $D_\xi$, then we already know that
  $e_{\circ}\PC^\vartheta_{\mathscr{D}}e_{\circ}=\PC^\vartheta_{\mathscr{D}^\circ_m}$,
  so in order to show that $e_{\circ}\PC^\vartheta_{\mathscr{D}}$
  induces a Morita equivalence, we need only show that
  $\PC^\vartheta_{\mathscr{D}}e_{\circ}\PC^\vartheta_{\mathscr{D}}=\PC^\vartheta_{\mathscr{D}}$.  
We'll also simplify by only considering the case where $\ck
<0$.  The case where $\ck>0$ follows by similar arguments.

    The underlying idea of the proof is that at $y=\nicefrac{1}{2}$, we
  push strands as far to the left as possible.  We do this by a series
  of reductions:

{\bf Push to left-justified:} Fix a real number
  $\epsilon$.  By applying an isotopy, we may assume that for any strand in the horizontal slice at
  $y=\nicefrac{1}{2}$, there is either a strand (red or black) or a
  ghost within $\epsilon$ to its left, or a strand within $\epsilon$
  to the left of its ghost.  Otherwise, we can simply move this
  strand to the left by $\epsilon$.  Eventually this process will
  terminate, or the slice at $y=\nicefrac{1}{2}$ will be unsteady and
  thus 0.  We call such a diagram {\bf left-justified}.

This defines an equivalence relation on strands generated by imposing that two strands are
equivalent if one is with $\epsilon$ of the other or its ghost.  Once
we shrink $\epsilon$ to be much smaller than
$\vartheta_i-\vartheta_j-p\ck$ for all $i\neq j\in [1,\ell]$ and $p\in \Z$, we cannot have any pair
of red strands which are equivalent, since the distance between two
equivalent strands must within $m\epsilon$ of a multiple of
$\ck$. On the other hand, every equivalence class must contain a
red strand, since otherwise, we can simply shift all its elements
$\epsilon$ units to
the left.   

{\bf Preorder on left-justified diagrams:} We now place a preorder on left-justified diagrams, given by the
dominance order on the slice at $y=\nicefrac 12$ and then ordering by the distance of
dots from the red line in its equivalence class.  That is, for each
equivalence class, we have a function $\delta(t)$ given by the number
of dots on strands in the equivalence class within $t$ units of the
red strand.  If we have two left-justified diagrams $a,b$ with the
same slice at $y=\nicefrac 12$, then $a\geq b$ if $\delta_a\geq \delta_b$ for every equivalence
class.  

The engine of the proof is that we will show: {\it Unless a
  left-justified diagram has slice at $y=\nicefrac{1}{2}$
  corresponding to a multi-partition, it can be rewritten in terms of
  diagrams which are higher in this preorder.}

{\bf Remove L's:} Now we wish to rule out certain configurations that
correspond to collections of boxes which are not Young diagrams.  
The first of these are L's such as the diagrams below:
 
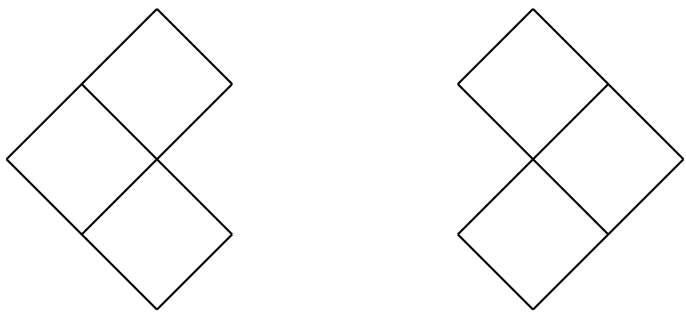
\[\tikz { \node at (-3,0){
      \tikz[thick,scale=2]{\draw (-1,0) -- (0,1); \draw (0,1) --
        (.5,.5); \draw (-1,0) -- (0,-1); \draw (0,-1) -- (.5,-.5);
        \draw (.5,.5) -- (-.5,-.5); \draw (-.5,.5) -- (.5,-.5); } }; \node at (3,0){ \tikz[thick,scale=-2]{\draw (-1,0) -- (0,1); \draw (0,1) --
        (.5,.5); \draw (-1,0) -- (0,-1); \draw (0,-1) -- (.5,-.5);
        \draw (.5,.5) -- (-.5,-.5); \draw (-.5,.5) -- (.5,-.5);  } };
  }\] Remember that
we draw in Russian notation, so the Young diagram condition is that
boxes will not fall when gravity is applied.

Depending on the sign of $\kappa$, one of these diagrams corresponds
to the situation where we have a pair of black strands within $2\epsilon$ of each other with a ghost
between them, but no strand between their ghosts.  In this case, we can apply
(\ref{eq:triple-point-2}) to write this in terms of slices higher in
this partial order.  
\begin{equation*}
     \begin{tikzpicture}[very thick,xscale=1.5,yscale=.8,baseline,green!50!black]
 \draw (3,0)
      +(.4,-1) -- +(.4,1); \draw (3,0)
      +(-.4,-1) -- +(-.4,1); 
 \draw[dashed] (3,0)
      +(0,-1) -- +(0,1); \draw (4,0)
      +(0,-1) -- +(0,1); 
    \end{tikzpicture}\quad=   \quad\begin{tikzpicture}[very thick,xscale=1.5,yscale=.8,baseline,green!50!black]
      \draw (-3,0) +(.4,-1)  to[out=90,in=-90] +(-.4,0)
      to[out=90,in=-90] (-3.4,1);
 \draw     (-3,0) +(-.4,-1) to[out=90,in=-90] +(.4,0)
      to[out=90,in=-90] (-2.6,1); 
\draw (-2,0) +(0,-1) .. controls (-2.7,0) ..  +(0,1);
 \draw[dashed] (-3,0) +(0,-1) .. controls (-3.7,0) ..  +(0,1); \end{tikzpicture}\quad- \quad
    \begin{tikzpicture}[very
      thick,xscale=1.5,yscale=.8,baseline,green!50!black]
\draw (0,0) +(.4,-1) to[out=90,in=-90] +(-.4,0)
      to[out=90,in=-90] (-.4,1);
 \draw   (0,0) +(-.4,-1) to[out=90,in=-90] +(-.4,-.4)
      to[out=90,in=-90]  (.4,1); 
    \draw[dashed] (0,0) +(0,-1)-- +(0,1);
 \draw (1,0) +(0,-1) --  +(0,1);
  \end{tikzpicture}
  \end{equation*}
The other will correspond to the situation where there is no ghost between the
strands, but a strand between the ghosts.  In this case  we can apply
(\ref{eq:triple-point-1}) similarly.

Thus, we need only consider the
possibilities that that two consecutive strands within $2\epsilon$ of
each other have both a strand
between ghosts and a ghost between strands, or neither.  These
correspond to the box configurations
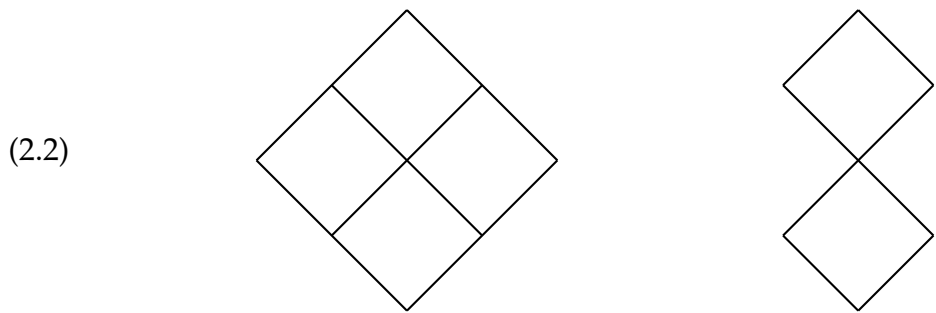
\begin{equation}
 \tikz[baseline] { \node at (-3,0){
      \tikz[thick,scale=2]{\draw (-1,0) -- (0,1); \draw (0,1) --
        (1,0); \draw (-1,0) -- (0,-1); \draw (0,-1) -- (1,0);
        \draw (.5,.5) -- (-.5,-.5); \draw (-.5,.5) -- (.5,-.5); } }; \node at (3,0){ \tikz[thick,scale=-2]{\draw (-.5,.5) -- (0,1); \draw (0,1) --
        (.5,.5); \draw (-.5,-.5) -- (0,-1); \draw (0,-1) -- (.5,-.5);
     \draw (.5,.5) -- (-.5,-.5); \draw (-.5,.5) -- (.5,-.5);  } };
  }\label{eq:box-on-box}
\end{equation}

The first of these is exactly what we are looking for, and the second
will be ruled out by other means.

{\bf Remove dots}:  Fix an equivalence class which has at least one
dot on a strand.  Consider the point of the closest dot in this equivalence class to
the red line.  The strand that this dot sits on must be constrained
from moving left by a ghost or a red strand.  If it is a red strand,
then we can apply the relation (\ref{qHcost}) to write this in terms
of a diagram with slice higher in dominance order and the diagram with the dot removed.  

If the dot is to the left of the red line, then it cannot be
constrained only by a strand left of its ghost,   since in this case,
we can just shift the strand and all to its left in the equivalence
class $\epsilon$ units
leftward.  As usual, this process must terminate or the idempotent
will be 0 in $\PC^\vartheta$.   Thus, either a
ghost or strand to its left is constraining it. 

If the constraint is a ghost, we can apply
(\ref{qHghost-bigon2}) to move this strand left.  The correction term
will have a dot closer to the red strands.
If the constraint is a strand, then we can apply the relation  
\begin{equation}
\begin{tikzpicture}[very thick,baseline, green!50!black]
  \draw (0,-1) --  (0,1);
  \draw (1,-1)-- (1,1);
\end{tikzpicture}
=
\begin{tikzpicture}[very thick,baseline,green!50!black]
  \draw (0,-1) to[out=90,in=-90] (1,0)
  to[out=90,in=-90] node[pos=.75, fill=green!50!black,inner sep=2pt]{} (0,1);
  \draw (1,-1)to[out=90,in=-90] node[below, at start]{$j$} (0,0)
  to[out=90,in=-90] node[at start, fill=green!50!black,inner sep=2pt]{} (1,1);
\end{tikzpicture}
-
\begin{tikzpicture}[very thick,baseline,green!50!black]
  \draw (0,-1) to[out=90,in=-90]  (1,0)
  to[out=90,in=-90] (0,1);
  \draw (1,-1) to[out=90,in=-90] node[pos=.25, fill,inner sep=2pt]{} (0,0)
  to[out=90,in=-90] node[at start, fill,inner sep=2pt]{} (1,1);
\end{tikzpicture}\label{eq:2}
\end{equation} to move the dot left.  Eventually, the dot will encounter a ghost
and we can apply an earlier argument.  Since the dot moved left by no
more than $m\epsilon$ and then moved right by $\kappa$, over all it
has moved right.

Symmetrically, if the dot is right of the red line, then we must have
that it is constrained by a strand, either immediately to its left or
left of its ghost. Otherwise, the original strand and all to its right
in the equivalence class can be moved left by $\epsilon$ units.  We
can apply \eqref{eq:2} for a strand immediately to the left or
(\ref{qHghost-bigon1}) for one left of the ghost, to show that this
factors through a slice higher in dominance order.

Thus, in all cases, if there is a dot anywhere, the diagram can be
written as a sum of ones higher in our preorder.  That is, we can
assume that there are no dots.  

{\bf Remove unsupported boxes:} If we have a consecutive
pair of strands with no ghost between them, this corresponds to the
second configuration of \eqref{eq:box-on-box}, and thus we must rule
it out.  We can apply the relation
\eqref{eq:2}, and rewrite as a sum of diagrams higher in our
order. Using the dot removal process, we see that every pair of
strands within $2\epsilon$ of each other must be separated by a ghost,
and their ghost must be separated by a strand, corresponding to the
first configuration of \eqref{eq:box-on-box}.

{\bf Find the Young diagram:}  
We have now performed sufficient reductions to show that we factor
through $D_\xi$, but let us describe $\xi$ in order to show this more
clearly.  
Each equivalence class breaks up into groups of strands within
$m\epsilon$ of the points $m\kappa+\vartheta_p$ for $m\in \Z$.  
Left of the
red strand, the leftmost element of each group of the equivalence
class must be a ghost, there is a central group where the red strand
itself is left-most, and then right of the red strand left-most
element must be a strand.

This precisely means that the resulting slice has strands at the
points in $D_\xi$ for some $\xi$:  the boxes of $\xi$ are in bijection
with strands; the equivalence classes correspond to the component
partitions in the multipartition $\xi$, the box $(1,1,p)$
corresponding to the strand which blocked by the red line $p$.  Given
the strand for the box $(i,j,p)$, the box $(i+1,j,p)$ is the strand whose
ghost is to the right of it, and the box $(i,j+1,p)$ is the strand
caught on its ghost.  The partition condition is precisely that two
consecutive close black strands correspond to $(i,j,p)$ and
$(i+1,j+1,p)$, the ghost between them to $(i+1,j,p)$ and the strand
between their ghosts to $(i,j+1,p)$.

This shows that the algebra is spanned by elements factoring through
$e_{D_\xi}$ for some multipartition $\xi$.
\end{proof}

\subsection{Relationship to the Hecke algebra}
\label{sec:relat-hecke-algebra}

For a real number $s>0$, let $D_{s,m}$ be the set $\{s,2s,\dots, ms\}$.
For any WF Hecke diagram, we can
embed it into the plane with top and bottom at $s,2s,\dots, ms$,
and if $s\gg |\kappa|$, we can assume that no strand passes between
any crossing and its ghost.  This will happen, for example, if we write the diagram as a composition of the
diagrams of the type 
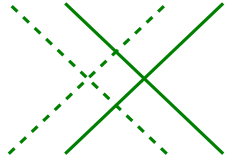
\begin{equation*}
    \begin{tikzpicture}[very thick,xscale=1.5,baseline,green!50!black]
      \draw (-2.5,0) +(.7,-1) -- +(-.7,1);
 \draw     (-2.5,0) +(-.7,-1) -- +(.7,1); 
      \draw[dashed] (-3,0) +(.7,-1) -- +(-.7,1);
 \draw [dashed]    (-3,0) +(-.7,-1) -- +(.7,1); 
    \end{tikzpicture}
  \end{equation*}
If such a strand exists, we
can just increase $s$ by scaling the diagram horizontally; however,
$\kappa$ is left unchanged, so the strand will be pushed out from
between the crossings.  

\begin{proposition}[\mbox{\cite[Thm. \ref{n-wfHecke}]{WebBKnote}}]\label{prop:Hecke-quotient}
For $s\gg |\kappa|$, there is an isomorphism
$H_m(\pq,\PQ_\bullet)\cong \PC^\vartheta_{D_{s,m}}$ sending  
\newseq 
\begin{align*}\label{Hecke-gens1} \subeqn
    \tikz[baseline]{
      \node at (0,0){ 
        \tikz[very thick,xscale=1.2,green!50!black]{
          \draw (-.5,-.5)-- (-.5,.5);
          \draw (.5,-.5)-- (.5,.5) node [midway,fill=green!50!black,inner
          sep=2.5pt]{};
          \draw (1.5,-.5)-- (1.5,.5);
          \node at (1,0){$\cdots$};
          \node at (0,0){$\cdots$};
        }
      };}&\leftrightarrow X_j\\ \label{Hecke-gens2}\subeqn
       \tikz[baseline]{\node[label=below:{}] at (4.5,0){ 
        \tikz[very thick,xscale=1.2, green!50!black]{
          \draw (-.5,-.5)-- (-.5,.5);
          \draw (.1,-.5)-- (.9,.5);
          \draw (.9,-.5)-- (.1,.5);
          \draw (1.5,-.5)-- (1.5,.5);
          \node at (1,0){$\cdots$};
          \node at (0,0){$\cdots$};
        }
      };
    }&\leftrightarrow
    \begin{cases}
      T_j+1 & \kappa<0\\
      T_i-q & \kappa>0
    \end{cases}
\end{align*}
\end{proposition}
\excise{\begin{proof}
 The Hecke relations follow from
  \cite{WebBKnote}
, and the cyclotomic
  relations from the non-local one.  The only thing that remains to be
  checked is that the non-local relation introduces no other new
  ones.  As in the proof of \cite[4.19]{Webmerged}, this follows from moving all but one strand
  right of all reds using (\ref{qHcost}\& \ref{smart-red-triple}), leaving a single one at far left.  Pulling this
  across all reds shows that the non-local relation is equivalent to the
  cyclotomic one via (\ref{qHcost})
.
\end{proof}}

This shows that the category $H_m(\pq,\PQ_\bullet)\mmod$ is a quotient
category of $\PC^\vartheta_{\mathscr{D}}$ for any collection
$\mathscr{D}$ containing $D_{s,m}$.

We can extend this theorem a bit further to a ``relative setting.'' Fix a collection $\mathscr{D}$, and fix $s>0$ such that $s>|d|+|\kappa|$ for all elements
$d\in D\in \mathscr{D}$.  
For $D\in \mathscr{D}$, let $D'=D\cup \{s,\dots, ms\}$, and
$\mathscr{D}'=\{D'\}_{D\in \mathscr{D}}$.  

One can use the same formulas to define a map $\eta\colon \PC^\vartheta_{\mathscr{D}}\otimes
  H_m(\pq)\to \PC^\vartheta_{\mathscr{D}'}$ by sending $a\otimes 1$ to
  the diagram $a$ with vertical stands added at $x=s,2s,\dots, ms$,
  and $1\otimes X_i$ and $1\otimes T_i$ with images as indicated in
  equations (\ref{Hecke-gens1}--\ref{Hecke-gens2}) on the strands at
  $x=s,2s,\dots, ms$, horizontally composed with the identity in
  $\PC^\vartheta_{\mathscr{D}}$.  Schematically, we have
\[a\otimes b \mapsto \quad \tikz[baseline=-2pt,very thick]{\node[draw, inner sep=10pt] at
  (0,0){$a$};  \node[draw, inner sep=10pt] at
  (1.5,0){$b$};}\]

\begin{lemma}\label{lem:horizontal-Hecke}
  The map $\eta\colon \PC^\vartheta_{\mathscr{D}}\otimes
  H_m(\pq)\to \PC^\vartheta_{\mathscr{D}'}$ is a well-defined ring homomorphism.
\end{lemma}
\begin{proof}
We only need to check that horizontally composed diagrams in
$\PC^\vartheta_{\mathscr{D}'}$ satisfy the correct relations.  
The relations of $H_m(\pq)$ are satisfied by the righthand set of
strands by \cite[\ref{n-wdHecke}]{WebBKnote}, since all these relations
are local in nature.  

For $\PC^\vartheta_{\mathscr{D}}$, we need only note that adding a
diagram at the right will not change any of the relations.  This is
clear for the local relations, and unsteady idempotents remain
unsteady, so the only non-local relation is preserved as well.
\end{proof}

Assume that $\mathscr{D}$ is a collection of sets of size $m$, and
$\mathscr{E}$ a collection of sets of size $m+1$. We can consider the
$
\PC^\vartheta_{\mathscr{E}}\operatorname{-}\PC^\vartheta_{\mathscr{D}}$
module $e_{\mathscr{E}}\PC^\vartheta e_{\mathscr{D}'}$ where
$\PC^\vartheta_{\mathscr{D}}$ acts on the right via the map of Lemma \ref{lem:horizontal-Hecke}, and
as before $\mathscr{D}'$ is the collection given by $\mathscr{D}$ with
$\{s\}$ added to each set (where $\{s\}$ is assumed to be $|\ck|$ larger than
any element of $D\in \mathscr{D}$).  Schematically, an element of this
bimodule looks like:
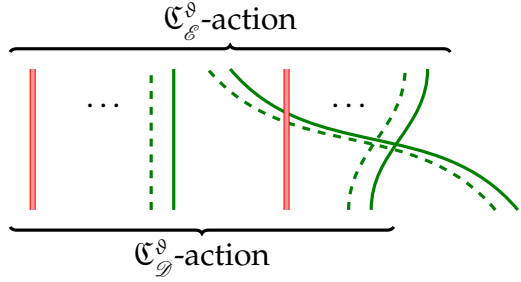
\[\begin{tikzpicture}[very thick,label distance=7pt]
\node at (-4,0){  \begin{tikzpicture}[very thick,yscale=-1.5,xscale=1.5]
\draw[green!50!black] (2.8,2.5) to[out=-130, in=50] (.25,1.25);
\draw[dashed,green!50!black] (2.6,2.5) to[out=-130, in=50] (.05,1.25);
\draw[wei] (-1.5,2.5)  to[out=-90, in=90] (-1.5,1.25);
\draw[green!50!black] (-.25,2.5)  to[out=-90, in=90] (-.25,1.25);
\draw[dashed,green!50!black] (-.45,2.5)  to[out=-90, in=90] (-.45,1.25);
\draw[wei] (.75,2.5) to[out=-90, in=90] (.75,1.25);
\draw[green!50!black] (1.5,2.5) to[out=-90, in=90] (2,1.25);
\draw[dashed,green!50!black] (1.3,2.5) to[out=-90, in=90] (1.8,1.25);
\node at (-.85,1.6) {$\cdots$};
\node at (1.325,1.6) {$\cdots$};

\draw[decorate,decoration=brace,-] (-1.7,1.1) --
    node[above,midway]{$\PC^\vartheta_{\mathscr{E}}$-action} (2.2,1.1);
\draw[decorate,decoration=brace,-] (1.7,2.65) --
    node[below,midway]{$\PC^\vartheta_{\mathscr{D}}$-action} (-1.7,2.65);
  \end{tikzpicture}};
  \end{tikzpicture}\]

Tensor and $\Hom$ with this bimodule induces adjoint $\mathscr{R}$-linear
    induction and restriction functors
    \[\operatorname{ind}\colon \PC^\vartheta_{\mathscr{D}}\to \PC^\vartheta_{\mathscr{E}}\qquad
    \operatorname{res}\colon  \PC^\vartheta_{\mathscr{E}}\to
    \PC^\vartheta_{\mathscr{D}}.\]
If we take $\mathscr{D}=\{{D}_{s,m-1}\}$ and
$\mathscr{E}=\{{D}_{s,m}\}$, then these functors coincide with
the usual induction and restriction functors for Hecke algebras.
We'll consider some important properties of these functors later.

\subsection{Cellular structure}
\label{sec:cellular-structure}

In this section, we define a cellular structure on this algebra.  Consider
a generic subset $D\subset \R$, and a $D$-tableau $\sS$ of shape
$\xi$.  We will describe a WF Hecke diagram
$\hB_{\sS}\in e_{D_\xi}\PC^\vartheta e_{D}$ which matches $D_\xi$ at
the top $y=1$ (i.e. its points are given by the projection of boxes in
the diagram, as in Figure \ref{fig:partition}) and given by the set
$D$ at the bottom.  The strands at the top are naturally in bijection
with boxes in the diagram of $\xi$, and those at the bottom have a
bijection given by the tableau $\sS$.  The strands of the diagram
$\hB_{\sS}$ connect the top and the bottom using this bijection,
without creating any bigons between pairs of strands or strands and
ghosts.  This diagram is not unique up to isotopy (since we have not
specified how to resolve triple points), but we can choose one such
diagram arbitrarily.  We let $*$ denote the
  reflection of a diagram through a horizontal axis, $\hB^*_\sS$ is
  this same diagram with top and bottom reversed.

\begin{example}\label{big-example-1}
Consider the example where $q=-1$ and $\ck=-4$, with
 $Q_1=1,Q_2=-1$ and $d=2$. 
The resulting category, weighted order, and basis only depend on
the difference of the weights $\vartheta_1-\vartheta_2$.  In fact,
there are only 3 different possibilities; the category changes when
this value passes $\pm 4$.
 
There are 5 multipartitions of size 2:
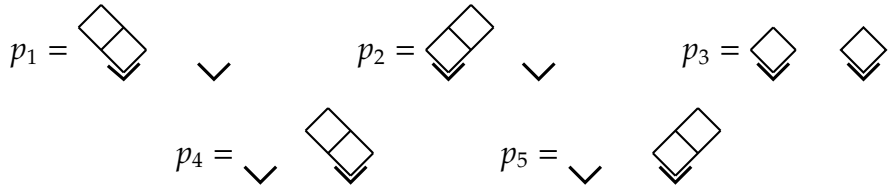
\[
p_1=\tikz[baseline=5pt,scale=.3,thick]{\draw (0,0) --(-2,2); \draw
  (0,0) --(1,1); \draw (-1,3) --(-2,2); \draw (-1,3) --(1,1); \draw
  (-1,1) -- (0,2); \draw[very thick] (-.7,.3) -- (0,-.3) -- (.7,.4); 
\draw[very thick] (3.3,.4) -- (4,-.3) -- (4.7,.4); 
}
\qquad \qquad 
p_2=\tikz[baseline=5pt,scale=.3,thick]{\draw (0,0) --(2,2); \draw
  (0,0) --(-1,1); \draw (1,3) --(2,2); \draw (1,3) --(-1,1); \draw
  (1,1) -- (0,2); \draw[very thick] (-.7,.4) -- (0,-.3) -- (.7,.4); 
\draw[very thick] (3.3,.4) -- (4,-.3) -- (4.7,.4); 
}
\qquad \qquad 
p_3=\tikz[baseline=5pt,scale=.3,thick]{\draw (0,0) --(1,1); \draw
  (0,0) --(-1,1); \draw (0,2) --(-1,1); \draw
  (1,1) -- (0,2); \draw[very thick] (-.7,.4) -- (0,-.3) -- (.7,.4); 
\draw[very thick] (3.3,.4) -- (4,-.3) -- (4.7,.4); \draw (4,0)
--(3,1)-- (4,2) -- (5,1)--cycle;
}
\]
\[
p_4=\tikz[baseline=5pt,scale=.3,thick]{\draw (0,0) --(-2,2); \draw
  (0,0) --(1,1); \draw (-1,3) --(-2,2); \draw (-1,3) --(1,1); \draw
  (-1,1) -- (0,2); \draw[very thick] (-.7,.4) -- (0,-.3) -- (.7,.4); 
\draw[very thick] (-3.3,.4) -- (-4,-.3) -- (-4.7,.4); 
}
\qquad \qquad 
p_5=\tikz[baseline=5pt,scale=.3,thick]{\draw (0,0) --(2,2); \draw
  (0,0) --(-1,1); \draw (1,3) --(2,2); \draw (1,3) --(-1,1); \draw
  (1,1) -- (0,2); \draw[very thick] (-.7,.4) -- (0,-.3) -- (.7,.4); 
\draw[very thick] (-3.3,.4) -- (-4,-.3) -- (-4.7,.4); 
}
\]

\subsubsection*{Case 1: $\vartheta_1-\vartheta_2<-4$}
We'll exemplify this case with $\vartheta_1=0,\vartheta_2=9$.
In this case, our order is $p_1 > p_2> p_3>p_4>p_5$. 

Consider the set $\mathscr{D}=\{\{1,3\},\{2,7\},
\{8,10\}\}$.  For this collection, the tableaux with their
corresponding $\hB_\sS$'s are:
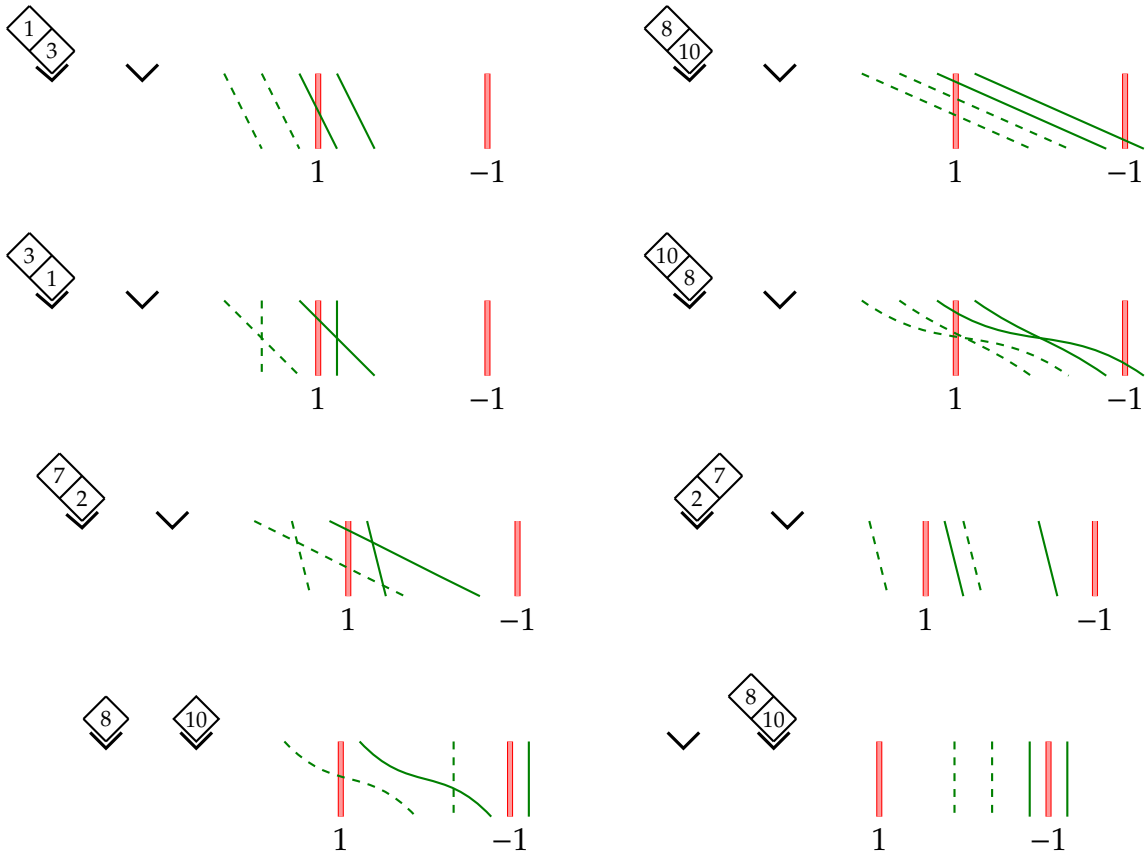
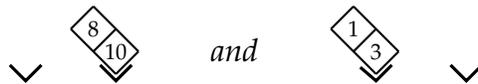
\[
\tikz[baseline=5pt,scale=.3,thick]{\draw (0,0) --(-2,2); \draw (0,0)
  --(1,1); \draw (-1,3) --(-2,2); \draw (-1,3) --(1,1); \draw (-1,1)
  -- (0,2); \draw[very thick] (-.7,.3) -- (0,-.3) -- (.7,.4);
  \draw[very thick] (3.3,.4) -- (4,-.3) -- (4.7,.4); \node[scale=.7]
  at (0,1) {$3$}; \node[scale=.7] at (-1,2) {$1$}; } \qquad
\tikz[baseline=5pt,xscale=.25, yscale=-1,thick]{\draw[wei] (0,0) -- node[below,at
  end]{$1$} (0,1); \draw[wei] (9,0) -- (9,1) node[below, at
  end]{$-1$};\draw[green!50!black] (-1,0) --  (1,1); \draw[green!50!black]
  (1,0) --  (3,1);\draw[dashed, green!50!black] (-5,0) --
  (-3,1); \draw[dashed, green!50!black] (-3,0) -- (-1,1);} \qquad \qquad
\tikz[baseline=5pt,scale=.3,thick]{\draw (0,0) --(-2,2); \draw (0,0)
  --(1,1); \draw (-1,3) --(-2,2); \draw (-1,3) --(1,1); \draw (-1,1)
  -- (0,2); \draw[very thick] (-.7,.3) -- (0,-.3) -- (.7,.4);
  \draw[very thick] (3.3,.4) -- (4,-.3) -- (4.7,.4); \node[scale=.7]
  at (0,1) {$10$}; \node[scale=.7] at (-1,2) {$8$};}\qquad
\tikz[baseline=5pt,xscale=.25, yscale=-1,thick]{\draw[wei] (0,0) -- node[below,at
  end]{$1$} (0,1); \draw[wei] (9,0) -- node[below, at end]{$-1$}
  (9,1);\draw[green!50!black] (-1,0) --
  (8,1); \draw[green!50!black] (1,0)-- (10,1);\draw[dashed, green!50!black]
  (-5,0) -- (4,1); \draw[dashed, green!50!black] (-3,0) -- (6,1);}\]
\[\tikz[baseline=5pt,scale=.3,thick]{\draw (0,0) --(-2,2); \draw (0,0)
  --(1,1); \draw (-1,3) --(-2,2); \draw (-1,3) --(1,1); \draw (-1,1)
  -- (0,2); \draw[very thick] (-.7,.3) -- (0,-.3) -- (.7,.4);
  \draw[very thick] (3.3,.4) -- (4,-.3) -- (4.7,.4); \node[scale=.7]
  at (0,1) {$1$}; \node[scale=.7] at (-1,2) {$3$}; } \qquad
\tikz[baseline=5pt,xscale=.25, yscale=-1,thick]{\draw[wei] (0,0) -- node[below,at
  end]{$1$} (0,1); \draw[wei] (9,0) -- (9,1) node[below, at
  end]{$-1$};\draw[green!50!black] (-1,0) --  (3,1); \draw[green!50!black]
  (1,0) --  (1,1);\draw[dashed, green!50!black] (-5,0) --
  (-1,1); \draw[dashed, green!50!black] (-3,0) -- (-3,1);} 
\qquad \qquad \tikz[baseline=5pt,scale=.3,thick]{\draw (0,0) --(-2,2); \draw
  (0,0) --(1,1); \draw (-1,3) --(-2,2); \draw (-1,3) --(1,1); \draw
  (-1,1) -- (0,2); \draw[very thick] (-.7,.3) -- (0,-.3) -- (.7,.4); 
\draw[very thick] (3.3,.4) -- (4,-.3) -- (4.7,.4); \node[scale=.7] at (0,1)
{$8$}; \node[scale=.7] at (-1,2) {$10$};}\qquad
\tikz[baseline=5pt,xscale=.25, yscale=-1,thick]{\draw[wei] (0,0) -- node[below,at
  end]{$1$} (0,1); \draw[wei] (9,0) --  node[below, at end]{$-1$}
  (9,1);\draw[green!50!black] (-1,0) to[out=10,in=-170]  
  (10,1); \draw[green!50!black] (1,0) to[out=10,in=-170]  (8,1);\draw[dashed, green!50!black] (-5,0) to[out=10,in=-170] (6,1); \draw[dashed, green!50!black] (-3,0) to[out=10,in=-170] (4,1);}
\]
\[ \tikz[baseline=5pt,scale=.3,thick]{\draw (0,0) --(-2,2); \draw
  (0,0) --(1,1); \draw (-1,3) --(-2,2); \draw (-1,3) --(1,1); \draw
  (-1,1) -- (0,2); \draw[very thick] (-.7,.3) -- (0,-.3) -- (.7,.4); 
\draw[very thick] (3.3,.4) -- (4,-.3) -- (4.7,.4); \node[scale=.7] at (0,1)
{$2$}; \node[scale=.7] at (-1,2) {$7$};}
\qquad \tikz[baseline=5pt,xscale=.25, yscale=-1,thick]{\draw[wei] (0,0) --
  node[below,at end]{$1$} (0,1); \draw[wei] (9,0) --  node[below, at
  end]{$-1$} (9,1);\draw[green!50!black] (-1,0) to[out=7,in=-173]  (7,1); \draw[green!50!black] (1,0) --   (2,1);\draw[dashed, green!50!black] (-5,0) to[out=7,in=-173] (3,1); \draw[dashed, green!50!black] (-3,0) -- (-2,1);}\qquad\qquad
\tikz[baseline=5pt,scale=.3,thick]{\draw (0,0) --(2,2); \draw
  (0,0) --(-1,1); \draw (1,3) --(2,2); \draw (1,3) --(-1,1); \draw
  (1,1) -- (0,2); \draw[very thick] (-.7,.4) -- (0,-.3) -- (.7,.4); 
\draw[very thick] (3.3,.4) -- (4,-.3) -- (4.7,.4); \node[scale=.7] at (0,1)
{$2$}; \node[scale=.7] at (1,2) {$7$};}\qquad 
\tikz[baseline=5pt,xscale=.25, yscale=-1,thick]{\draw[wei] (0,0) -- node[below,at
  end]{$1$} (0,1); \draw[wei] (9,0) --   node[below,at
  end]{$-1$}  (9,1);\draw[green!50!black] (6,0)
 --   (7,1); \draw[green!50!black] (1,0) --  (2,1);\draw[dashed, green!50!black] (2,0) -- (3,1); \draw[dashed, green!50!black] (-3,0) -- (-2,1);}
\]\[
\tikz[baseline=5pt,scale=.3,thick]{\draw (0,0) --(1,1); \draw
  (0,0) --(-1,1); \draw (0,2) --(-1,1); \draw
  (1,1) -- (0,2); \draw[very thick] (-.7,.4) -- (0,-.3) -- (.7,.4); 
\draw[very thick] (3.3,.4) -- (4,-.3) -- (4.7,.4); \draw (4,0)
--(3,1)-- (4,2) -- (5,1)--cycle;\node[scale=.7] at (0,1)
{$8$}; \node[scale=.7] at (4,1) {$10$};
}
\qquad 
\tikz[baseline=5pt,xscale=.25, yscale=-1,thick]{\draw[wei] (0,0) -- node[below,at
  end]{$1$} (0,1); \draw[wei] (9,0) -- node[below, at end]{$-1$}
  (9,1);\draw[green!50!black] (10,0) --   (10,1); \draw[green!50!black] (1,0)
  to[out=15,in=-165]  (8,1);\draw[dashed, green!50!black] (-3,0)
  to[out=15,in=-165] (4,1); \draw[dashed, green!50!black] (6,0) --
  (6,1);}\qquad \qquad
\tikz[baseline=5pt,scale=.3,thick]{\draw (0,0) --(-2,2); \draw
  (0,0) --(1,1); \draw (-1,3) --(-2,2); \draw (-1,3) --(1,1); \draw
  (-1,1) -- (0,2); \draw[very thick] (-.7,.4) -- (0,-.3) -- (.7,.4); 
\draw[very thick] (-3.3,.4) -- (-4,-.3) -- (-4.7,.4); 
\node[scale=.7] at (-1,2)
{$8$}; \node[scale=.7] at (0,1) {$10$};
} \qquad 
\tikz[baseline=5pt,xscale=.25, yscale=-1,thick]{\draw[wei] (0,0) -- node[below,at end]{$1$} (0,1);
  \draw[wei] (9,0) --  node[below, at end]{$-1$} (9,1);\draw[green!50!black] (10,0)
  --   (10,1); \draw[green!50!black] (8,0) to  (8,1);\draw[dashed, green!50!black] (4,0) to (4,1); \draw[dashed, green!50!black] (6,0) -- (6,1);}
\]
Note that \[\tikz[baseline=5pt,scale=.3,thick]{\draw (0,0) --(-2,2); \draw
  (0,0) --(1,1); \draw (-1,3) --(-2,2); \draw (-1,3) --(1,1); \draw
  (-1,1) -- (0,2); \draw[very thick] (-.7,.4) -- (0,-.3) -- (.7,.4); 
\draw[very thick] (-3.3,.4) -- (-4,-.3) -- (-4.7,.4); 
\node[scale=.7] at (-1,2)
{$8$}; \node[scale=.7] at (0,1) {$10$};
}\qquad\text{ and }\qquad \tikz[baseline=5pt,scale=.3,thick]{\draw (0,0) --(-2,2); \draw
  (0,0) --(1,1); \draw (-1,3) --(-2,2); \draw (-1,3) --(1,1); \draw
  (-1,1) -- (0,2); \draw[very thick] (-.7,.3) -- (0,-.3) -- (.7,.4); 
\draw[very thick] (3.3,.4) -- (4,-.3) -- (4.7,.4); \node[scale=.7] at (0,1)
{$3$}; \node[scale=.7] at (-1,2) {$1$};
}\] are not standard tableaux in the usual sense, but are
$D$-tableaux as defined above.

\subsubsection*{Case 2: $-4<\vartheta_1-\vartheta_2<4$}
We'll exemplify this case with $\vartheta_1=0,\vartheta_2=1.5$.
In this case, our partial order is $p_1, p_4> p_3>p_2,p_5$. 

A loading in this case is given by specifying the position the point $a$
labeled $1$ and the point $b$ labeled $2$.  We denote this loading
$\Bi_{a,b}$. With $\mathscr{D}$ as before, the tableaux with their
corresponding $\hB_\sS$'s are:
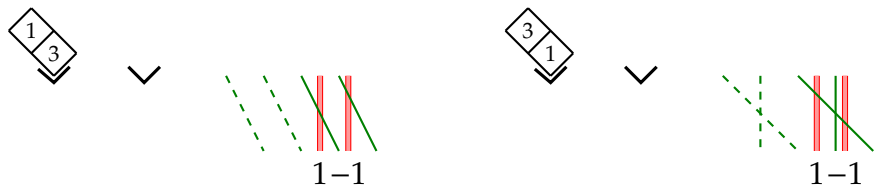
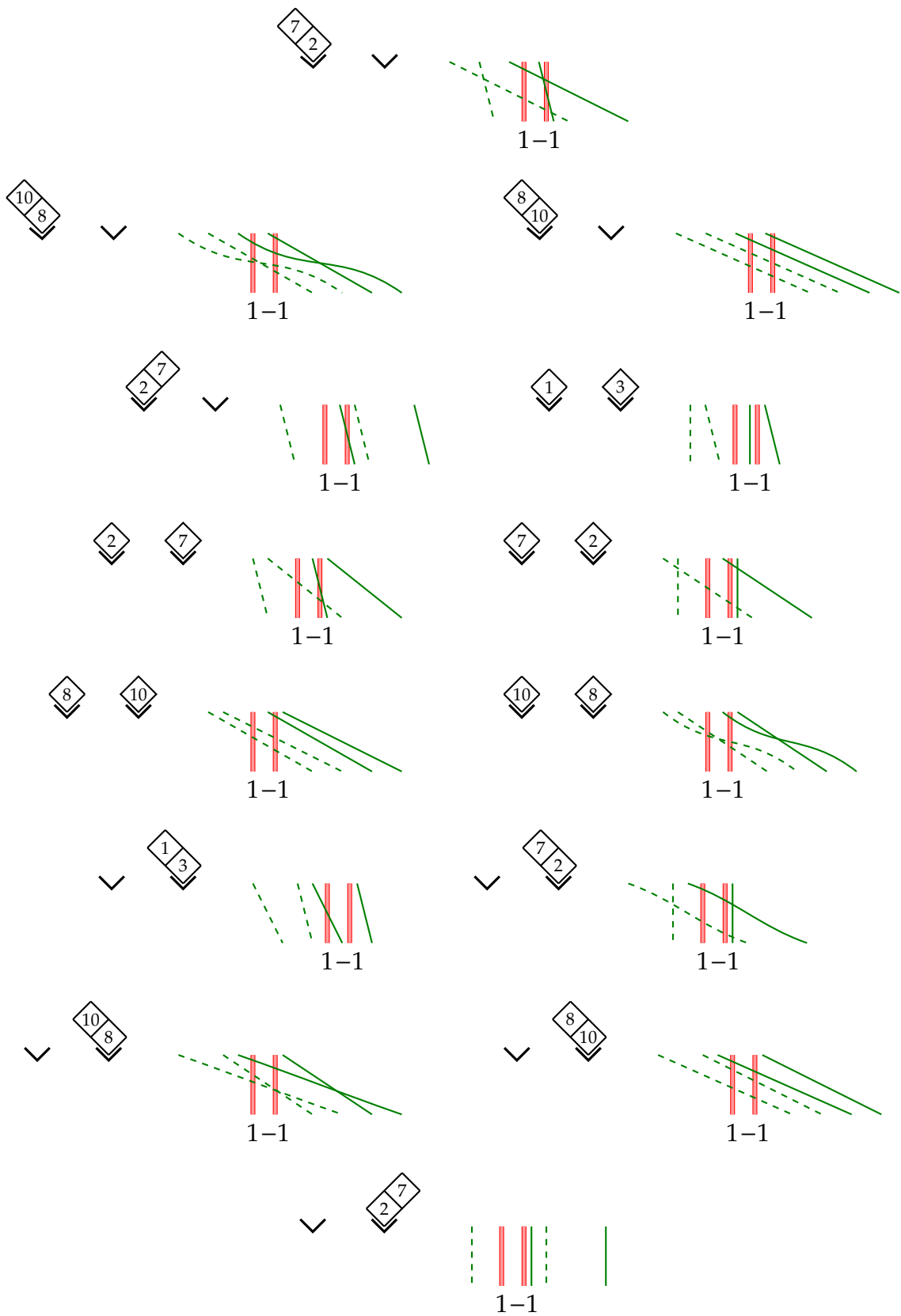
\[
\tikz[baseline=5pt,scale=.3,thick]{\draw (0,0) --(-2,2); \draw
  (0,0) --(1,1); \draw (-1,3) --(-2,2); \draw (-1,3) --(1,1); \draw
  (-1,1) -- (0,2); \draw[very thick] (-.7,.3) -- (0,-.3) -- (.7,.4); 
\draw[very thick] (3.3,.4) -- (4,-.3) -- (4.7,.4); \node[scale=.7] at (0,1)
{$3$}; \node[scale=.7] at (-1,2) {$1$};
}
\qquad 
\tikz[baseline=5pt,xscale=.25, yscale=-1,thick]{\draw[wei] (0,0) -- node[below,at end]{$1$} (0,1); \draw[wei] (1.5,0) -- node[below, at end] {$-1$} (1.5,1);\draw[green!50!black] (-1,0) --  (1,1); \draw[green!50!black] (1,0) -- (3,1);\draw[dashed, green!50!black] (-5,0) -- (-3,1); \draw[dashed, green!50!black] (-3,0) -- (-1,1);}
\qquad \qquad
\tikz[baseline=5pt,scale=.3,thick]{\draw (0,0) --(-2,2); \draw
  (0,0) --(1,1); \draw (-1,3) --(-2,2); \draw (-1,3) --(1,1); \draw
  (-1,1) -- (0,2); \draw[very thick] (-.7,.3) -- (0,-.3) -- (.7,.4); 
\draw[very thick] (3.3,.4) -- (4,-.3) -- (4.7,.4); \node[scale=.7] at (0,1)
{$1$}; \node[scale=.7] at (-1,2) {$3$};
}
\qquad 
\tikz[baseline=5pt,xscale=.25, yscale=-1,thick]{\draw[wei] (0,0) -- node[below,at
  end]{$1$} (0,1); \draw[wei] (1.5,0) -- node[below, at end]
  {$-1$} (1.5,1);\draw[green!50!black] (-1,0) --  (3,1);
  \draw[green!50!black] (1,0) -- (1,1);\draw[dashed, green!50!black]
  (-5,0) -- (-1,1); \draw[dashed, green!50!black] (-3,0) -- (-3,1);}\]
\[
\tikz[baseline=5pt,scale=.3,thick]{\draw (0,0) --(-2,2); \draw
  (0,0) --(1,1); \draw (-1,3) --(-2,2); \draw (-1,3) --(1,1); \draw
  (-1,1) -- (0,2); \draw[very thick] (-.7,.3) -- (0,-.3) -- (.7,.4); 
\draw[very thick] (3.3,.4) -- (4,-.3) -- (4.7,.4); \node[scale=.7] at (0,1)
{$2$}; \node[scale=.7] at (-1,2) {$7$};
}
\qquad 
\tikz[baseline=5pt,xscale=.25, yscale=-1,thick]{\draw[wei] (0,0) -- node[below,at end]{$1$} (0,1); \draw[wei] (1.5,0) -- node[below, at end] {$-1$} (1.5,1);\draw[green!50!black] (-1,0) --  (7,1); \draw[green!50!black] (1,0) -- (2,1);\draw[dashed, green!50!black] (-5,0) -- (3,1); \draw[dashed, green!50!black] (-3,0) -- (-2,1);}
\]
\[
\tikz[baseline=5pt,scale=.3,thick]{\draw (0,0) --(-2,2); \draw
  (0,0) --(1,1); \draw (-1,3) --(-2,2); \draw (-1,3) --(1,1); \draw
  (-1,1) -- (0,2); \draw[very thick] (-.7,.3) -- (0,-.3) -- (.7,.4); 
\draw[very thick] (3.3,.4) -- (4,-.3) -- (4.7,.4); \node[scale=.7] at (0,1)
{$8$}; \node[scale=.7] at (-1,2) {$10$};}\qquad 
\tikz[baseline=5pt,xscale=.25, yscale=-1,thick]{\draw[wei] (0,0) -- node[below,at
  end]{$1$} (0,1); \draw[wei] (1.5,0) --  node[below, at end]
  {$-1$}  (1.5,1);\draw[green!50!black] (-1,0) to[out=10,in=-170]
  (10,1); \draw[green!50!black] (1,0) -- (8,1);\draw[dashed,
  green!50!black] (-5,0) to[out=10,in=-170] (6,1); \draw[dashed,
  green!50!black] (-3,0) -- (4,1);}
\qquad\qquad\tikz[baseline=5pt,scale=.3,thick]{\draw (0,0) --(-2,2); \draw
  (0,0) --(1,1); \draw (-1,3) --(-2,2); \draw (-1,3) --(1,1); \draw
  (-1,1) -- (0,2); \draw[very thick] (-.7,.3) -- (0,-.3) -- (.7,.4); 
\draw[very thick] (3.3,.4) -- (4,-.3) -- (4.7,.4); \node[scale=.7] at (0,1)
{$10$}; \node[scale=.7] at (-1,2) {$8$};}\qquad 
\tikz[baseline=5pt,xscale=.25, yscale=-1,thick]{\draw[wei] (0,0) -- node[below,at
  end]{$1$} (0,1); \draw[wei] (1.5,0) --  node[below, at end]
  {$-1$}  (1.5,1);\draw[green!50!black] (-1,0) to (8,1);
  \draw[green!50!black] (1,0) -- (10,1);\draw[dashed, green!50!black]
  (-5,0) to (4,1); \draw[dashed, green!50!black] (-3,0) --
  (6,1);}\]
\[ \tikz[baseline=5pt,scale=.3,thick]{\draw (0,0) --(2,2); \draw
  (0,0) --(-1,1); \draw (1,3) --(2,2); \draw (1,3) --(-1,1); \draw
  (1,1) -- (0,2); \draw[very thick] (-.7,.4) -- (0,-.3) -- (.7,.4); 
\draw[very thick] (3.3,.4) -- (4,-.3) -- (4.7,.4); \node[scale=.7] at (0,1)
{$2$}; \node[scale=.7] at (1,2) {$7$};
}\qquad 
\tikz[baseline=5pt,xscale=.25, yscale=-1,thick]{\draw[wei] (0,0) -- node[below,at
  end]{$1$} (0,1); \draw[wei] (1.5,0) -- node[below, at end] {$-1$}
  (1.5,1);\draw[green!50!black] (6,0) -- (7,1); \draw[green!50!black]
  (1,0) -- (2,1) ;\draw[dashed, green!50!black] (-3,0) -- (-2,1);
  \draw[dashed, green!50!black] (2,0) -- (3,1);}
\qquad \qquad 
\tikz[baseline=5pt,scale=.3,thick]{\draw (0,0) --(1,1); \draw
  (0,0) --(-1,1); \draw (0,2) --(-1,1); \draw
  (1,1) -- (0,2); \draw[very thick] (-.7,.4) -- (0,-.3) -- (.7,.4); 
\draw[very thick] (3.3,.4) -- (4,-.3) -- (4.7,.4); \draw (4,0)
--(3,1)-- (4,2) -- (5,1)--cycle;\node[scale=.7] at (0,1)
{$1$}; \node[scale=.7] at (4,1) {$3$};
}
\qquad 
\tikz[baseline=5pt,xscale=.25, yscale=-1,thick]{\draw[wei] (0,0) -- node[below,at
  end]{$1$} (0,1); \draw[wei] (1.5,0) -- node[below, at end] {$-1$}
  (1.5,1);\draw[green!50!black] (2,0) --   (3,1);
  \draw[green!50!black] (1,0) --  (1,1);\draw[dashed, green!50!black]
  (-3,0) -- (-3,1); \draw[dashed, green!50!black] (-2,0) -- (-1,1);}\]
\[
\tikz[baseline=5pt,scale=.3,thick]{\draw (0,0) --(1,1); \draw
  (0,0) --(-1,1); \draw (0,2) --(-1,1); \draw
  (1,1) -- (0,2); \draw[very thick] (-.7,.4) -- (0,-.3) -- (.7,.4); 
\draw[very thick] (3.3,.4) -- (4,-.3) -- (4.7,.4); \draw (4,0)
--(3,1)-- (4,2) -- (5,1)--cycle;\node[scale=.7] at (0,1)
{$2$}; \node[scale=.7] at (4,1) {$7$};
}
\qquad 
\tikz[baseline=5pt,xscale=.25, yscale=-1,thick]{\draw[wei] (0,0) -- node[below,at
  end]{$1$} (0,1); \draw[wei] (1.5,0) -- node[below, at end] {$-1$}
  (1.5,1);\draw[green!50!black] (2,0) --   (7,1); \draw[green!50!black] (1,0) --  (2,1);\draw[dashed, green!50!black] (-3,0) -- (-2,1); \draw[dashed, green!50!black] (-2,0) -- (3,1);}
\qquad \qquad 
\tikz[baseline=5pt,scale=.3,thick]{\draw (0,0) --(1,1); \draw
  (0,0) --(-1,1); \draw (0,2) --(-1,1); \draw
  (1,1) -- (0,2); \draw[very thick] (-.7,.4) -- (0,-.3) -- (.7,.4); 
\draw[very thick] (3.3,.4) -- (4,-.3) -- (4.7,.4); \draw (4,0)
--(3,1)-- (4,2) -- (5,1)--cycle;\node[scale=.7] at (0,1)
{$7$}; \node[scale=.7] at (4,1) {$2$};
}
\qquad 
\tikz[baseline=5pt,xscale=.25, yscale=-1,thick]{\draw[wei] (0,0) -- node[below,at
  end]{$1$} (0,1); \draw[wei] (1.5,0) -- node[below, at end] {$-1$}
  (1.5,1);\draw[green!50!black] (2,0) --  (2,1);
  \draw[green!50!black] (1,0) to (7,1);\draw[dashed,
  green!50!black] (-3,0) to (3,1); \draw[dashed,
  green!50!black] (-2,0) -- (-2,1);}\]
\[
\tikz[baseline=5pt,scale=.3,thick]{\draw (0,0) --(1,1); \draw
  (0,0) --(-1,1); \draw (0,2) --(-1,1); \draw
  (1,1) -- (0,2); \draw[very thick] (-.7,.4) -- (0,-.3) -- (.7,.4); 
\draw[very thick] (3.3,.4) -- (4,-.3) -- (4.7,.4); \draw (4,0)
--(3,1)-- (4,2) -- (5,1)--cycle;\node[scale=.7] at (0,1)
{$8$}; \node[scale=.7] at (4,1) {$10$};
}
\qquad 
\tikz[baseline=5pt,xscale=.25, yscale=-1,thick]{\draw[wei] (0,0) -- node[below,at
  end]{$1$} (0,1); \draw[wei] (1.5,0) -- node[below, at end] {$-1$}
  (1.5,1);\draw[green!50!black] (2,0) --   (10,1); \draw[green!50!black] (1,0) --  (8,1);\draw[dashed, green!50!black] (-3,0) -- (4,1); \draw[dashed, green!50!black] (-2,0) -- (6,1);}
\qquad \qquad 
\tikz[baseline=5pt,scale=.3,thick]{\draw (0,0) --(1,1); \draw
  (0,0) --(-1,1); \draw (0,2) --(-1,1); \draw
  (1,1) -- (0,2); \draw[very thick] (-.7,.4) -- (0,-.3) -- (.7,.4); 
\draw[very thick] (3.3,.4) -- (4,-.3) -- (4.7,.4); \draw (4,0)
--(3,1)-- (4,2) -- (5,1)--cycle;\node[scale=.7] at (0,1)
{$10$}; \node[scale=.7] at (4,1) {$8$};
}
\qquad 
\tikz[baseline=5pt,xscale=.25, yscale=-1,thick]{\draw[wei] (0,0) -- node[below,at
  end]{$1$} (0,1); \draw[wei] (1.5,0) -- node[below, at end] {$-1$}
  (1.5,1);\draw[green!50!black] (2,0) --  (8,1); \draw[green!50!black] (1,0) to[in=-169,out=11] (10,1);\draw[dashed, green!50!black] (-3,0) to[in=-169,out=11] (6,1); \draw[dashed, green!50!black] (-2,0) -- (4,1);}\]
\[
\tikz[baseline=5pt,scale=.3,thick]{\draw (0,0) --(-2,2); \draw
  (0,0) --(1,1); \draw (-1,3) --(-2,2); \draw (-1,3) --(1,1); \draw
  (-1,1) -- (0,2); \draw[very thick] (-.7,.4) -- (0,-.3) -- (.7,.4); 
\draw[very thick] (-3.3,.4) -- (-4,-.3) -- (-4.7,.4); 
\node[scale=.7] at (-1,2)
{$1$}; \node[scale=.7] at (0,1) {$3$};
} \qquad 
\tikz[baseline=5pt,xscale=.25, yscale=-1,thick]{\draw[wei] (0,0) -- node[below,at
  end]{$1$} (0,1); \draw[wei] (1.5,0) -- node[below, at end] {$-1$}
  (1.5,1);\draw[green!50!black] (-1,0) --  (1,1); \draw[green!50!black] (2,0)
  to  (3,1);\draw[dashed, green!50!black] (-5,0) to (-3,1); \draw[dashed, green!50!black] (-2,0) -- (-1,1);}
\qquad\qquad\tikz[baseline=5pt,scale=.3,thick]{\draw (0,0) --(-2,2); \draw
  (0,0) --(1,1); \draw (-1,3) --(-2,2); \draw (-1,3) --(1,1); \draw
  (-1,1) -- (0,2); \draw[very thick] (-.7,.4) -- (0,-.3) -- (.7,.4); 
\draw[very thick] (-3.3,.4) -- (-4,-.3) -- (-4.7,.4); 
\node[scale=.7] at (-1,2)
{$7$}; \node[scale=.7] at (0,1) {$2$};
}
\qquad\tikz[baseline=5pt,xscale=.25, yscale=-1,thick]{\draw[wei] (0,0) -- node[below,at end]{$1$} (0,1);
  \draw[wei] (1.5,0) -- node[below, at end] {$-1$} (1.5,1);\draw[green!50!black]
  (2,0) --  (2,1); \draw[green!50!black] (-1,0)
  to[out=5,in=-175]  (7,1);\draw[dashed, green!50!black] (-5,0) to[out=5,in=-175]
  (3,1); \draw[dashed, green!50!black] (-2,0) -- (-2,1);}
\]
\[\tikz[baseline=5pt,scale=.3,thick]{\draw (0,0) --(-2,2); \draw
  (0,0) --(1,1); \draw (-1,3) --(-2,2); \draw (-1,3) --(1,1); \draw
  (-1,1) -- (0,2); \draw[very thick] (-.7,.4) -- (0,-.3) -- (.7,.4); 
\draw[very thick] (-3.3,.4) -- (-4,-.3) -- (-4.7,.4); 
\node[scale=.7] at (-1,2)
{$10$}; \node[scale=.7] at (0,1) {$8$};
}
\qquad\tikz[baseline=5pt,xscale=.25, yscale=-1,thick]{\draw[wei] (0,0) -- node[below,at end]{$1$} (0,1);
  \draw[wei] (1.5,0) -- node[below, at end] {$-1$} (1.5,1);\draw[green!50!black]
  (2,0) --  (8,1); \draw[green!50!black] (-1,0)
  to[out=5,in=-175]  (10,1);\draw[dashed, green!50!black] (-5,0) to[out=5,in=-175]
  (6,1); \draw[dashed, green!50!black] (-2,0) -- (4,1);}
\qquad \qquad \tikz[baseline=5pt,scale=.3,thick]{\draw (0,0) --(-2,2); \draw
  (0,0) --(1,1); \draw (-1,3) --(-2,2); \draw (-1,3) --(1,1); \draw
  (-1,1) -- (0,2); \draw[very thick] (-.7,.4) -- (0,-.3) -- (.7,.4); 
\draw[very thick] (-3.3,.4) -- (-4,-.3) -- (-4.7,.4); 
\node[scale=.7] at (-1,2)
{$8$}; \node[scale=.7] at (0,1) {$10$};
}
\qquad\tikz[baseline=5pt,xscale=.25, yscale=-1,thick]{\draw[wei] (0,0) -- node[below,at end]{$1$} (0,1);
  \draw[wei] (1.5,0) -- node[below, at end] {$-1$} (1.5,1);\draw[green!50!black]
  (2,0) --  (10,1); \draw[green!50!black] (-1,0)
  to (8,1);\draw[dashed, green!50!black] (-5,0) to
  (4,1); \draw[dashed, green!50!black] (-2,0) -- (6,1);}
\]
\[
\tikz[baseline=5pt,scale=.3,thick]{\draw (0,0) --(2,2); \draw
  (0,0) --(-1,1); \draw (1,3) --(2,2); \draw (1,3) --(-1,1); \draw
  (1,1) -- (0,2); \draw[very thick] (-.7,.4) -- (0,-.3) -- (.7,.4); 
\draw[very thick] (-3.3,.4) -- (-4,-.3) -- (-4.7,.4); \node[scale=.7] at (1,2)
{$7$}; \node[scale=.7] at (0,1) {$2$};
}\qquad 
\tikz[baseline=5pt,xscale=.25, yscale=-1,thick]{\draw[wei] (0,0) -- node[below,at
  end]{$1$} (0,1); \draw[wei] (1.5,0) -- node[below, at end] {$-1$}
  (1.5,1);\draw[green!50!black] (2,0) --  (2,1); \draw[green!50!black] (7,0) to  (7,1);\draw[dashed, green!50!black] (-2,0) to (-2,1); \draw[dashed, green!50!black] (3,0) -- (3,1);}
\]

\subsubsection*{Case 3: $\vartheta_1-\vartheta_2>4$}
This is essentially the same as Case 1 with components reversed; in
particular, the partial order is $p_4 > p_5 > p_3> p_1> p_2$.  
\end{example}

\begin{definition}
  For each pair $\sS,\sT$ of tableaux of the same shape, we let
  $\hC_{\sS,\sT}=\hB_{\sS}^*\hB_{\sT}$.
\end{definition}
For example:
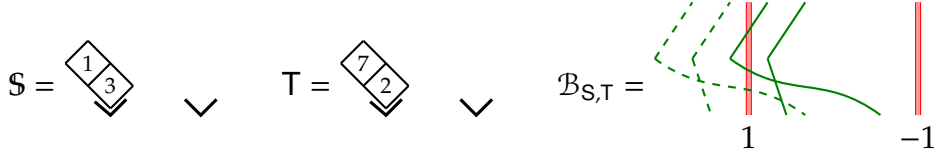
\[
\bS=\tikz[baseline=5pt,scale=.3,thick]{\draw (0,0) --(-2,2); \draw (0,0)
  --(1,1); \draw (-1,3) --(-2,2); \draw (-1,3) --(1,1); \draw (-1,1)
  -- (0,2); \draw[very thick] (-.7,.3) -- (0,-.3) -- (.7,.4);
  \draw[very thick] (3.3,.4) -- (4,-.3) -- (4.7,.4); \node[scale=.7]
  at (0,1) {$3$}; \node[scale=.7] at (-1,2) {$1$}; } \qquad \sT=\tikz[baseline=5pt,scale=.3,thick]{\draw (0,0) --(-2,2); \draw
  (0,0) --(1,1); \draw (-1,3) --(-2,2); \draw (-1,3) --(1,1); \draw
  (-1,1) -- (0,2); \draw[very thick] (-.7,.3) -- (0,-.3) -- (.7,.4); 
\draw[very thick] (3.3,.4) -- (4,-.3) -- (4.7,.4); \node[scale=.7] at (0,1)
{$2$}; \node[scale=.7] at (-1,2) {$7$};}
\qquad 
\hB_{\sS,\sT}=\tikz[baseline=7.5pt,xscale=.25,thick,yscale=1.5]{\draw[wei] (0,0) -- node[below,at
  start]{$1$} (0,1); \draw[wei] (9,0) -- (9,1) node[below, at
  start]{$-1$};\draw[green!50!black] (-1,.5) --  (1,1); \draw[green!50!black]
  (1,.5) --  (3,1);\draw[dashed, green!50!black] (-5,.5) --
  (-3,1); \draw[dashed, green!50!black] (-3,.5) -- (-1,1);  \draw[green!50!black] (-1,.5) to[out=-7,in=173]  (7,0); \draw[green!50!black] (1,.5) --   (2,0);\draw[dashed, green!50!black] (-5,.5) to[out=-7,in=173] (3,0); \draw[dashed, green!50!black] (-3,.5) -- (-2,0);} \]

This yields $4^2+2^2+3^2+2^2+1^2=34$ basis vectors,
which we will not write all of in the interest of saving trees.

\begin{lemma}\label{lem:span} The
  space $e_{D}\PC^\vartheta e_{D'}$ is spanned by the elements
  $\hC_{\sS,\sT}$ for $\sS$ a $D$-tableau and $\sT$ a $D'$-tableau.  
\end{lemma}
\begin{proof}
By Lemma \ref{lem:Morita}, every element can be written as a sum of
elements of the form $ae_{\xi}b$ for different multipartitions $\xi$.  

We need to show that $ae_{\xi}b$
can be written as a sum of the elements $\hC_{\sS,\sT}$.  Let us induct first on
$\xi$ according to
weighted dominance order, and then on the
number of crossings in the diagram below $e_{D_\xi}$ plus the number
above $e_{D_\xi}$.  Note that we can assume that any bigon which
appears is bisected by the line $y=\nicefrac 12$, and that all dots
lie on this line.  Thus,
we can associate the top half and the bottom half to two fillings of
the diagram of $\xi$, by filling each box with the top and bottom
endpoint of each strand.  

If a diagram has no crossings, it must be
ordered in Russian reading order.  There is only one way up to isotopy
of drawing this diagram (since there are no crossings of two strands
or two ghosts, and thus no triangles).

If there is a pair of entries which violate the partition condition,
that means either a strand for the box $(i,j,p)$ crosses a red strand
to its left if $\sS(1,1,p) < \vartheta_i$,  the ghost to its
left if $\sS(i,j,p) <\sS(i,j-1,p)+\kappa$, or the ghost to its
right if $\sS(i,j,p) >\sS(i+1,j,p)+\kappa$.  In either case, doing just this crossing will result in a
slice higher in dominance order, and we can isotope to assume that
this crossing is the first thing we do.  Thus, we can write this
element using those corresponding to $D$-tableaux, and elements
factoring through higher multipartitions.

Now, consider the general case.  First of all, any pair of diagrams
corresponding to the same tableau differ by shorter elements, which
lie in
the desired span by induction.  Thus, we need only show that this is
the span of some diagrams corresponding to tableaux (in our sense),
not the fixed ones $\hC_{\sS,\sT}$.

However, if $\sS$ is not a tableau, as we argued above, then either
\begin{enumerate}
\renewcommand{\labelenumi}{(\roman{enumi})}
\item $\sS(1,1,p) < \vartheta_p$, holds for some $p$,
\item or  $\sS(i,j,p) <\sS(i,j-1,p)+\kappa$, holds for some $i,j,p$,
\item  or  $\sS(i,j,p) >\sS(i+1,j,p)+\kappa$ holds for some $i,j,p$.
\end{enumerate}

Each of these inequalities implies that there is a ``bad crossing'':
\begin{enumerate}
\renewcommand{\labelenumi}{(\roman{enumi})}
\item the green strand corresponding to the box $(1,1,p)$ crosses the
  $p$th red strand,
\item or the green strand for the box $(i,j,p)$ crosses the ghost of that for $(i,j-1,p)$
\item  or the green strand for $(i,j,p)$ crosses the ghost of that for $(i+1,j)$.
\end{enumerate}
If we choose a diagram for
this filling where this ``bad crossing'' is the first that occurs,
then after isotopy, the slice after the ``bad crossing'' is higher in
dominance order than $e_{D_\xi}$.  Thus, this diagram is in the span
of diagrams factoring through a multipartition greater in weighted
dominance order.  Thus, these diagrams are in the span of
$\hC_{\sS',\sT'}$ for $\sS',\sT'$ tableaux by induction; this
completes the proof that these elements span.  
\end{proof}

\begin{lemma}\label{lem:hbasis}
The elements
  $\hC_{\sS,\sT}$ for $\sS$ a $D$-tableau and $\sT$ a $D'$-tableau of
  the same shape are a basis of $e_{D}\PC^\vartheta e_{D'}$ as a free $S$-module.
\end{lemma}

\begin{proof}
  Since we already know that these vectors span, we need only show
  that they are linearly independent.  Note that if $D,D'=D_{s,m}$ for
  $s\gg 0$, then we know that $e_{D_{s,m}}\PC^\vartheta e_{D_{s,m}}$
  is a free $S$-module of rank $m!\ell^m$.  Thus, any spanning set of
  this size must be a basis. The vectors $\hC_{\sS,\sT}$ are thus a
  basis in this case, since $D_{s,m}$-tableaux for $s\gg 0$ are in
  canonical bijection with usual standard tableaux.

  For a general choice of $D,D'$, assume that we find a linear
  combination
$\sum_{\sS,\sT} c_{\sS,\sT}\hC_{\sS,\sT}=0$.  Assume $\sS$ has shape
$\xi$ which is minimal in dominance order among those with non-zero
coefficients and that the number of
crossings in $\hC_{\sS,\sT}$ is maximal among those corresponding to
$\xi$ with non-zero
coefficients.  

We order the boxes in the diagram of $\xi$ according to the
difference\footnote{We thank Chris Bowman-Scargill for pointing out to
  us that this is the correct order on these boxes to make the proof
  below work.}
between the value of the box $(i,j,k)$ under  $\sS$ and the  $x$-value corresponding to the box in $\xi$, which is
$\vartheta_k+(i+j)\epsilon+\ck (j-i)$.  We call this statistic \[\gamma
(i,j,k)=\sS (i,j,k)-\vartheta_k-(i+j)\epsilon-\ck (j-i).\]  

Consider the diagram $\phi_{\sS}$ in $e_{D_{s,m}}\PC^\vartheta e_{D}$ which
connects $s$ at $y=1$ to the strand corresponding to the 
dot corresponding to the box with the smallest  $\gamma
(i,j,k)$ at $y=0$, connects $2s$ at $y=1$ to the strand
corresponding to the second smallest $\gamma
(i,j,k)$, and so on; if there are any ties, we choose an arbitrary way
of breaking them.  We can define a
similar diagram $\phi_{\sT}^*$ in  $e_{D'}\PC^\vartheta
e_{D_{s,m}} $.  Consider the tableaux $\sS',\sT'$ with the filling
$s,\dots, ms$ which induce the same order on boxes as
$\gamma$.  Note that if two strands are
crossed in $\hB_{\sS}^*$ or the ghost of one crosses the other, then
the crossing strand or ghost that began further right must
correspond to a strictly smaller value of $\gamma$.  Thus, the same pair of
strands or of strand and ghost
do not cross
again in $\phi_{\sS}$.  This shows that $\phi_{\sS}\hB_{\sS}^*$ is an
acceptable choice for $\hB_{\sS'}^*$.  That is, we can take the product $\phi_{\sS}\hC_{\sS,\sT}\phi^*_{\sT}$
to be the basis vector $\hC_{\sS',\sT'}$.  For every $\sS'',\sT''$ such
that $c_{\sS'',\sT''}\neq 0$, we have that
$\phi_{\sS}\hC_{\sS'',\sT''}\phi^*_{\sT}$ is a sum of diagrams
with no more crossings than $\hC_{\sS',\sT'}$, and is thus a sum of
basis vectors for higher multipartitions in dominance order, and ones
for $\xi$ with tableaux different from $\sS$ and $\sT$.  

Thus, if this linear combination is 0, it must be that the coefficient
$c_{\sS,\sT}$ is 0, since we have a basis of
$\PC^\vartheta_{D_{s,m}}$.  This is a contradiction and shows that the
vectors $\hC_{\sS,\sT}$ are linearly independent.  
\end{proof}

\begin{definition}\label{def:cellular}
  A {\bf cellular $S$-algebra} is an associative unital $S$-algebra
  $A$, free of finite rank, 
  together with a {\bf cell datum} $(\cP, M, C, *)$ such
  that \begin{itemize} \item[(1)] $\cP$ is a partially ordered set and
    $M(p)$ is a finite set for each $p \in \cP$; \item[(2)]
    $C:\dot\bigsqcup_{p \in \cP} M(p) \times M(p) \rightarrow A,
    (\sT,\sS) \mapsto C^p_{\sT,\sS}$ is an injective map whose image is a
    basis for $A$;
  \item[(3)] the map $*:A \rightarrow A$ is an algebra
    anti-automorphism such that $(C^p_{\sT,\sS})^* = C_{\sS,\sT}^p$ for
    all $p \in \cP$ and $\sS, \sT \in M(p)$; \item[(4)] if
    $p \in \cP$ and $\sS, \sT \in M(p)$ then for any $x \in
    A$ we have that $$ x C_{\sS,\sT}^p \equiv \sum_{\sS'
      \in M(p)} r_x(\sS',\sS) C_{\sS',\sT}^p \pmod{A(>
      p)}$$ where the scalar $r_x(\sS',\sS)$ is independent of
    $\sT$ and $A(> \mu)$ denotes the subspace of $A$ generated by
    $\{C_{\sS'',\sT''}^q\mid q > p, \sS'',\sT'' \in
    M(q)\}$. \end{itemize} The basis consisting of the $C^p_{\sS,\sT}$
  is then a {\bf cellular basis} of $A $.
\end{definition}
Recall that if $A$ is an algebra with cellular basis, there is a
natural {\bf cell representation} $S_\xi$ of $A$ for each $\xi\in \cP$ which is
freely generated over $S$ by symbols $c_{\sT}$ for each $\sT\in
M(\xi)$, with the action rule $xc_{\sT}=\sum_{\sS\in M(\xi)}r_x(\sS,\sT)c_{\sS}$.

Fix a collection $\mathscr{D}$ of subsets of $\R$ such that each $D\in
\mathscr{D}$ is generic.  
Let $M_{\mathscr{D}} (\xi)$ for a multipartition $\xi$ be the set of
tableaux whose entries form a set $D\in \mathscr{D}$.  Let
$\cP^\vartheta_\ell$ be the set of $\ell$-multipartitions with
$\vartheta$-weighted dominance order.  Let $*\colon
\PC^\vartheta_{\mathscr{D}}\to \PC^\vartheta_{\mathscr{D}}$ be the
anti-automorphism given by reflection in a horizontal axis.

\begin{theorem}\label{th:cellular}
  The data $(\cP^\vartheta_\ell,M_{\mathscr{D}},\hC,*)$ define a
  cellular $S$-algebra structure on $\PC^\vartheta_{\mathscr{D}}$.
\end{theorem}
\begin{proof}
  Consider the axioms of a cellular algebra, as given in Definition 
 \ref{def:cellular}.  Condition (1) is manifest.

 Condition (2), that a basis is formed by the vectors $\hC_{\sS,\sT}$
 where $\sS$ and $\sT$ range over tableaux for loadings from
 $\mathscr{D}$ of the same shape, follows from Lemma \ref{lem:hbasis}.

Condition (3) is clear from the calculation \[\hC_{\sS,\sT}^*=(\hB_{\sS}^*\hB_{\sT})^*=\hB_{\sT}^*\hB_{\sS}=\hC_{\sT,\sS},\]

Thus, we need only check the final axiom, that for all $x$, we have an
equality
\begin{equation*}\label{eq:cell1}
 x\hC_{\sS,\sT}\equiv \sum_{\sS'\in M_B(\xi)} r_x(\sS',\sS)\hC_{\sS',\sT} \tag{$\star$}
\end{equation*}
modulo the vectors associated to partitions higher in dominance
order.  The numbers $r_x(\sS',\sS)$ are just the structure
coefficients of $x^*$ acting on the basis of $S_\xi$ given by $\hB_\sS$.
By Lemma \ref{lem:span},  we have that $ x \hB_{\sS}^*$ can be
rewritten as a sum of $\hC_{\sS,\sT}$.  Furthermore, since the top of
the diagram is the loading corresponding to $D_\xi$, the vectors
$\hC_{\sS,\sT}$ appearing must correspond to multipartitions $\geq
\xi$ in weighted dominance order.  Thus, we have $x \hB_{\sS}^*\equiv\sum_{\sS'\in M_B(\xi)}
r_x(\sS',\sS)\hB_{\sS'}^*$ modulo diagrams factoring through loadings that
are higher in weighted dominance order.  Multiplying by $\hB_{\sT}$ on
the right, we see that the equation \eqref{eq:cell1}
holds.  This completes the proof.
\end{proof}
If $A$ is an finite $S$-algebra with cellular basis $(\cP, M, C, *)$,
then for any basis vector $C_{\sS,\sT}^\xi$, we have that
$(C_{\sS,\sT}^\xi)^2=a_{\sS,\sT}^\xi C_{\sS,\sT}^\xi+\cdots$ where other
terms are in higher cells. A standard lemma (see \cite[2.1(3)]{KX} for
the case of a field) shows that:
\begin{lemma}\label{lem:cell-hw}
  The category $A\mmod$ is highest weight with standard modules given
  by the cell modules if for every $\xi\in \cP$, there is
  some $\sS,\sT$ with $a_{\sS,\sT}^\xi $ a unit.
\end{lemma}

  \begin{corollary}\label{cor:PC-hw}
    The category $\PC_{\mathscr{D}^\circ_m}^\vartheta\mmod$ is highest weight. 
  \end{corollary}
  \begin{proof}
    For any multipartition $\xi$, there's a tautological tableau $\sT$
    filling each box with the $x$-value of the corresponding point in
    $D_\xi$.  Since $\hC^\xi_{\sT,\sT}=e_{D_\xi}$, this is an
    idempotent, and thus satisfies the conditions of Lemma
    \ref{lem:cell-hw}.
  \end{proof}

This cellular structure is also useful because it allows one to check
that maps are isomorphisms by means of dimension counting.  For
example, this shows:
\begin{proposition}\label{ind-res-filt}
  For any multipartition $\eta$, the restriction of a cell module
  $\operatorname{res}(S_\eta)$
  has a filtration $N_n\subset N_{n-1}\subset \cdots \subset N_1$,
  such that $N_p/N_{p+1}\cong S_{\xi_p}$, where $\xi_p$ is the multipartition given by removing from
  $\xi$ the $p$th
  removable box (read from left to right in Russian notation with
  weightings given by $\vartheta_i$).

Similarly, $\operatorname{ind}(S_\eta)$ has a filtration $M_1\subset
\cdots \subset M_q$ such that $M_p/M_{p-1}\cong S_{\xi^p}$, where $\xi^p$ is
the multipartition given by adding to 
  $\xi$ the $p$th
  addable box.
\end{proposition}
\begin{proof}
  The module $e_{\mathscr{D}}\operatorname{res}(S_\eta)$ is spanned by
  a basis $c_{\sS}$ indexed by tableaux $\sS$ where the filling is given by a set in
  $\mathscr{D}$ with $\{s\}$ for $s\gg 0$ added.  The entry $s$ must
  be in a removable box, since it is more than $|\ck|$  greater than
  any other entry.  In terms of the diagram, this means we can factor it $c_{\sS}=ab$ into
  two parts:  in the bottom part $b$, we grab the strand
  corresponding to this removable box at $y=0$, and pull it over to
  match with $x=s$; in the top part $a$, the strand at $x=s$
  remains unchanged, and we act on the other strands by the tableau
  $\sS\setminus \{s\}$, the tableau with the box labeled by $s$
  removed. 

By definition, $N_p$ is the span of the basis
  vectors where this removable box is the $p$th, or one further
  leftward.  
The relations show that when multiplying by
  a diagram that doesn't touch the strand at $s$, this strand can only
  be shortened, not lengthened.  After all, we can't create new
  crossings with this strand, only break them with the correction
  terms in (\ref{nilHecke-2}--\ref{dumb-red-triple}).  
  Thus, $N_p$ is a submodule.  

Now, assume that $\sT$ is a tableau with $\{s\}$ in the $p$th
removable box.  When we act on $c_\sT$ by
  $x\in \PC^\vartheta_{\mathscr{D}}$ (so not acting on the strand at
  $s$), we have \[xc_\sT=x\hB_{\sT\setminus \{s\}}b=\sum_{\sS\in
    M(\xi_p)}r_x(\sS,\sT\setminus\{s\})\hB_\sS b+\cdots \]The
  remaining terms all lie in $N_{p+1}$, so we have seen that the
  map sending $c_{\sT}\mapsto c_{\sT\setminus\{s\})}$ is an
  isomorphism $N_p/N_{p+1}\cong S_{\xi_p}$.  

The proof for $\operatorname{ind}(S_\eta)$ proceeds along similar
lines; now we add a new strand at the bottom of the diagram, and $O_p$
is the submodule spanned by all diagrams where the new strand goes no
further left than the $x$-value for the $p$th addable box.  
\end{proof}

Let us make a useful observation on the structure of the cell modules
of this cell structure.  Fix a set $D\subset \R$, and let $X_d$ for
$d\in D$ be the idempotent $e_D$ with a square added on the strand at
$x=d$.  This defines an action $\K[X_{d_1}^{\pm 1},\dots, X_{d_m}^{\pm 1}]$ on any
$\PC^\vartheta_{\{D\}}$-module.  One can easily check that the
symmetric polynomials in these variables are central.  Summing over
all $D\in \mathscr{D}$, we obtain a map $\zeta\colon \K[X_1^{\pm 1},\dots, X_m^{\pm 1}]^{S_m}\to
Z(\PC^\vartheta_{\mathscr{D}})$ whenever all sets in $ \mathscr{D}$
have size $m$.
 
Let $\sigma$ be the sign of $\kappa$.
\begin{lemma}\label{lem:joint-spectrum}
  The joint spectrum of $\K[X_{d_1},\dots, X_{d_m}]$ acting on
  $e_{D}S_\eta$ is the image of the map sending a $D$-tableau $\sS$ of shape
  $\eta$ to the point in $(\C^*)^D$ to the vector whose entry for
  $d\in D$ is
 $Q_pq^{\sigma(i-j)}$ where $(i,j,p)$ in the diagram of $\eta$ is the unique
 box with $\sS(i,j,p)=d$.  In particular, a symmetric Laurent
 polynomial $p(X_1,\dots, X_m)$ acts as a unipotent transformation
 times this polynomial applied to the set $\{Q_pq^{\sigma(i-j)}\}$ for
 $(i,j,p)$ ranging over the diagram of $\eta$.  
\end{lemma}
\begin{proof}
  Assume that $\kappa<0$.  Now, filter $S_\eta$ by all the span $T_g$ of the
  basis vectors with $\geq g$ crossings of strands.  The subspace
  $T_g$ is invariant under 
  $\K[X_{d_1}^{\pm 1},\dots, X_{d_m}^{\pm 1}]$, and on the associated
  graded, we have that $X_dc_{\sS}=\hB_{\sS}X_{d'}c_{\sT}$ where $\sT$ is
  the tautological tableau with filling $D_\xi$ where
  $d'=b_p+\kappa(j-i)+\epsilon(j-i)$ and $d$ fills the box $(i,j,p)$
  in $\sS$. 

If $i=j=0$ then  by (\ref{qHcost}), we have that \begin{equation*}
  \begin{tikzpicture}[very thick,baseline,scale=.7]
    \draw [wei]  (-1.5,0)  +(0,-1) -- +(0,1);
       \draw[green!50!black](-1,0)  +(0,-1) .. controls (-2.6,0) ..  +(0,1);
           \node at (0,0) {=};
    \draw [green!50!black] (2.2,0)  +(0,-1) -- node[midway,
       fill=green!50!black,inner sep=2.5pt]{}+(0,1);
       \draw[wei] (1.2,0)  +(0,-1) -- +(0,1);
 \node at (-5.2,0) {$\PQ_p$};
        \draw[green!50!black] (-3.5,0)  +(0,-1) -- +(0,1);
       \draw [wei] (-4.5,0)  +(0,-1) -- +(0,1);
           \node at (-2.7,0) {+};
  \end{tikzpicture}
\end{equation*}
The second term of the LHS is 0 in $S_\eta$, so this is a $X_d$
eigenvector with eigenvalue $\PQ_p$.  

If $j\geq i$, then the strand corresponding to   $(i,j,p)$
  is protected to the left by a ghost corresponding to $(i,j-1,p)$.
  Using the relation (\ref{qHghost-bigon2}):
\[ 
\begin{tikzpicture}[very thick,xscale=1.3,baseline=25pt,green!50!black]
 \node[black] at (.5,1) {$-$};
\draw[dashed]  (1,0) to[in=-90,out=90]  (1.5,1) to[in=-90,out=90] (1,2)
;
  \draw(1.5,0) to[in=-90,out=90] (1,1) to[in=-90,out=90] (1.5,2);
  \draw (2,0) to[in=-90,out=90]  (2.5,1) to[in=-90,out=90] (2,2);
\node[black] at (3,1) {+};
  \draw[dashed] (3.7,0) --(3.7,2) 
 ;
  \draw (4.2,0) to node[midway,fill,inner sep=3pt]{}  (4.2,2);
  \draw (4.7,0) -- (4.7,2);
\node[black] at (5.6,1) {$=\pq$};
  \draw[dashed] (6.2,0) -- (6.2,2);
  \draw (6.7,0)-- (6.7,2);
  \draw (7.2,0) -- node[midway,fill,inner sep=3pt]{} (7.2,2);
\end{tikzpicture}
\] 
Since the RHS is $\pq\cdot\PQ_pq^{j-1-i}$ times $c_{\sS}$ (in the associated
graded), and thus $X_d$ has eigenvalue $\PQ_pq^{j-i}$.  Similarly, if
$j<i$, then the strand is protected by a strand to the left of its
ghost, and a similar argument using (\ref{qHghost-bigon1}) shows that
the eigenvalue is the same in this case.
\end{proof}

\section{Comparison of Cherednik and WF Hecke algebras}
\label{sec:comp-honest-pict}

In this section, we'll prove a comparison theorem between the WF Hecke
algebra and category $\cO$ for a Cherednik algebra, using Theorem
\ref{thm:Cherednik-unique}.  Before moving to this proof, we need some
preparatory lemmata.

\subsection{Preparation}
\label{sec:preparation}

If $\pq-\zeta$ is a unit for every a root of unity $\zeta$, and for
  every  $i,j,p$, we have that
  $\PQ_i-\pq^p\PQ_j$ is a unit, then the Hecke algebra
  $H_m(\pq,\PQ_\bullet)$ is semi-simple by \cite{ArikiRQA}.  In particular:
  \begin{corollary}
    After base change to the fraction field
    $R=\C((h,z_1,\cdots,z_\ell))$, the Hecke algebra
    $H_m(\pq,\PQ_\bullet)\otimes_{\mathscr{R}}R$ is semi-simple. 
  \end{corollary}
\begin{lemma}\label{lem:Hecke-Morita}
 The isomorphism of Proposition
  \ref{prop:Hecke-quotient} induces a Morita equivalence of
  $\PC^\vartheta_{\mathscr{D}}$ and $H_m(\pq,\PQ_\bullet)$ for every
  $\mathscr{D}$ of sets of size $m$ containing $D_{s,m}$ if and only if
  the latter algebra is semi-simple.  In particular, it is an
  equivalence after the base change $-\otimes_{\mathscr{R}}R$.
\end{lemma}
\begin{proof}
  Since  $\PC^\vartheta_{\mathscr{D}}$ is cellular with the number of
  cells given by the number of $\ell$-multipartitions of $m$, this
  gives an upper bound on the number of simple modules this algebra
  can have.  Corollary \ref{cor:PC-hw} shows that for at least one
  choice of ${\mathscr{D}}$, this bound is achieved.  On the other
  hand, $H_m(\pq,\PQ_\bullet)$ has this number of non-isomorphic
  simples if and only if it is semisimple.

We know that $H_m(\pq,\PQ_\bullet)\mmod$ is a quotient category of
$\PC^\vartheta_{\mathscr{D}}\mmod$.  Since both categories are
Noetherian, this quotient functor kills no module iff it kills no
simple iff the number of simples over the two algebras coincide.  This
can only occur for all ${\mathscr{D}}$ if $H_m(\pq,\PQ_\bullet)$ is semi-simple.
\end{proof}
\excise{\begin{lemma}\label{cor:hstandard-socle}
 No simple in the socle of a cell module is killed by
 $e_{s,m}$ for $s\gg 0$.  
\end{lemma}
\begin{proof}
  Assume there is a submodule a cell module killed by $e_{s,m}$.  Then it must contain a non-trivial linear combination of
  basis vectors.  As argued in the proof of Lemma \ref{lem:hbasis},  composing these with the diagram that pulls all
  green strands to the far right while increasing the distance between
  them (preserving all ratios of distances) results in another non-trivial linear
  combination of basis vectors, which is thus non-zero.
\end{proof} }

\begin{lemma}\label{lem:-1-faithful}
  The functor $K\colon \PC_{\mathscr{D}^\circ_m}^\vartheta\mmod\to
  H_m(\pq,\PQ_\bullet)\mmod$ is faithful on standard filtered objects, that
  is, $-1$-faithful.
\end{lemma}
\begin{proof}
In the proof of Lemma \ref{lem:hbasis}, we showed that for any
non-zero element
$a\in e_DA_\xi$, we can choose $\phi_\sS\in e_{D_{s,m}}
\PC_{\mathscr{D}^\circ_m}^\vartheta e_D$ such that $\phi_\sS a\neq 0$.  That is,
no submodule of 
a cell module is killed by $e_D$.  Thus, the same is true of any
module with a cell filtration.  In particular, if $M\to N$ is a
non-zero map between cell filtered modules, then the image of this map
is not killed by $e_D$, so we have a non-zero map $e_DM\to e_DN$.
\end{proof}

As noted in the proof of Theorem \ref{thm:Cherednik-unique},
\cite[2.18]{RSVV} now implies that:
\begin{corollary}\label{cor:0-faithful}
  The functor $K$ is $0$-faithful, and thus, in particular, fully
  faithful on projectives.  
\end{corollary}

A further corollary that will be quite useful for us regards the
natural transformations of functors. For any monomials $F,F'\colon \PC_{\mathscr{D}^\circ_m}^\vartheta\to \PC_{\mathscr{D}^\circ_{m'}}^\vartheta$ in the functors
  $\mathrm{ind},\mathrm{res}$, there are functors
  $F_{\mathsf{H}},F'_{\mathsf{H}}\colon H_m(\pq,\PQ_\bullet)\mmod\to
  H_{m'}(\pq,\PQ_\bullet)\mmod$ given by the same monomials applied to
  the restriction and induction functors of these algebras.
  \begin{lemma}
    $F_{\mathsf{H}}\circ K\cong K\circ F$
  \end{lemma}
  \begin{proof}
    It's enough to prove this when $F=\mathrm{ind},\mathrm{res}$
    itself.  

The composition $\mathrm{res}_{\mathsf{H}}\circ K(M)$ is
    given by the vector space $e_{D_{s,m}}M$, where
    $H_{m-1}(\pq,\PQ_\bullet)$ acts on the left-most $m-1$ terminals.
    The functor $K\circ \mathrm{res}(M)$ goes to the same
    vector space, but first separates the right-most terminal, and
    then acts by $e_{D_{s,m-1}}$ on the remaining terminals.  Thus,
    these functors are canonically isomorphic by the identity map.

The functor $K\circ \mathrm{ind}$ is given by tensor product with the
bimodule $e_{D_{s,m+1}}\PC^\vartheta e_{\mathscr{D}'_m}$ where as
before $\mathscr{D}'_m$ is $\mathscr{D}^\circ_m$ with one point added
to each set, which we may as well take at $\{s(m+1)\}$.  On the other
hand, $\mathrm{ind}_{\mathsf{H}}\circ K$ is given by
$H_{m+1}(\pq,\PQ_\bullet)\otimes_{H_{m}(\pq,\PQ_\bullet)}e_{D_{s,m}}\PC^\vartheta
e_{\mathscr{D}^\circ_m}$.  The map 
\[H_{m+1}(\pq,\PQ_\bullet)\otimes_{H_{m}(\pq,\PQ_\bullet)}e_{D_{s,m}}\PC^\vartheta
e_{\mathscr{D}^\circ_m}\to e_{D_{s,m+1}}\PC^\vartheta
e_{\mathscr{D}'_m}\]
is given by considering an element of $H_{m+1}(\pq,\PQ_\bullet)$ as a
diagram between the slices $D_{s,m+1}$, attaching this to a diagram
in $e_{D_{s,m}}\PC^\vartheta
e_{\mathscr{D}^\circ_m}$ leaving the terminal at $s(m+1)$ free, and
then attaching a segment to the strand at $s(m+1)$ to extend to the
top of the diagram.  This map is obviously surjective.  

Since both sides deform flatly as we change $\pq$ and $\PQ_\bullet$,
it suffices to show we have an isomorphism when these values are
generic, and the corresponding algebras are semi-simple.  In this
case, the cell modules are just the irreducibles, with $K$ giving an
equivalence of categories.  In this case, Frobenius reciprocity shows
that the match of dimensions for $\mathrm{ind}$ also shows that  $K\circ
\mathrm{ind}(M)$ and $\mathrm{ind}_{\mathsf{H}}\circ K(M)$ have the
same dimension for all $M$; thus our surjective map is an isomorphism. 
  \end{proof}
\begin{corollary}\label{cor:trans-iso}
  We have a canonical isomorphism respecting composition between
  the natural transformations $\Hom(F,F')$ and $\Hom(F_{\mathsf{H}},F'_{\mathsf{H}})$.
\end{corollary}
\begin{proof}
  We have natural maps
\[A\colon\Hom(F,F')\to \Hom(K\circ F,K\circ F')\qquad
B\colon \Hom(F_{\mathsf{H}},F'_{\mathsf{H}}) \to\Hom(F_{\mathsf{H}}\circ
K,F'_{\mathsf{H}}\circ K).\]
It suffices to prove that both these maps are isomorphisms.  We can
modify the 
argument of \cite[2.4]{ShanCrystal} to show this: we know from
Proposition \ref{ind-res-filt} that induction and restriction preserve
the categories of standard filtered modules, and by 0-faithfulness,
the functor $K$ is fully faithful on the subcategory of standard
filtereds.  Thus any element of the kernel of $A$ must kill all
standard filtered modules and be 0; on the other hand, the
surjectivity follows from fullness, since any object has a
representation by projectives, which are standard filtered.

The map $B$ is injective because $K$ is a quotient functor.  On other
hand, $0$-faithfulness implies that any projective has a
copresentation by modules induced from $H_{m}(\pq,\PQ_\bullet)$.
Thus, the action of any natural transformation a projective is
determined by its action on an induction.  This shows the surjectivity
of $B$.
\end{proof}

Note that this shows that any property of $\mathrm{ind},\mathrm{res}$
that can be phrased in terms of natural transformations can be
transfered from the analogous properties of the Hecke algebra.  

Recall that functors $\operatorname{ind}$ and
$\operatorname{res}$ have a natural action of $H_1(\mathbf{q})\cong \C[X^{\pm
  1}]$.   The generalized $u$-eigenspace of the natural transformation $X$
defines a subfunctor of $\eE_u\subset \operatorname{ind}$ and
$\eF_u\subset \operatorname{res}$, usually called $u$-induction or $u$-restriction. Since we are only considering the
action on finite dimensional modules, these functors are in fact sums 
\[ \operatorname{ind}\cong \oplus_{u\in U} \eF_u\qquad
\operatorname{res}\cong \oplus_{u\in U} \eE_u.\]  Corollary
\ref{cor:trans-iso} similarly shows that any statement involving these
functors phrased in terms of natural transformations can be
transfered from the $u$-induction and $u$-restriction functors of the
Hecke algebra.  In particular:
\begin{corollary}\label{biadjoint}
  The functors $\mathrm{ind},\mathrm{res}$ are biadjoint and commute
  with duality (this holds in the Hecke case by
  \cite[Lem. 2.6]{ShanCrystal}).  If $e\neq 1$, then the functors $\eF_u,\eE_u$ induce a categorical
  $\gu$-action (following the argument of \cite[Lem. 5.1]{ShanCrystal}).  
\end{corollary}

\begin{lemma}\label{q=-1}
  If $q=-1$, then $H_2(T+1)$ is
  in the image of the functor $K\colon \PC^{\Bs}_2\mmod\to H_2(\pq,\PQ_\bullet)\mmod$.  
\end{lemma} 
\begin{proof}
  We claim that if $D$ is the set $\{s,s+\nicefrac{\kappa}2\}$ for
  $s\gg 0$ then $e_{d,s}\PC^\Bs_d e_D$ is isomorphic to 2 copies of
  this module.  The module $e_{d,s}\PC^\Bs_d e_D$ is generated by the
  two elements
  \begin{equation}
  \tikz[baseline, very thick,xscale=1.7,yscale=1.2]{\draw[wei]
    (0,-.5) -- node[at start,below]{$\PQ_1$}(0,.5);\node at
    (.5,0){$\cdots$};\draw[wei] (1,-.5) -- node[at
    start,below]{$\PQ_\ell$}(1,.5); \draw[green!50!black] (1.75,-.5)
    -- (1.25,.5); \draw[green!50!black,dashed] (2.25,-.5) --
    (1.75,.5); \draw[green!50!black] (2,-.5) -- (2.5,.5);
    \draw[green!50!black,dashed] (2.5,-.5) -- (3,.5); }\qquad \qquad
  \tikz[baseline, very thick,xscale=2,yscale=1.2]{\draw[wei] (0,-.5)
    -- node[at start,below]{$\PQ_1$}(0,.5);\node at
    (.5,0){$\cdots$};\draw[wei] (1,-.5) -- node[at
    start,below]{$\PQ_\ell$}(1,.5); \draw[green!50!black] (1.75,-.5)
    -- (2.5,.5); \draw[green!50!black,dashed] (2.25,-.5) -- (3,.5);
    \draw[green!50!black] (2,-.5) -- (1.25,.5);
    \draw[green!50!black,dashed] (2.5,-.5) -- (1.75,.5); }\label{eq:3}
\end{equation}
Both of
  these elements are killed by $T+1$, and thus give maps from
  $H_2(T+1)\cong H_2/H_2(T+1)\to e_{d,s}\PC^\Bs_d e_D$.  The dimension of this module is
  $\ell^2$.  On the other hand, the dimension of $e_{d,s}\PC^\Bs_d
  e_D$ is the number of pairs of tableaux of the same shape on
  $\ell$-multipartitions of 2, one with filling $s,2s$ and the other
  with filling $s,s+\nicefrac{\kappa}/2$.

  Each of $\ell(\ell-1)/2$ different $\ell$-multipartitions
  consisting of 2 different 1 box diagrams give $4$ basis vectors, so
  together they contribute $2\ell(\ell-1)$ basis vectors.  For a
  multipartition with a single 2-box diagram, we can only have a
  tableau with filling $s,s+\nicefrac{\kappa}2$ on $(2)$ if $\kappa<0$
  or $(1,1)$ if $\kappa>0$.  In either case, the $\ell$ ways of
  placing this in different components contribute $2$ basis vectors
  each, since either filling with $s,s+\nicefrac{\kappa}2$ gives a
  tableau, but only one filling with $s,2s$ does.  Thus, we have
  dimension $2\ell(\ell-1)+2\ell=2\ell^2$.  This shows that the map
  from $H_2(T+1)^{\oplus 2}$ is an isomorphism.

  Thus, either of the elements shown in \eqref{eq:3} generate a
  summand of $e_{d,s}\PC^\Bs_d e_D$ whose image under $K$ is $H_2(T+1)$.
\end{proof}

\subsection{A comparison theorem}
\label{sec:comparison-theorem}

Now, we'll consider the case where $\K=\C$, and $S$ is one of $\C,\mathscr{R}=\C[[h,z_1,\dots, z_\ell]]$ or
$R=\C((h,z_1,\dots, z_\ell))$.  As before, we have parameters
$\kappa,s_1,\dots, s_\ell\in \C$ for the rational Cherednik algebra,
and we consider \[\sk=k+\frac{h}{2\pi i}\qquad\sss_j=
(ks_j-\frac{z_j}{2\pi i})/\sk\]\[ q=\exp(2\pi i k)\qquad
Q_i=\exp(2\pi i k s_i) \qquad \pq=qe^h\qquad \PQ_i=Q_ie^{-z_i}.\]

We let $\kappa=\Re(k)$ and $\vartheta_i=\Re(ks_i)-i/\ell$, and let
$\PC^{\Bs}_d:=\PC^\vartheta_{\mathscr{D}^\circ_d}$ denote the
WF Hecke algebra over $\mathscr{R}$ defined above
attached to the collection $\mathscr{D}=\{D_\xi\}$ for $\xi$ all
$\ell$-multipartitions of $d$.   

\begin{theorem}\label{theorem:diagrammatic-Cherednik}
  We have an equivalence of categories $\bO^{\Bs}_d\cong
  \PC^{\Bs}_d\mmod$ intertwining the functor $\KZ$ with the quotient
  functor $ M\mapsto e_{D_{s,d}}M$.
\end{theorem}
\begin{proof}
  Of course, we'll use Theorem \ref{thm:Cherednik-unique}.  Let's
  confirm the conditions of this theorem:
  \begin{enumerate}
  \item we have an isomorphism $\PC^{\Bs}_0\cong \mathscr{R}$.
  \item the highest weight structure follows from Lemma
    \ref{lem:cell-hw}.
  \item the desired induction functors are induced by the map of Lemma
    \ref{lem:horizontal-Hecke}; extension of scalars always preserves
    projectives.
\item The image $\operatorname{ind}(\mathcal{R},H_q(\pq,\PQ_\bullet))$
  is the projective $\PC^\Bs_d e_{s,d}$.  Thus, the functor $K$ is
  just $ M\mapsto e_{D_{s,d}}M$.  This is clearly a quotient functor,
  and becomes an equivalence after base change by Lemma
  \ref{lem:Hecke-Morita}.
\item The desired duality is just $M^\star:=\Hom(M,\mathscr{R})$, which is
  naturally a $(\PC^\Bs_d)^{\operatorname{op}}$-module.  We use the
  anti-automorphism $*$ to make this a $ \PC^\Bs_d$-module again.
We have $eM^\star\cong (e^*M)^\star$, so the commutation of this
duality with the analogous one on the Hecke algebra follows from the
fact that $e_{s,d}^*=e_{s,d}$.  The duality on the Hecke algebra
corresponds to the anti-automorphism sending $T_i\mapsto T_i$ and
$X_i\mapsto X_i$.  
\item in both cases, the order induced on simples is a coarsening of
  $c$-function ordering.  These match as calculated in Proposition
  \ref{prop:c-functions}.
\item Finally, we need that if $q=-1$, then $H_2(T+1)$ is
  in the image.  This is precisely Lemma \ref{q=-1}. 
  \end{enumerate}
This confirms all the hypotheses, and thus shows that we have an equivalence.
\end{proof}
Let $\bar{\PC}^{\Bs}_d:=\C\otimes_{\mathscr{R}}\PC^\Bs_d$:
\begin{corollary}
  The category $\cO^{\Bs}_d$ over $\mathsf{H}$ is equivalent to the
  category $\bar{\PC}^{\Bs}_d\mmod$.
\end{corollary}

While this equivalence is somewhat abstract, at least it gives us a
concrete description of the image of projectives under the $\KZ$ functor.  This image is
generated as an additive category by the $H_d(q,Q_\bullet)$-modules
$e_{s,d}\bar{\PC}^{\Bs}_de_{\xi}$ for different partitions $\xi$.
This is an explicit cell-filtered module over $H_d(q,Q_\bullet)$, with
a basis we can compute with, though of course, not without some
effort.

\subsection{Cyclotomic $q$-Schur algebras}
\label{sec:cyclotomic-q-schur}

 This comparison theorem can also be applied to 
cyclotomic $q$-Schur algebras.  The cyclotomic $q$-Schur algebra
$\mathscr{S}_d(\pq,\PQ^\bullet)$ over the ring $\mathscr{R}$ attached to the
data $(\pq,\PQ^\bullet)$ was defined by Dipper, James and Mathas
\cite[6.1]{DJM} (for the set $\Lambda$, we will use all multi-compositions
with $d$ parts).    One can easily confirm that the category of
representations of this algebra satisfies all the properties of $\cQ^\Bs_d$
in Theorem \ref{thm:Cherednik-unique}, except that the order does not necessarily
have a common refinement with the ordering on the simples of the
Cherednik algebra.  Thus, Theorem \ref{thm:Cherednik-unique} shows
that
\begin{corollary}
  If the $c$-function order for $k,\Bs$ on charged $\ell$ partitions
  refines the usual dominance order on $\ell$-multipartitions of
  $d\leq D$, then we have an equivalence of highest weight categories
  $\cO^\Bs_d\cong \mathscr{S}_d(\pq,\PQ^\bullet)\mmod\cong
  \bar{\PC}^\bS_d\mmod $ for all $d\leq D$.
\end{corollary}
This condition will necessarily hold whenever $D|\ck|
<\vartheta_{j+1}-\vartheta_j$ for all $i,j$, but there is no uniform
choice of $\Bs$ where we have this Morita equivalence for all $D$;
eventually, the orders will start to differ.
Note that in \cite[\ref{n-cqs-morita}]{WebBKnote}, we showed the latter Morita
equivalence directly when the inequality above holds.

\subsection{Change-of-charge functors: Hecke case}
\label{sec:derived-equivalences}

In the algebra $\PC^\vartheta$, we have required that the red lines
are vertical, that is, the quantities $\vartheta_i$, as well as
$\kappa$ are fixed.  However, a natural and important question is how
these algebras compare if these quantities are changed.  We can relate
them using natural bimodules between such pairs of algebras.  

Given different choices $\vartheta_i,\kappa$ and
$\vartheta'_i,\kappa'$ of these parameters, we can define a bimodule
over $\PC^\vartheta$ and $\PC^{\vartheta'}$ (we'll leave the use of
$\kappa$ and $\kappa'$ in the two algebras implicit).
\begin{definition}\label{WF-definition}
  We let a {\bf WF $\vartheta\operatorname{-}\vartheta'$
    diagram} be a diagram like the a WF Hecke
  diagram with
  \begin{itemize}
  \item $\ell$ red line segments which go from $(\vartheta_i',0)$ to
    $(\vartheta_i,1)$.
  \item green strands, which as usual project diffeomorphically to
    $[0,1]$ on the $y$-axis and can carry squares.  Each strand has a
    ghost whose distance from the strand now varies with the value of
    $y$: it is $y\kappa+(1-y)\kappa'$ units to the right of the
    strand.
  \end{itemize}
These diagrams must satisfy the genericity conditions from before,
though these must be interpreted carefully: if two red strands cross,
or a strand crosses its own ghost, this is not a ``true crossing'' and
it can be ignored for purposes of genericity.  In particular, 
we can isotope another strand through it without issues.
\end{definition}
Here is one example of a WF $\vartheta\operatorname{-}\vartheta'$
diagram, with $\kappa<0$ and $\kappa'>0$.  
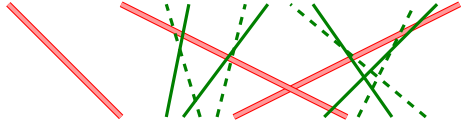
\[ \begin{tikzpicture}[very thick,xscale=1.5,yscale=1.5,baseline,green!50!black]
  \draw[wei] (-1,0) -- (-2,1); \draw[wei] (0,0) -- (2,1); \draw[wei]
  (1,0) -- (-1,1);
\draw (1.4, 0) to (.7,1) ;
\draw[dashed] (1.7, 0) to  (.5,1);
\draw (.8, 0) to(1.8,1);  
\draw[dashed] (1.1, 0) to (1.6,1);
\draw (-.6, 0) to (-.4,1);  
\draw[dashed] (-.3, 0) to (-.6,1);
\draw (-.45, 0) to (.3,1);  
\draw[dashed] (-.15, 0) to (.1,1);
    \end{tikzpicture}\]

    \begin{definition}
      Let $\EuScript{K}^{\vartheta,\vartheta'}$ be the $\K$-span of
      the WF $\vartheta\operatorname{-}\vartheta'$ diagrams modulo the
      relations (\ref{nilHecke-2}--\ref{dumb-red-triple}) and the
      steadying relation that a diagram is 0 if at some fixed
      $y$-value, the strands can be divided into two groups with all
      strands and ghosts of the left hand group with $x$-values $<a$
      and all strands and ghosts of the right hand group, which
      contains all red strands, with $x$-values $>a$, for some real
      number $a$.
    \end{definition}

\nc{\bK}{\EuScript{K}}
\begin{proposition}
  The space $\EuScript{K}^{\vartheta,\vartheta'}$ is naturally a
  $\PC^\vartheta \operatorname{-}\PC^{\vartheta'}$-bimodule.  
\end{proposition}
\begin{proof}
  We wish to stack a diagram $a$ from $\PC^\vartheta$ on top of one
  $b$ from
  $\EuScript{K}^{\vartheta,\vartheta'}$.  This will not literally be
  the case, since we require a diagram from
  $\EuScript{K}^{\vartheta,\vartheta'}$ to have its red lines to be
  straight, and the composition will have a kink where the diagrams
  join, and similarly a kink in each ghost at this point.
However, we can apply a combination of isotopies and the relations to
get rid of this kink.  There is some $\epsilon$ such that replacing
the red strands in $a$ by ones going from $(\epsilon\vartheta_i'+(1-\epsilon)\vartheta_i,0)$ to
    $(\vartheta_i,1)$, and placing the ghosts 
    $\kappa+(1-y)\epsilon(\kappa'-\kappa)$ units right of each strand
    results in an isotopic diagram.   We can further choose this
    $\epsilon$ so that in the diagram $b$, replacing the red strands
    by ones going from $(\vartheta_i',0)$ to
    $(\epsilon\vartheta_i'+(1-\epsilon)\vartheta_i,1)$ and placing the
    ghosts 
    $\kappa'+y(1-\epsilon)(\kappa-\kappa')$ units right of each strand
    results in an isotopic diagram as well.  Now, we can stack these
    diagrams, with $a$ scaled to fit between $y=1-\epsilon$ and $y=1$,
    and $b$ to fit between $y=0$ and $y=1-\epsilon$.  
\end{proof}

In this bimodule, we can construct analogues of the elements $\hCST$,
which we will also denote $\hCST$ by abuse of notation (the original
elements $\hCST$ will be a special case of these where
$\vartheta=\vartheta'$).  Unlike the algebra $\PC^\vartheta$, the
construction of these requires
breaking the symmetry between top and bottom of the diagram.  Thus, we can
make one choice to obtain a cellular basis of
$\bK^{\vartheta,\vartheta'}$ as a left module and another to obtain a
cellular basis as a right module.

Let us first describe the basis which is cellular for the right module
structure.  Let $\hD_\sS$ be the element of the bimodule
$\bK^{\vartheta,\vartheta'}$ defined analogously with $\hB_{\sS}$. Its
top is given by the set $\hD_\eta$ (for the weighting $\vartheta$).  Its
bottom is given by the entries of $\sS$, with each entry giving the
$x$-coordinate of a strand.  The diagram proceeds by connecting the points
in the loading associated to the same box in the top and bottom, while
introducing the smallest number of crossings.  As usual, this diagram
is not unique; we choose any such diagram and fix it from now on.

\begin{definition}
  The right cellular basis for $e_{\Bi}\bK^{\vartheta,\vartheta'}e_{\Bj}$ is given
  by $\hD_\sS ^*\hB_{\sT}$ for $\sS$ an $\Bi$-tableau for some
loading $\Bi$ and the weighting $\vartheta$ (upon which the definition of
$\Bi$-tableau depends),  and $\sT$ a $\Bj$-tableau for some loading
$\Bj$ and the weighting  $\vartheta'$.

The left cellular basis for
$e_{\Bj}\bK^{\vartheta',\vartheta}e_{\Bi}$ is given by the
reflections of these vectors, that is by $\hB_{\sT}^*\hD_{\sS}$.  
\end{definition}

\begin{example}
  Let us illustrate with a small example.  
  Consider $\PC^\vartheta$ with two red lines, both labeled with 1, and
  a single green line.    Let $\vartheta=(1,-1)$ and
  $\vartheta'=(-1,1)$.  Thus, in each diagram, we have a red cross.  A
  loading is determined by the position of its single dot.  Let $e_0$
  be the loading where it is at $y=0$ and $e_2$ that where it is at
  $y=2$.  Each basis vector is attached to a pair of Young diagrams with one box
  total, so one is a single box and the other empty.  A tableau on
  such a diagram is a single number, which is greater than the
  associated value of $\vartheta$ or $\vartheta'$.

Thus, if the box is in the first component, its filling in $\sS$ must
be $>1$ and in $\sT$ must be  $>-1$; if the box is in the second
component, the filling in $\sS$ must
be $>-1$ and in $\sT$ must be  $>1$.
Thus, $e_{0}\bra^{\vartheta',\vartheta}e_{0}$ is the 0 space, since 0
cannot give a tableau for both $\vartheta$ and $\vartheta'$ for either
diagram.  On the other hand, $e_{2}\bra^{\vartheta',\vartheta}e_{0}$
and $e_{0}\bra^{\vartheta',\vartheta}e_{2}$ are both 1-dimensional,
with the only basis vector associated to $((1),\emptyset)$ in the
first case, and to $(\emptyset ,(1))$ in the second.  Both these
diagrams have a tableau with filling with all 2's, so
$e_{2}\bra^{\vartheta',\vartheta}e_{2}$ is 2-dimensional.  For the
right basis, these vectors are given by:
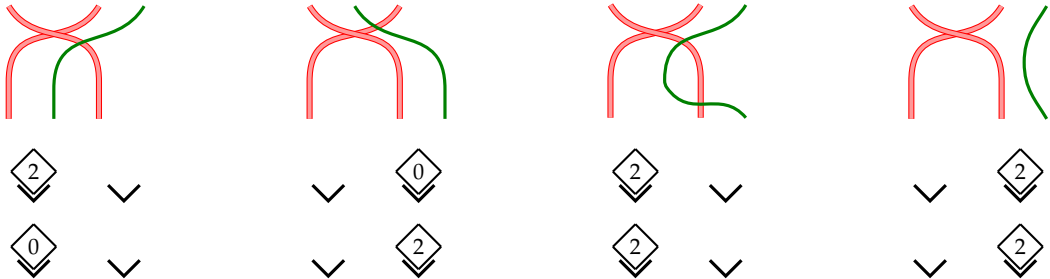
\[
\begin{tikzpicture}[xscale=2]
  \node at (-3,0){
    \begin{tikzpicture}[xscale=.6]
      \draw[wei] (-1,-.5) to[out=90,in=-90] (-1,0) to[out=90,in=-135]
      (1,1); \draw[wei] (1,-.5) to[out=90,in=-90] (1,0)
      to[out=90,in=-45] (-1,1); \draw[very thick,green!50!black] (0,-.5) to[out=90,in=-90] (0,-.1)
      to[out=90,in=-135] (2,1);
    \end{tikzpicture}
};
\node at (-3,-1.5){\tikz[baseline=5pt,scale=.3,thick]{\draw (0,0) --(1,1); \draw
  (0,0) --(-1,1); \draw (0,2) --(-1,1); \draw
  (1,1) -- (0,2); \draw[very thick] (-.7,.4) -- (0,-.3) -- (.7,.4); \draw[very thick] (3.3,.4) -- (4,-.3) -- (4.7,.4); 
\node[scale=.7] at (0,1)
{$2$}; 
}};
\node at (-3,-2.5){\tikz[baseline=5pt,scale=.3,thick]{\draw (0,0) --(1,1); \draw
  (0,0) --(-1,1); \draw (0,2) --(-1,1); \draw
  (1,1) -- (0,2); \draw[very thick] (-.7,.4) -- (0,-.3) -- (.7,.4); \draw[very thick] (3.3,.4) -- (4,-.3) -- (4.7,.4); 
\node[scale=.7] at (0,1)
{$0$}; 
}};
  \node at (-1,0){
    \begin{tikzpicture}[xscale=.6]
      \draw[wei] (-1,-.5) to[out=90,in=-90] (-1,0) to[out=90,in=-135]
      (1,1); \draw[wei] (1,-.5) to[out=90,in=-90] (1,0)
      to[out=90,in=-45] (-1,1); \draw[very thick, green!50!black] (2,-.5) to[out=90,in=-90] (2,-.1)
      to[out=90,in=-45] (0,1);
    \end{tikzpicture}
};
\node at (-1,-1.5){\tikz[baseline=5pt,scale=.3,thick]{\draw (4,0) --(5,1); \draw
  (4,0) --(3,1); \draw (4,2) --(3,1); \draw
  (5,1) -- (4,2); \draw[very thick] (-.7,.4) -- (0,-.3) -- (.7,.4); \draw[very thick] (3.3,.4) -- (4,-.3) -- (4.7,.4); 
\node[scale=.7] at (4,1)
{$0$}; 
}};
\node at (-1,-2.5){\tikz[baseline=5pt,scale=.3,thick]{\draw (4,0) --(5,1); \draw
  (4,0) --(3,1); \draw (4,2) --(3,1); \draw
  (5,1) -- (4,2);\draw[very thick] (-.7,.4) -- (0,-.3) -- (.7,.4); \draw[very thick] (3.3,.4) -- (4,-.3) -- (4.7,.4); 
\node[scale=.7] at (4,1)
{$2$}; 
}};
  \node at (1,0){
    \begin{tikzpicture}[xscale=.6]
      \draw[wei] (-1,-.5) to[out=90,in=-90] (-1,0) to[out=90,in=-135]
      (1,1); \draw[wei] (1,-.5) to[out=90,in=-90] (1,0)
      to[out=90,in=-45] (-1,1);  \draw[very thick, green!50!black] (2,-.5) to[out=145,in=-40] (.25,-.1) to[out=140,in=-90] (.2,0)
      to[out=90,in=-135] (2,1);
    \end{tikzpicture}
};
  \node at (1,-1.5){\tikz[baseline=5pt,scale=.3,thick]{\draw (0,0) --(1,1); \draw
  (0,0) --(-1,1); \draw (0,2) --(-1,1); \draw
  (1,1) -- (0,2); \draw[very thick] (-.7,.4) -- (0,-.3) -- (.7,.4); \draw[very thick] (3.3,.4) -- (4,-.3) -- (4.7,.4); 
\node[scale=.7] at (0,1)
{$2$}; 
}};
\node at (1,-2.5){\tikz[baseline=5pt,scale=.3,thick]{\draw (0,0) --(1,1); \draw
  (0,0) --(-1,1); \draw (0,2) --(-1,1); \draw
  (1,1) -- (0,2); \draw[very thick] (-.7,.4) -- (0,-.3) -- (.7,.4); \draw[very thick] (3.3,.4) -- (4,-.3) -- (4.7,.4); 
\node[scale=.7] at (0,1)
{$2$}; 
}};
  \node at (3,0){
    \begin{tikzpicture}[xscale=.6]
      \draw[wei] (-1,-.5) to[out=90,in=-90] (-1,0) to[out=90,in=-135]
      (1,1); \draw[wei] (1,-.5) to[out=90,in=-90] (1,0)
      to[out=90,in=-45] (-1,1); \draw[very thick, green!50!black] (2,-.5) to[out=135,in=-90] (1.5,.25)
      to[out=90,in=-135] (2,1);
    \end{tikzpicture}
};
\node at (3,-1.5){\tikz[baseline=5pt,scale=.3,thick]{\draw (4,0) --(5,1); \draw
  (4,0) --(3,1); \draw (4,2) --(3,1); \draw
  (5,1) -- (4,2); \draw[very thick] (-.7,.4) -- (0,-.3) -- (.7,.4); \draw[very thick] (3.3,.4) -- (4,-.3) -- (4.7,.4); 
\node[scale=.7] at (4,1)
{$2$}; 
}};
\node at (3,-2.5){\tikz[baseline=5pt,scale=.3,thick]{\draw (4,0) --(5,1); \draw
  (4,0) --(3,1); \draw (4,2) --(3,1); \draw
  (5,1) -- (4,2);\draw[very thick] (-.7,.4) -- (0,-.3) -- (.7,.4); \draw[very thick] (3.3,.4) -- (4,-.3) -- (4.7,.4); 
\node[scale=.7] at (4,1)
{$2$}; 
}};
\end{tikzpicture}
\]
Note that we have drawn these in a way that the factorization into two
diagrams is clear, but according the definition, we should really
perform isotopies of these so that the red lines are straight. 
\end{example}

\begin{lemma}\label{lem:hbim-basis}
The vectors $\hD_\sS ^*\hB_{\sT}$ are a basis for the  bimodule
$\bK^{\vartheta,\vartheta'}$.  Furthermore, the sum of vectors
attached to partitions $\leq \xi$ in $\vartheta'$-weighted order is a
right submodule. In particular, as a
right module, $\bK^{\vartheta,\vartheta'}$ is standard filtered.

Similarly, the left cellular basis shows that the bimodule
$\bK^{\vartheta,\vartheta'}$ is standard filtered as a left module.
\end{lemma}

\begin{proof}[Proof of Lemma \ref{lem:hbim-basis}]
  First, we wish to show that these elements span.  By the Morita
  equivalence of Lemma \ref{lem:Morita}, the bimodule
  $\bK^{\vartheta,\vartheta'}$ is spanned by elements of the form
  $ae_\xi b$ where $a \in \bK^{\vartheta,\vartheta'}$, $\xi$ a multipartition and
  $b\in \PC^{\vartheta'}$.   We prove by induction that
  $ae_\xi b$ lies in the span of the vectors  $\hD_\sS ^* \hB_{\sT}$
  for $\sS,\sT$ of shape $\geq \xi$ in $\vartheta'$-weighted dominance
  order.   

Without loss of generality, we can assume that
  $b$ is one of the vectors of our cellular basis of Theorem
  \ref{th:cellular}.  If the associated cell is not $\xi$, then $b$
  factors through $e_\nu$ for $\nu>\xi$, and the result follows by
  induction.  If it is $\xi$, then we must have $b=\hB_{\sT}$ for
  some $\sT$.  

We can also assume that $a$ is a single diagram, with no bigons
between pairs of strands or strands and ghosts.  The slice at $y=0$ of
$b$ is precisely $D_\xi$, and we can use this identification to match
the strands with boxes of the diagram of this multipartition.
Now, we can apply the argument of Lemma \ref{lem:hbasis} to $a$: we
can fill the diagram of $\xi$ by the $x$-value at $y=1$ of the strand
corresponding to that box at $y=0$.  Let $D$ be the set given by the
slice at $y=1$. If this filling isn't a $D$-tableau
for the weighting $\vartheta$, then the corresponding diagram must
have a ``bad crossing'' in the same sense of the proof of Lemma
\ref{lem:hbasis}, which we can slide to the bottom of the diagram,
showing it factors through $e_\nu$ for $\nu>\xi$ in
$\vartheta'$-dominance order.  Thus, we can assume that this filling
is a $D$-tableau. As usual, any two diagrams for the same tableau
differ by diagrams with fewer crossings, so by induction, choosing one
diagram for each tableau suffices to span.

Thus, we need only show that these are linearly independent.  As
before, we can reduce to the case where $D=D'=D_{s,m}$ for $s\gg 0$ by
Lemma~\ref{lem:-1-faithful}; in
this case, the bimodule $ e_{D}\bK^{\vartheta,\vartheta'}e_{D}$ is
precisely the same as $e_{D}\PC^\vartheta e_{D}$.  We can identify this
space with the image of the corresponding idempotents acting on the
cyclotomic Hecke algebra $\PC^\la$, so it has the correct dimension by
Lemma \ref{lem:hbasis}.
\end{proof}
As with any cellularly filtered module, we can study the
multiplicities of cell modules $S_\xi$ for $\PC^{\vartheta}$ in $\bK^{\vartheta,\vartheta'}e_{D'}$.

\begin{corollary}\label{cor:B-multiplicities}
  We have an equality of multiplicities 
  \[[\bK^{\vartheta,\vartheta'}e_{D}
  :S_\xi]=[\PC^{\vartheta'}
  e_{D}:S_\xi'] .\] 
\end{corollary}

We'll prove later (Lemma \ref{lem:c-o-c-equiv}) that derived tensor
product with this bimodule is an equivalence of derived categories,
and in fact, that these can be organized into an action of the affine
braid group.

One way to think about the significance of a weighting is that it
induces a total order on the columns of the diagram of $\xi$ (remember, we
are always using Russian notation; in the usual notation for
partitions, these would be diagonals).  Let $>_{\vartheta}$ be this
order.
\begin{definition}
  For total orders on a finite set, we say $>'$ is {\bf between} $>$ and
  $>''$ if there is no pair of elements $a, b$ such that $a>b, a>''b$
  and $b>'a$.

  We say that a weighting $\vartheta'$ is {\bf between} $\vartheta$ and
  $\vartheta''$ if for any multi-partition $\xi$, the induced
 order $>_{\vartheta'}$ on columns of $\xi$ is between
 $>_{\vartheta}$ and $>_{\vartheta''}$
\end{definition}
\begin{lemma}
  If $\vartheta'$ is between $\vartheta$ and
  $\vartheta''$, then we have that \[\bK^{\vartheta,\vartheta'}
  \Lotimes \bK^{\vartheta',\vartheta''}\cong \bK^{\vartheta,\vartheta''}.\]
\end{lemma}
\begin{proof}
There is an obvious map $\bK^{\vartheta,\vartheta'}
  \Lotimes \bK^{\vartheta',\vartheta''}\to
  \bK^{\vartheta,\vartheta''}$ given by stacking the diagrams.  First, we need to show that this map
  is surjective if
  $\vartheta'$ is between $\vartheta$ to
  $\vartheta''$.  This follows since after applying an isotopy, any
  diagram in $\bK^{\vartheta,\vartheta''}$ can have its red strands
  meet with $\vartheta'$ at $y=\nicefrac{1}{2}$.  Thus, slicing this
  diagram in half, we obtain diagrams from
  $\bK^{\vartheta,\vartheta'}$ and $\bK^{\vartheta',\vartheta''}$
  which will hit this one under the stacking map.

  Note that \[\dot{S}_\xi\Lotimes S_{\xi'}=
  \begin{cases}
    \K & \xi =\xi'\\
    0 & \xi\neq \xi'\\
  \end{cases}.\]
Furthermore, the multiplicity of $\dot{S}_\xi$ in
$\bK^{\vartheta,\vartheta'}$ as a right module is the number of $D$-tableau for $\vartheta$ of shape
  $\xi$ and the multiplicity of $S_\xi$ as a left module is the number
  of $D'$-tableau for $\vartheta''$ of shape $\xi$. 
Thus, the dimension of $e_D \bK^{\vartheta,\vartheta'}
  \Lotimes \bK^{\vartheta',\vartheta''}e_{D'}$ is exactly the number
  of pairs of these with the same shape, which is the dimension of  $e_D \bK^{\vartheta,\vartheta''}e_{D'}$.  Since a surjective map
  between finite dimensional vector spaces of the same dimension is an
  isomorphism, we are done.
\end{proof}

  \section{Gradings and weighted KLR algebras}
  \label{sec:basic}

One great advantage of having a concrete presentation of the category
$\cO$ for a Cherednik
algebra is that it allows us to think in a straightforward way about
graded lifts of this category: they simply correspond to gradings on
this algebra.  The presentation we gave before is not homogenous for
an obvious grading, but we can give a different presentation which is,
in the spirit of Brundan and Kleshchev's approach to gradings on Hecke
algebras \cite{BKKL}.  

  \subsection{Weighted KLR algebras}
  \label{sec:weighted-algebras}

As before, we choose $(r_1,\dots,r_\ell)\in (\C/\Z)^\ell$ and a scalar
$k\in \C$.  Given this data, we have a graph $U\subset \C/\Z$ and
associated Lie algebra $\gu$, as defined
in Section \ref{sec:comb-prel-}.  We have an associated highest weight $\la=\sum_i\omega_{r_i}$ of $\gu$
of level $\ell$.  Attached to this choice, we have a {\bf Crawley-Boevey
  quiver} $U_\la$, as defined in \cite[3.1]{WebwKLR}.  This adds a single
vertex, which in this paper we index by $\infty$, and $\la^u=\al_u^
\vee(\la)$ new
edges from $u$ to $\infty$.  We let $\Omega_\la$ be the edge set of
this new graph.  We'll often refer to the edges of the original cycle
as {\bf old} and those we have added to the Crawley-Boevey vertex as
{\bf new}.

\excise{
\begin{figure}[tbh]
  \centering
  \begin{tikzpicture}[very thick, ->, scale=2]
    \node (a) at (0,-1){$1$}; 
    \node (b) at (-1,0){$2$};
    \node (c) at (0,1){$3$};
    \node (d) at (1,0){$4$};
    \node (e) at (0,0){$\infty$};
\draw (a) to[out=180,in=-90] (b);
\draw (b) to[out=90,in=180] (c);
\draw (c) to[out=0,in=90] (d);
\draw (d) to[out=-90,in=0] (a);
\draw[red] (d) to[out=-170,in=-10] (e);
\draw[red] (a) to[out=75,in=-75] (e);
\draw[red] (c) to[out=-80,in=80] (e);
\draw[red] (a) to[out=105,in=-105] (e);
  \end{tikzpicture}
  \caption{The Crawley-Boevey quiver of $2\omega_1+\omega_3+\omega_4$ for $\mathfrak{\widehat{sl}}_4$.}
  \label{fig:crawley-boevey}
\end{figure}}

The data $\vartheta_i$ and $\ck=\Re(k)$ gives a weighting on the Crawley-Boevey graph, that
is, a function $\vartheta\colon \Omega_\la \to \R$, such that
every edge of $U$ has weight $\ck$ and
$\vartheta_i$ giving the weights of the new edges.

As before, we let $\K$ be a field and we now assume that ${S}$
is a $\K$-algebra  with a choice of elements
$h,z_1,\dots,z_\ell\in{S}$.  The most interesting choices for
us will be $\K$ itself with $h=z_1=\cdots=z_\ell=0$ or $\mathscr{R}$.  
For each edge, we set the polynomials
$Q_e(u,v)=u-v+h\in \K[u,v]$ for old edges and $Q_{e_i}(u,v)=v-u-z_i\in \K[u,v]$
for new edges, and consider
the {\bf weighted KLR algebra} $ W^\vartheta$ of
the Crawley-Boevey quiver as defined in \cite[\S \ref{w-sec:Crawley-Boevey}]{WebwKLR}.  As in
that paper, we will only consider dimension vectors with $d_\infty=1$.

Let us briefly recall the definition of this algebra. 

\begin{definition}
  We let a  {\bf weighted
KLR diagram} be a collection of curves in
  $\R\times [0,1]$ with each curve mapping diffeomorphically to
  $[0,1]$ via the projection to the $y$-axis.  Each curve is allowed
  to carry any number of dots, and has a label that lies in $U$. We draw:
  \begin{itemize}
  \item a dashed line $\ck$ units to the right of each strand,
    which we call a {\bf ghost},
\item red
lines at $x=\vartheta_i$ each of which carries a label
$\omega_{r_j}$. 
  \end{itemize} 
We now require that there are no
  triple points or tangencies involving any combination of strands,
  ghosts or red lines and no dots lie on crossings.  We consider these diagrams equivalent if they
  are related by an isotopy that avoids these tangencies, double
  points and dots on crossings.
\end{definition} 
Note that this is a bit different from the description in
\cite{WebwKLR}; we've specialized to the case of Crawley-Boevey quiver
with one vertical strand at $x=0$ labeled with the vertex $\infty$.  The
red lines
are the ghosts of this single vertical stand with label $\infty$.

 This definition is quite similar to the
conditions we considered in Section \ref{sec:pict-cher-algebr}; the
only difference is that we use black in place of green, label each of
these strands with an element of $U$ and denote the
polynomial generators with a dot instead of a square (and don't allow
negative powers of them).

For example, consider the case where $k=\nicefrac 34$ and $r_1=r_2=0,r_3=\nicefrac
34,r_4=\nicefrac 12$ and
$\vartheta_1=4,\vartheta_2=1,\vartheta_3=6,\vartheta_4=-4$.  Thus, the
diagram with no black strands for this choice of weighting looks like:
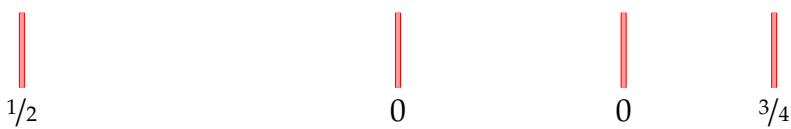
\[
\begin{tikzpicture}
  \draw[wei] (4,0) -- (4,1) node[below, at start]{$0$};
  \draw[wei] (1,0) -- (1,1)  node[below, at start]{$0$};
  \draw[wei] (6,0) -- (6,1) node[below, at start]{$\nicefrac
34$};
  \draw[wei] (-4,0) -- (-4,1)  node[below, at start]{$\nicefrac 12$};  
\end{tikzpicture}
\]
Adding in black strands will result in a diagram which
looks (for example) like:
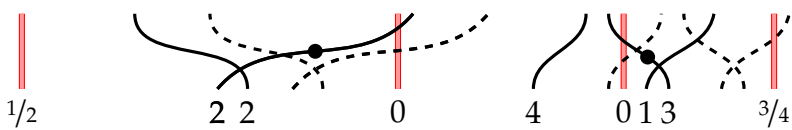
\[
\begin{tikzpicture}[very thick]
  \draw[wei] (4,0) -- (4,1) node[below, at start]{$0$};
  \draw[wei] (1,0) -- (1,1)  node[below, at start]{$0$};
  \draw[wei] (6,0) -- (6,1) node[below, at start]{$\nicefrac
34$};
  \draw[wei] (-4,0) -- (-4,1)  node[below, at start]{$\nicefrac 12$}; 
\draw (-1,0) to[out=90,in=-90] node[below, at start]{$2$}(-2.5,1);
\draw[dashed] (0,0) to[out=90,in=-90] (-1.5,1);
\draw (-1.4,0) to[out=50,in=-130] node[below, at start]{$2$}(1.2,1);
\draw[dashed] (-.4,0) to[out=50,in=-130] (2.2,1);
\draw (-1.4,0) to[out=50,in=-130] node [pos=.5,fill,color=black,inner
sep=2pt,circle]{} node[below, at start]{$2$}(1.2,1);
\draw[dashed] (-.4,0) to[out=50,in=-130] (2.2,1);
\draw (4.3,0) to[out=80,in=-100] node[below, at start]{$1$}(5.2,1);
\draw[dashed] (5.3,0) to[out=80,in=-100] (6.2,1);
\draw (4.6,0) to[out=90,in=-90]  node [pos=.4,fill,color=black,inner
sep=2pt,circle]{} node[below, at start]{$3$}(3.8,1);
\draw[dashed] (5.6,0) to[out=90,in=-90] (4.8,1);
\draw (2.8,0) to[out=90,in=-90] node[below, at start]{$4$}(3.5,1);
\draw[dashed] (3.8,0) to[out=90,in=-90] (4.5,1);
\end{tikzpicture}
\]
In $W^\vartheta$, we
have idempotents $e_\Bi$ indexed not just by sequences of nodes in the Dynkin
diagram, but by combinatorial objects we call {\bf loadings},
discussed earlier.  A loading is
a function from the real line to $U\cup\{\emptyset\}$ which is $\emptyset$ at all
but finitely many points.  Diagrammatically, we think of this as
encoding the positions of the black strands on a horizontal line.
Thus, a loading will arise from a generic horizontal slice of a
weighted KLR diagram, and the idempotent corresponding to a loading
has exactly that slice at every value of $y$.  

Of course, there are infinitely many such loadings.  Typically, we
will only consider these loadings up to {\bf equivalence}, as defined in
\cite[\ref{w-def:equivalent}]{WebwKLR}.  There only finitely many equivalence classes, so the
resulting algebra is more tractable.  

\begin{definition}
The weighted KLR algebra $\tilde{T}^\vartheta$ is the quotient of the span of
weighted KLR diagrams by the local relations:
\newseq
\begin{equation*}\subeqn\label{dots-1}
    \begin{tikzpicture}[scale=.45,baseline]
      \draw[very thick](-4,0) +(-1,-1) -- +(1,1) node[below,at start]
      {$i$}; \draw[very thick](-4,0) +(1,-1) -- +(-1,1) node[below,at
      start] {$j$}; \fill (-4.5,.5) circle (5pt);
      \node at (-2,0){=}; \draw[very thick](0,0) +(-1,-1) -- +(1,1)
      node[below,at start] {$i$}; \draw[very thick](0,0) +(1,-1) --
      +(-1,1) node[below,at start] {$j$}; \fill (.5,-.5) circle (5pt);
      \node at (4,0){for $i\neq j$};
    \end{tikzpicture}\end{equation*}
\begin{equation*}\label{dots-2}\subeqn
    \begin{tikzpicture}[scale=.45,baseline]
      \draw[very thick](-4,0) +(-1,-1) -- +(1,1) node[below,at start]
      {$i$}; \draw[very thick](-4,0) +(1,-1) -- +(-1,1) node[below,at
      start] {$i$}; \fill (-4.5,.5) circle (5pt);
      \node at (-2,0){=}; \draw[very thick](0,0) +(-1,-1) -- +(1,1)
      node[below,at start] {$i$}; \draw[very thick](0,0) +(1,-1) --
      +(-1,1) node[below,at start] {$i$}; \fill (.5,-.5) circle (5pt);
      \node at (2,0){+}; \draw[very thick](4,0) +(-1,-1) -- +(-1,1)
      node[below,at start] {$i$}; \draw[very thick](4,0) +(0,-1) --
      +(0,1) node[below,at start] {$i$};
    \end{tikzpicture}\qquad 
    \begin{tikzpicture}[scale=.45,baseline]
      \draw[very thick](-4,0) +(-1,-1) -- +(1,1) node[below,at start]
      {$i$}; \draw[very thick](-4,0) +(1,-1) -- +(-1,1) node[below,at
      start] {$i$}; \fill (-4.5,-.5) circle (5pt);
      \node at (-2,0){=}; \draw[very thick](0,0) +(-1,-1) -- +(1,1)
      node[below,at start] {$i$}; \draw[very thick](0,0) +(1,-1) --
      +(-1,1) node[below,at start] {$i$}; \fill (.5,.5) circle (5pt);
      \node at (2,0){+}; \draw[very thick](4,0) +(-1,-1) -- +(-1,1)
      node[below,at start] {$i$}; \draw[very thick](4,0) +(0,-1) --
      +(0,1) node[below,at start] {$i$};
    \end{tikzpicture}
  \end{equation*}
\begin{equation*}\label{strand-bigon}\subeqn
    \begin{tikzpicture}[very thick,scale=.8,baseline]
      \draw (-2.8,0) +(0,-1) .. controls (-1.2,0) ..  +(0,1)
      node[below,at start]{$i$}; \draw (-1.2,0) +(0,-1) .. controls
      (-2.8,0) ..  +(0,1) node[below,at start]{$i$}; \node at (-.5,0)
      {=}; \node at (0.4,0) {$0$};
\node at (1.5,.05) {and};
    \end{tikzpicture}
\hspace{.4cm}
    \begin{tikzpicture}[very thick,scale=.8 ,baseline]

      \draw (-2.8,0) +(0,-1) .. controls (-1.2,0) ..  +(0,1)
      node[below,at start]{$i$}; \draw (-1.2,0) +(0,-1) .. controls
      (-2.8,0) ..  +(0,1) node[below,at start]{$j$}; \node at (-.5,0)
      {=};

\draw (1.8,0) +(0,-1) -- +(0,1) node[below,at start]{$j$};
      \draw (1,0) +(0,-1) -- +(0,1) node[below,at start]{$i$}; 
    \end{tikzpicture}
  \end{equation*} 
\begin{equation*}\label{ghost-bigon1}\subeqn
\begin{tikzpicture}[very thick,xscale=1.6 ,yscale=.8,baseline]
 \draw (1,-1) to[in=-90,out=90]  node[below, at start]{$i$} (1.5,0) to[in=-90,out=90] (1,1)
;
  \draw[dashed] (1.5,-1) to[in=-90,out=90] (1,0) to[in=-90,out=90] (1.5,1);
  \draw (2.5,-1) to[in=-90,out=90]  node[below, at start]{$j$} (2,0) to[in=-90,out=90] (2.5,1);
\node at (3,0) {=};
  \draw (3.7,-1) -- (3.7,1) node[below, at start]{$i$}
 ;
  \draw[dashed] (4.2,-1) to (4.2,1);
  \draw (5.2,-1) -- (5.2,1) node[below, at start]{$j$};  \node at (6.2,0){for $i+k\neq j$};
\end{tikzpicture}
\end{equation*} 
\begin{equation*}\label{ghost-bigon1a}\subeqn
\begin{tikzpicture}[very thick,xscale=1.6 ,yscale=.8,baseline]
 \draw (1.5,-1) to[in=-90,out=90]  node[below, at start]{$i$} (1,0) to[in=-90,out=90] (1.5,1)
;
  \draw[dashed] (1,-1) to[in=-90,out=90] (1.5,0) to[in=-90,out=90] (1,1);
  \draw (2,-1) to[in=-90,out=90]  node[below, at start]{$j$} (2.5,0) to[in=-90,out=90] (2,1);
\node at (3,0) {=};
  \draw (4.2,-1) -- (4.2,1) node[below, at start]{$i$}
 ;
  \draw[dashed] (3.7,-1) to (3.7,1);
  \draw (4.7,-1) -- (4.7,1) node[below, at start]{$j$};  \node at (6.2,0){for $i+k\neq j$};
\end{tikzpicture}
\end{equation*} \begin{equation*}\label{ghost-bigon2}\subeqn
\begin{tikzpicture}[very thick,xscale=1.4,baseline=25pt]
 \draw (1,0) to[in=-90,out=90]  node[below, at start]{$i$} (1.5,1) to[in=-90,out=90] (1,2)
;
  \draw[dashed] (1.5,0) to[in=-90,out=90] (1,1) to[in=-90,out=90] (1.5,2);
  \draw (2.5,0) to[in=-90,out=90]  node[below, at start]{$i+k$} (2,1) to[in=-90,out=90] (2.5,2);
\node at (3,1) {=};
  \draw (3.7,0) -- (3.7,2) node[below, at start]{$i$}
 ;
  \draw[dashed] (4.2,0) to (4.2,2);
  \draw (5.2,0) -- (5.2,2) node[below, at start]{$i+k$} node[midway,fill,inner sep=2.5pt,circle]{};
\node at (5.75,1) {$-$};

  \draw (6.2,0) -- (6.2,2) node[below, at start]{$i$} node[midway,fill,inner sep=2.5pt,circle]{};
  \draw[dashed] (6.7,0)-- (6.7,2);
  \draw (7.7,0) -- (7.7,2) node[below, at start]{$i+k$};
\node at (8.15,1) {$+$}; \node at (8.45,1){$h$};
  \draw (8.7,0) -- (8.7,2) node[below, at start]{$i$};
  \draw[dashed] (9.2,0)-- (9.2,2);
  \draw (10.2,0) -- (10.2,2) node[below, at start]{$i+k$};
\end{tikzpicture}
\end{equation*}
 \begin{equation*}\label{ghost-bigon2a}\subeqn
\begin{tikzpicture}[very thick,xscale=1.4,baseline=25pt]
 \draw (1.5,0) to[in=-90,out=90]  node[below, at start]{$i$} (1,1) to[in=-90,out=90] (1.5,2)
;
  \draw[dashed] (1,0) to[in=-90,out=90] (1.5,1) to[in=-90,out=90] (1,2);
  \draw (2,0) to[in=-90,out=90]  node[below, at start]{$i+k$} (2.5,1) to[in=-90,out=90] (2,2);
\node at (3,1) {=};
  \draw (4.2,0) -- (4.2,2) node[below, at start]{$i$}
 ;
  \draw[dashed] (3.7,0) to (3.7,2);
  \draw (4.7,0) -- (4.7,2) node[below, at start]{$i+k$} node[midway,fill,inner sep=2.5pt,circle]{};
\node at (5.25,1) {$-$};

  \draw (6.2,0) -- (6.2,2) node[below, at start]{$i$} node[midway,fill,inner sep=2.5pt,circle]{};
  \draw[dashed] (5.7,0)-- (5.7,2);
  \draw (6.7,0) -- (6.7,2) node[below, at start]{$i+k$};
\node at (7.15,1) {$+$}; \node at (7.45,1){$h$};
  \draw (8.2,0) -- (8.2,2) node[below, at start]{$i$};
  \draw[dashed] (7.7,0)-- (7.7,2);
  \draw (8.7,0) -- (8.7,2) node[below, at start]{$i+k$};
\end{tikzpicture}
\end{equation*}
 \begin{equation*}\subeqn\label{triple-boring}
    \begin{tikzpicture}[very thick,scale=1 ,baseline]
      \draw (-3,0) +(1,-1) -- +(-1,1) node[below,at start]{$m$}; \draw
      (-3,0) +(-1,-1) -- +(1,1) node[below,at start]{$i$}; \draw
      (-3,0) +(0,-1) .. controls (-4,0) ..  +(0,1) node[below,at
      start]{$j$}; \node at (-1,0) {=}; \draw (1,0) +(1,-1) -- +(-1,1)
      node[below,at start]{$m$}; \draw (1,0) +(-1,-1) -- +(1,1)
      node[below,at start]{$i$}; \draw (1,0) +(0,-1) .. controls
      (2,0) ..  +(0,1) node[below,at start]{$j$};
    \end{tikzpicture}
  \end{equation*}
\begin{equation*}\subeqn \label{eq:triple-point1}
    \begin{tikzpicture}[very thick,xscale=1.6,yscale=.8,baseline]
      \draw[dashed] (-3,0) +(.4,-1) -- +(-.4,1);
 \draw[dashed]      (-3,0) +(-.4,-1) -- +(.4,1); 
    \draw (-1.5,0) +(.4,-1) -- +(-.4,1) node[below,at start]{$i+k$}; \draw
      (-1.5,0) +(-.4,-1) -- +(.4,1) node[below,at start]{$i+k$}; 
 \draw (-3,0) +(0,-1) .. controls (-3.5,0) ..  +(0,1) node[below,at
      start]{$i$};\node at (-.75,0) {=};  \draw[dashed] (0,0) +(.4,-1) -- +(-.4,1);
 \draw[dashed]      (0,0) +(-.4,-1) -- +(.4,1); 
    \draw (1.5,0) +(.4,-1) -- +(-.4,1) node[below,at start]{$i+k$}; \draw
      (1.5,0) +(-.4,-1) -- +(.4,1) node[below,at start]{$i+k$}; 
 \draw (0,0) +(0,-1) .. controls (.5,0) ..  +(0,1) node[below,at
      start]{$i$};
\node at (2.25,0)
      {$-$};   
     \draw (4.5,0)
      +(.4,-1) -- +(.4,1) node[below,at start]{$i+k$}; \draw (4.5,0)
      +(-.4,-1) -- +(-.4,1) node[below,at start]{$i+k$}; 
 \draw[dashed] (3,0)
      +(.4,-1) -- +(.4,1); \draw[dashed] (3,0)
      +(-.4,-1) -- +(-.4,1); 
\draw (3,0)
      +(0,-1) -- +(0,1) node[below,at start]{$i$};
    \end{tikzpicture}
  \end{equation*}
\begin{equation*}\subeqn\label{eq:triple-point2}
    \begin{tikzpicture}[very thick,xscale=1.6,yscale=.8,baseline]
\draw[dashed] (-3,0) +(0,-1) .. controls (-3.5,0) ..  +(0,1) ;  
  \draw (-3,0) +(.4,-1) -- +(-.4,1) node[below,at start]{$i$}; \draw
      (-3,0) +(-.4,-1) -- +(.4,1) node[below,at start]{$i$}; 
 \draw (-1.5,0) +(0,-1) .. controls (-2,0) ..  +(0,1) node[below,at
      start]{$i+k$};\node at (-.75,0) {=};  
    \draw (0,0) +(.4,-1) -- +(-.4,1) node[below,at start]{$i$}; \draw
      (0,0) +(-.4,-1) -- +(.4,1) node[below,at start]{$i$}; 
 \draw[dashed] (0,0) +(0,-1) .. controls (.5,0) ..  +(0,1);
 \draw (1.5,0) +(0,-1) .. controls (2,0) ..  +(0,1) node[below,at
      start]{$i+k$};
\node at (2.25,0)
      {$+$};   
     \draw (3,0)
      +(.4,-1) -- +(.4,1) node[below,at start]{$i$}; \draw (3,0)
      +(-.4,-1) -- +(-.4,1) node[below,at start]{$i$}; 
\draw[dashed] (3,0)
      +(0,-1) -- +(0,1);\draw (4.5,0)
      +(0,-1) -- +(0,1) node[below,at start]{$i+k$};
    \end{tikzpicture}.
  \end{equation*}
  \begin{equation*}\subeqn\label{cost}
  \begin{tikzpicture}[very thick,baseline,scale=.8,xscale=.7]
    \draw (-2.8,0)  +(0,-1) .. controls (-1.2,0) ..  +(0,1) node[below,at start]{$i$};
       \draw[wei] (-1.2,0)  +(0,-1) .. controls (-2.8,0) ..  +(0,1) node[below,at start]{$i$};
           \node at (-.3,0) {=};
    \draw[wei] (2.8,0)  +(0,-1) -- +(0,1) node[below,at start]{$i$};
       \draw (1.2,0)  +(0,-1) -- +(0,1) node[below,at start]{$i$};
       \fill (1.2,0) circle (3pt);
 \node at (3.8,0) {$-$};\node at (4.5,0){$z_k$};
        \draw[wei] (6.8,0)  +(0,-1) -- +(0,1) node[below,at start]{$i$};
       \draw (5.2,0)  +(0,-1) -- +(0,1) node[below,at start]{$i$};
  \end{tikzpicture}\qquad\qquad
  \begin{tikzpicture}[very thick,baseline,scale=.8,xscale=.7]
          \draw[wei] (6.8,0)  +(0,-1) .. controls (5.2,0) ..  +(0,1) node[below,at start]{$j$};
  \draw (5.2,0)  +(0,-1) .. controls (6.8,0) ..  +(0,1) node[below,at start]{$i$};
           \node at (7.7,0) {=};
    \draw (9.2,0)  +(0,-1) -- +(0,1) node[below,at start]{$i$};
       \draw[wei] (10.8,0)  +(0,-1) -- +(0,1) node[below,at start]{$j$};
  \end{tikzpicture}
\end{equation*} 
  \begin{equation*}\subeqn
    \begin{tikzpicture}[very thick,scale=.8,baseline]
      \draw (-3,0)  +(1,-1) -- +(-1,1) node[at start,below]{$i$};
      \draw (-3,0) +(-1,-1) -- +(1,1)node [at start,below]{$j$};
      \draw[wei] (-3,0)  +(0,-1) .. controls (-4,0) ..  node[at start,below]{$m$}+(0,1);
      \node at (-1,0) {=};
      \draw (1,0)  +(1,-1) -- +(-1,1) node[at start,below]{$i$};
      \draw (1,0) +(-1,-1) -- +(1,1) node [at start,below]{$j$};
      \draw[wei] (1,0) +(0,-1) .. controls (2,0) .. node[at start,below]{$m$} +(0,1);   
\node at (2.8,0) {$+ $};
      \draw (6.5,0)  +(1,-1) -- +(1,1)  node[at start,below]{$i$};
      \draw (6.5,0) +(-1,-1) -- +(-1,1) node [at start,below]{$j$};
      \draw[wei] (6.5,0) +(0,-1) -- node[at start,below]{$m$} +(0,1);
\node at (3.8,-.2){$\delta_{i,j,m} $}  ;
 \end{tikzpicture}
  \end{equation*}
\begin{equation*}\subeqn\label{dumb}
    \begin{tikzpicture}[very thick,scale=.8,baseline]
      \draw[wei] (-3,0)  +(1,-1) -- +(-1,1);
      \draw (-3,0)  +(0,-1) .. controls (-4,0) ..  +(0,1);
      \draw (-3,0) +(-1,-1) -- +(1,1);
      \node at (-1,0) {=};
      \draw[wei] (1,0)  +(1,-1) -- +(-1,1);
  \draw (1,0)  +(0,-1) .. controls (2,0) ..  +(0,1);
      \draw (1,0) +(-1,-1) -- +(1,1);    \end{tikzpicture}\qquad \qquad
    \begin{tikzpicture}[very thick,scale=.8,baseline]
  \draw(-3,0) +(-1,-1) -- +(1,1);
  \draw[wei](-3,0) +(1,-1) -- +(-1,1);
\fill (-3.5,-.5) circle (3pt);
\node at (-1,0) {=};
 \draw(1,0) +(-1,-1) -- +(1,1);
  \draw[wei](1,0) +(1,-1) -- +(-1,1);
\fill (1.5,.5) circle (3pt);
    \end{tikzpicture}
  \end{equation*}
For the relations (\ref{dumb}), we also include their mirror images.

Some care must be used when understanding what it means to apply these
relations locally.  In each case, the LHS and RHS have a dominant term
which are related to each other via an isotopy through a disallowed
diagram with a tangency, triple point or a dot on a crossing.  You can only apply the
relations if this isotopy avoids tangencies, triple points and dots on crossings
everywhere else in the diagram; one can always choose isotopy
representatives sufficiently generic for this to hold.
\end{definition}

This algebra is graded if $S$ is graded with $h,z_i$ having degree 2.
This is satisfied if $S=\K$ or $S=\K[h,z_1,\dots, z_\ell]$.
\begin{itemize}
\item As usual, the dot has degree 2.
\item The crossing of two strands has degree 0, unless they have the
  same label, in which case it's $-2$.  
\item The crossing of a strand with label $i$ from right of a ghost to left of it has
  degree 1 if the
  ghost has label $i+k$ and degree 0 otherwise. 
\item Such a crossing from left to right has degree 1 if the
  ghost has label $i+k$ and degree 0 otherwise.
\end{itemize}
That is,
\[
\deg\tikz[baseline,very thick,scale=1.5]{\draw (.2,.3) --
  (-.2,-.1) node[at end,below, scale=.8]{$i$}; \draw
  (.2,-.1) -- (-.2,.3) node[at start,below,scale=.8]{$j$};} =-2\delta_{i,j} \qquad  \deg\tikz[baseline,very thick,scale=1.5]{\draw
  (0,.3) -- (0,-.1) node[at end,below,scale=.8]{$i$}
  node[midway,circle,fill=black,inner
  sep=2pt]{};}=2 \qquad \deg\tikz[baseline,very thick,scale=1.5]{\draw[densely dashed] 
  (-.2,-.1)-- (.2,.3) node[at start,below, scale=.8]{$i$}; \draw
  (.2,-.1) -- (-.2,.3) node[at start,below,scale=.8]{$j$};} =\delta_{j,i-k} \qquad \deg\tikz[baseline,very thick,scale=1.5]{\draw (.2,.3) --
  (-.2,-.1) node[at end,below, scale=.8]{$i$}; \draw [densely dashed]
  (.2,-.1) -- (-.2,.3) node[at start,below,scale=.8]{$j$};} =\delta_{j,i+k}\]

This algebra has a {\bf reduced steadied quotient}, which we will denote
$T^\vartheta$.  This is obtained by
\begin{itemize}
\item killing all idempotents where the strands can be broken into two
  groups separated by a blank space of size $>|\ck|$ (so no ghost from
  the right group can be left of a strand in the left group and {\it
    vice versa}) and all red strands in the right group; we call such idempotents {\bf unsteady}.
\item killing all dots on the strand with label $\infty$.  
\end{itemize}

We'll just remind the reader that we allow the case where $k\in \Z$
(so $e=1$).
In this case, the graph $U$ is just elements of $\C/\Z$ equal to one
of the $r_i$, connected to itself by a loop and the equations $i=j,i=j+k,i=j-k$ are all equivalent.

\begin{remark}
  We should note that unlike in the tensor product algebras for
  $\slehat$ in \cite[\S 3]{Webmerged}, a black line being left of a red is not
  enough to conclude the diagram is 0; it must be far enough left to
  avoid all entanglements with ghosts.  See Example \ref{e-is-2} below.
\end{remark}

We can associate the elements of $U$ to roots of $\gu$.
As in \cite{Webmerged}, we'll let $T^\vartheta_\nu$ for $\nu$ a weight
of $\glehat$ be the subalgebra where the sum of the weights $\la_i$
minus the sum of the roots labeling the black strands is $\nu$. For
$e\neq 1$, it is sufficient to consider the $\gu$-weight, but for
$e=1$, it is not quite clear what this means.  The algebra $\mathfrak{\widehat{gl}}_1$
has a ``Cartan algebra'' which is 2-dimensional
with basis $c,\partial$; we let $\omega, \al$ be the dual basis.  The
weights of the highest weight Fock representation are $\om,\om-\al,\om-2\al,\dots$.

\begin{example}\label{e-is-2}
  Let $k=-\nicefrac{1}{2}$, $Q_1=0$.  Rather than list idempotents up to
  equivalence, which is still a bit redundant, let us implicitly
  identify idempotents easily found to be isomorphic using the
  relations above.  If we have one black strand, then we can 
  see that we obtain the trivial algebra if it is labeled $\nicefrac{1}{2}$, and a
  1-dimensional algebra if it is labeled $0$ (in both cases, this is
  just the corresponding cyclotomic quotient).
Similarly, if we have two black strands with the same label we get the
trivial algebra  again.  

On the other hand, for one strand labeled $0$ and one labeled $\nicefrac{1}{2}$, the
picture is more interesting.  We get 2 interesting idempotents, which
can be represented visually by 
\[e_1=\:\tikz[baseline=-2pt,very thick, xscale=3] { 
\draw (.05,-.5) -- node[below, at start]{$\nicefrac{1}{2}$} (0.05,.5);
\draw[dashed] (-.55,-.5) -- (-.55,.5);
\draw[dashed] (-.1,-.5) -- (-.1,.5);
\draw (.5,-.5) -- node[below, at start]{$0$} (.5,.5);
\draw[wei](.35,-.5) -- node[below, at start]{$\omega_0$} (.35,.5);
 }
\qquad e_2=\:\tikz[baseline=-2pt,very thick, xscale=3] { 
\draw (1.25,-.5) -- node[below, at start]{$\nicefrac{1}{2}$} (1.25,.5);
\draw[dashed] (.65,-.5) -- (.65,.5);
\draw[dashed] (-.1,-.5) -- (-.1,.5);
\draw (.5,-.5) -- node[below, at start]{$0$} (.5,.5);
\draw[wei](.35,-.5) -- node[below, at start]{$\omega_0$} (.35,.5);
 }\]
One can easily calculate that $e_1T^{\vartheta}e_1\cong \K$ and that
$e_2T^{\vartheta}e_2\cong \K[y_2]/(y_2^2)$ where $y_2$ represents the dot on
the rightward strand. 

Note that $e_1$ is not unsteady (and in fact is nonzero in the
steadied quotient), even though it contains a black strand left of
a red one, since that strand is ``protected'' by a ghost.  The
idempotents 
\[\tikz[baseline=-2pt,very thick, xscale=3] { 
\draw (-.25,-.5) -- node[below, at start]{$\nicefrac{1}{2}$} (-0.25,.5);
\draw[dashed] (-.85,-.5) -- (-.85,.5);
\draw[dashed] (-.1,-.5) -- (-.1,.5);
\draw (.5,-.5) -- node[below, at start]{$0$} (.5,.5);
\draw[wei](.35,-.5) -- node[below, at start]{$\omega_0$} (.35,.5);
 } \qquad\text{\rm and} \qquad\tikz[baseline=-2pt,very thick, xscale=3] { 
\draw (-.25,-.5) -- node[below, at start]{$\nicefrac{1}{2}$} (-0.25,.5);
\draw[dashed] (-.85,-.5) -- (-.85,.5);
\draw[dashed] (-.4,-.5) -- (-.4,.5);
\draw (.2,-.5) -- node[below, at start]{$0$} (.2,.5);
\draw[wei](.35,-.5) -- node[below, at start]{$\omega_0$} (.35,.5);
 }\] are unsteady, and thus sent to 0.  Note that the idempotent 
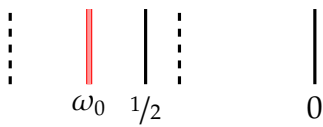
\[\tikz[baseline=-2pt,very thick, xscale=3] { 
\draw (-.25,-.5) -- node[below, at start]{$\nicefrac{1}{2}$} (-0.25,.5);
\draw[dashed] (-.85,-.5) -- (-.85,.5);
\draw[dashed] (-.1,-.5) -- (-.1,.5);
\draw (.5,-.5) -- node[below, at start]{$0$} (.5,.5);
\draw[wei](-.5,-.5) -- node[below, at start]{$\omega_0$} (-.5,.5);
 }\] isn't unsteady, but it is isomorphic to the left-hand unsteady
 idempotent above, by relation (\ref{cost}).

Savvy representation theorists will have already guessed that we've
arrived at the familiar highest weight category with these
endomorphism rings for its projectives; for example, this is given by
a regular integral block of category $\cO$ for $\mathfrak{sl}_2$.  A
basis of this ring is given by $e_1,e_2$ as above and
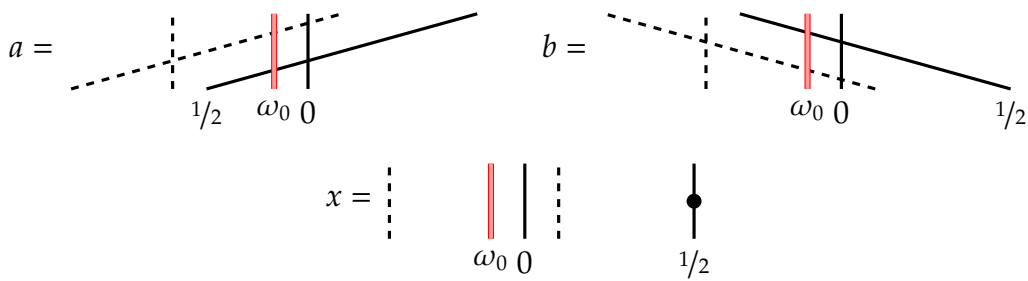
\[a=\:\tikz[baseline=-2pt,very thick, xscale=3] { 
\draw (.05,-.5) -- node[below, at start]{$\nicefrac{1}{2}$} (1.25,.5);
\draw[dashed] (-.55,-.5) -- (.65,.5);
\draw[dashed] (-.1,-.5) -- (-.1,.5);
\draw (.5,-.5) -- node[below, at start]{$0$} (.5,.5);
\draw[wei](.35,-.5) -- node[below, at start]{$\omega_0$} (.35,.5);
 }
\qquad b=\:\tikz[baseline=-2pt,very thick, xscale=3] { 
\draw (1.25,-.5) -- node[below, at start]{$\nicefrac{1}{2}$} (0.05,.5);
\draw[dashed] (.65,-.5) -- (-.55,.5);
\draw[dashed] (-.1,-.5) -- (-.1,.5);
\draw (.5,-.5) -- node[below, at start]{$0$} (.5,.5);
\draw[wei](.35,-.5) -- node[below, at start]{$\omega_0$} (.35,.5);
 }\]
\[x=\:\tikz[baseline=-2pt,very thick, xscale=3] { 
\draw (1.25,-.5) -- node[below, at start]{$\nicefrac{1}{2}$} node [pos=.5,fill,color=black,inner
sep=2pt,circle]{}  (1.25,.5);
\draw[dashed] (.65,-.5) -- (.65,.5);
\draw[dashed] (-.1,-.5) -- (-.1,.5);
\draw (.5,-.5) -- node[below, at start]{$0$} (.5,.5);
\draw[wei](.35,-.5) -- node[below, at start]{$\omega_0$} (.35,.5);
 }\]
The product $x=ab$ follows from the relation (\ref{ghost-bigon2})
since
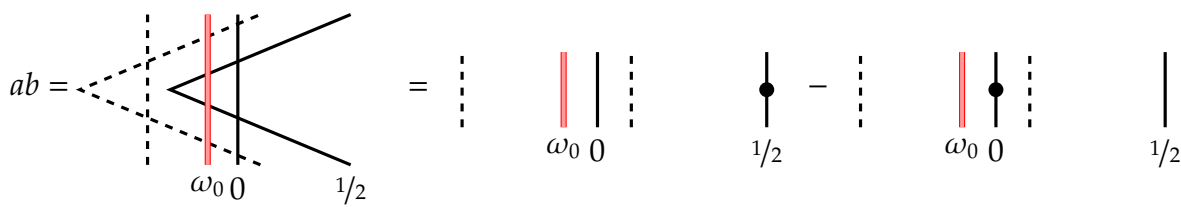
\[ab=\:\tikz[baseline=-2pt,very thick, xscale=2] { 
\draw (1.25,-1) --node[below, at start]{$\nicefrac{1}{2}$} (0.05,0)--  (1.25,1);
\draw[dashed] (.65,-1) --  (-.55,0) --(.65,1);
\draw[dashed] (-.1,-1) -- (-.1,1);
\draw (.5,-1) -- node[below, at start]{$0$} (.5,1);
\draw[wei](.3,-1) -- node[below, at start]{$\omega_0$} (.3,1);
 }\hspace{3mm}=\hspace{3mm}\tikz[baseline=-2pt,very thick, xscale=3] { 
\draw (1.25,-.5) -- node[below, at start]{$\nicefrac{1}{2}$} node [pos=.5,fill,color=black,inner
sep=2pt,circle]{}  (1.25,.5);
\draw[dashed] (.65,-.5) -- (.65,.5);
\draw[dashed] (-.1,-.5) -- (-.1,.5);
\draw (.5,-.5) -- node[below, at start]{$0$} (.5,.5);
\draw[wei](.35,-.5) -- node[below, at start]{$\omega_0$} (.35,.5);
 }\hspace{2mm}-\hspace{4mm}\tikz[baseline=-2pt,very thick, xscale=3] { 
\draw (1.25,-.5) -- node[below, at start]{$\nicefrac{1}{2}$} (1.25,.5);
\draw[dashed] (.65,-.5) -- (.65,.5);
\draw[dashed] (-.1,-.5) -- (-.1,.5);
\draw (.5,-.5) -- node[below, at start]{$0$} node [pos=.5,fill,color=black,inner
sep=2pt,circle]{}  (.5,.5);
\draw[wei](.35,-.5) -- node[below, at start]{$\omega_0$} (.35,.5);
 }
\]
The latter term is 0, by the calculation 
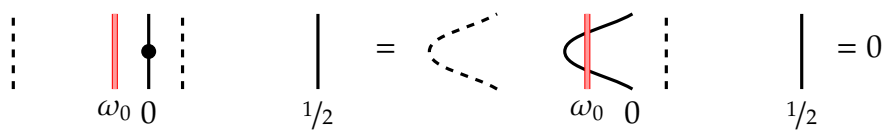
\[\tikz[baseline=-2pt,very thick, xscale=3] { 
\draw (1.25,-.5) -- node[below, at start]{$\nicefrac{1}{2}$} (1.25,.5);
\draw[dashed] (.65,-.5) -- (.65,.5);
\draw[dashed] (-.1,-.5) -- (-.1,.5);
\draw (.5,-.5) -- node[below, at start]{$0$} node [pos=.5,fill,color=black,inner
sep=2pt,circle]{}  (.5,.5);
\draw[wei](.35,-.5) -- node[below, at start]{$\omega_0$} (.35,.5);
 }\hspace{3mm}=\hspace{3mm}\tikz[baseline=-2pt,very thick, xscale=3] { 
\draw (1.25,-.5) -- node[below, at start]{$\nicefrac{1}{2}$} (1.25,.5);
\draw[dashed] (.65,-.5) -- (.65,.5);
\draw[dashed] (-.1,-.5) to[out=120,in=-90]  (-.4,0) to[out=90,in=-120] (-.1,.5);
\draw (.5,-.5) to[out=120,in=-90]  node[below, at start]{$0$} (.2,0) to[out=90,in=-120] (.5,.5);
\draw[wei](.3,-.5) -- node[below, at start]{$\omega_0$} (.3,.5);
 }=0
\]
since the second diagram factors through an unsteady idempotent.  One
can similarly calculate that $ba=ax=xa=bx=xb=0$.
\end{example}

If $\ck=0$, then we recover the tensor product algebra $T^\bla$
described in \cite[\S 3]{Webmerged} for the Lie algebra $\gu$.  We
view moving to $\ck\neq 0$ as passing from $\slehat$ to $\glehat$ in a
way that we shall make more precise.  This idea has appeared in
several places, for example, the work of Frenkel and Savage on quiver
varieties \cite{FS03}.

We'll generally be interested in the category $T^\vartheta\mmod$ of {\it
  graded} $T^\vartheta$-modules.  When we consider the category of
modules without a grading (for a graded or ungraded algebra), we'll
use the symbol $T^\vartheta\umod$.

Assuming that $e\neq 1$, the category $T^\vartheta\mmod$ carries a
categorical action of $\gu$, via functors $\eF_i$ and $\eE_i$,
which basically correspond to the addition and removal of a black line
with label $i\in U$, defined in \cite[\S
\ref{w-sec:Crawley-Boevey}]{WebwKLR}.   We can view these are the
induction and restriction functors for the map of rings
$\alg^\vartheta_\nu\to \alg^\vartheta_{\nu-\al_i}$ which adds a black
strand with label $i$ at least $\ck$ units right of any other strand.

 If $e=1$, we still have functors $\eF_i$ and $\eE_i$ corresponding to
 the different points in $i\in U$ given by induction and restriction functors for
the same inclusion $\alg^\vartheta_\nu\to \alg^\vartheta_{\nu-\al_i}$. The functors $\eE_i$ and $\eF_i$ have a structure
 reminiscent of, but not identical with a categorical Heisenberg action
in the sense of Cautis and Licata \cite{CaLi}.  In particular, they do
categorify a level $\ell$-Fock space representation of
$U_q(\gu)$, as we'll prove later.

There is a symmetry of this picture:
\begin{proposition}\label{prop:symmetry}
  The map on a weighted KLR diagrams which keeps all red and black strands in the
  same place, reindexes their labels sending $\omega_i\mapsto
  \omega_{-i},\al_{i}\mapsto\al_{-i}$ and sends $\ck\mapsto -\ck$ is
  an isomorphism.
\end{proposition}
In terms of Uglov weightings, this sends $\vartheta_\Bs^\pm\mapsto
\vartheta_{\Bs^\star}^\mp$, where $\Bs^\star=(-s_{\ell},\dots,-s_1)$.

\subsection{An algebra isomorphism}
\label{sec:an-algebra-isom}

We use the same parameters as in Section \ref{sec:comparison-theorem}.
Let $\mathscr{D}$ be some collection of sets, and let $B$ be the
collection of all loadings on the graph $U$ where the underlying set is in
$\mathscr{D}$. 

If $\PC^\vartheta_{\mathscr{D}}$ is the WF Hecke algebra as
defined earlier over $\mathscr{R}$, then the spectrum of the action of a square lies in
this set by Lemma \ref{lem:joint-spectrum}.  Now, consider a
$U$-valued loading on a set $D\in \mathscr{D}$, that is, a
function $\Bi\colon D\to U$; we'll use $u_1,\dots,u_m$ be the
list of values of this function in increasing order. By abstract Jordan
decomposition, there's an idempotent $\epsilon_{\Bi}$ which projects
to the $\Bi(d)$ generalized eigenspace of $X_d$ for $d\in D$.  We'll
let $\dalg^\vartheta_B(\mathscr{R})$ denote the deformed steadied weighted KLR
algebra attached to the elements $r_i=k s_i\in U$ and the set of loadings $B$, base changed
by the natural map $\K[h,z_1,\dots, z_\ell]\to \mathscr{R}$.

We'll now define an algebra isomorphism between the WF Hecke algebras and steadied weighted KLR algebras.  This
isomorphism will be local in nature:  on each diagram, it operates by
replacing every crossing of strands or ghosts and every square with a
linear combination of diagrams in the weighted KLR algebra.
\begin{theorem}[\mbox{\cite[\ref{n-thm:Hecke-KLR}]{WebBKnote}}]\label{thm:Hecke-KLR}
  We have an isomorphism of $\mathscr{R}$-algebras $\PC^\vartheta_{\mathscr{D}}\cong
  \dalg^\vartheta_B$ sending
\[\epsilon_{\Bi}\mapsto e_{\Bi}\qquad\qquad X_p\mapsto\sum_{\Bi}u_pe^{y_p}e_{\Bi}\qquad \qquad \tikz[baseline,very thick,scale=1.5, green!50!black]{\draw (.2,.3) --
  (-.2,-.1); \draw [wei]
  (.2,-.1) -- node[ below, at start,black]{$\PQ_s$} (-.2,.3);} \mapsto\tikz[baseline,very
  thick,scale=1.5]{\draw (.2,.3) -- (-.2,-.1); \draw [wei] (.2,-.1) --
    node[ below, at start,black]{$r_s$} (-.2,.3);}\]
 \vspace{-3mm}

\[\tikz[baseline,very thick,scale=1.5, green!50!black]{\draw (-.2,.3) --
  (.2,-.1); \draw [wei]
  (-.2,-.1) --node[ below, at start,black]{$\PQ_s$} (.2,.3);} \mapsto
\begin{cases}\displaystyle
(u_pe^{y_p}-\PQ_s)\, \tikz[baseline,very thick,scale=1.5]{\draw (-.2,.3) --
  (.2,-.1); \draw [wei]
  (-.2,-.1) -- node[ below, at start,black]{$r_s$} (.2,.3);}&u_p\neq Q_s\\
  \displaystyle\frac{u_pe^{y_p}-\PQ_s}{y_p}\, \tikz[baseline,very thick,scale=1.5]{\draw (-.2,.3) --
  (.2,-.1); \draw [wei]
  (-.2,-.1) -- node[ below, at start,black]{$r_s$} (.2,.3);}&u_p=Q_s
\end{cases}
\]
\[\tikz[baseline,very thick,scale=1.5, green!50!black]{\draw (.2,.3) --
  (-.2,-.1); \draw
  (.2,-.1) -- (-.2,.3);}\, e_{\Bi}\mapsto
\begin{cases}
\displaystyle\frac{1}{u_{p+1}e^{y_{p+1}}-u_{p}e^{y_p}}\Big(\,\tikz[baseline,very thick,scale=1.5]{\draw (.2,.3) --
  (-.2,-.1); \draw
  (.2,-.1) -- (-.2,.3);}\,-\tikz[baseline,very thick,scale=1.5]{\draw (.1,.3) --
  (.1,-.1); \draw
  (-.1,-.1) -- (-.1,.3);}\,\Big) e_{\Bi} & u_p\neq u_{[+1}\\
 \displaystyle \frac{y_{p+1}-y_p}{u_p(e^{y_{p+1}}-e^{y_p})}\tikz[baseline,very thick,scale=1.5]{\draw (.2,.3) --
  (-.2,-.1); \draw
  (.2,-.1) -- (-.2,.3);}\, e_{\Bi}& u_p=u_{p+1}
\end{cases}\]
\[
\tikz[baseline,very thick,scale=1.5, green!50!black]{\draw[densely dashed] 
  (-.2,-.1)-- (.2,.3); \draw
  (.2,-.1) -- (-.2,.3);}\,e_{\Bi} \mapsto
\begin{cases}
\displaystyle (u_pe^{y_p}-\pq u_se^{y_s})\tikz[baseline,very thick,scale=1.5]{\draw[densely dashed]
    (-.2,-.1)-- (.2,.3); \draw
    (.2,-.1) -- (-.2,.3);}\, e_{\Bi}&
  u_r\neq qu_s\\ 
 \displaystyle \frac {u_p(e^{y_p}-e^{y_s+h})}{y_{s}-y_{p}+h}\tikz[baseline,very thick,scale=1.5]{\draw[densely dashed]
    (-.2,-.1)-- (.2,.3); \draw
    (.2,-.1) -- (-.2,.3);}\,e_{\Bi}& u_p=qu_s
\end{cases}
\qquad \qquad\tikz[baseline,very thick,scale=1.5, green!50!black]{\draw (.2,.3) --
  (-.2,-.1); \draw [densely dashed]
  (.2,-.1) -- (-.2,.3);} \mapsto \tikz[baseline,very thick,scale=1.5]{\draw (.2,.3) --
  (-.2,-.1); \draw [densely dashed]
  (.2,-.1) -- (-.2,.3);} \]
where the solid strand shown is the $p$th (and $p+1$st in the first line),
and the ghost is associated to the $s$th from the left, or the red
line is $s$th from the left.
\end{theorem}
Many interesting structures can be transported over from the Hecke
side to the KLR.  For example, if $\mathscr{D}_{s,m}\in \mathscr{D}
$, then
$e_{D_{s,m}}\PC^\vartheta_{\mathscr{D}}e_{D_{s,m}}\cong
H_m(\pq,\PQ_\bullet)$ by Proposition \ref{prop:Hecke-quotient}.  We
let $e^0$ be the
image of $e_{D_{s,m}}$ in $\dalg^\vartheta$; this is the sum of loadings on
the points $x=s,2s,\dots, ms$.   We call such loadings {\bf Hecke}.  The image of $H_m(\pq,\PQ_\bullet)$,
which is of course  $e^0\dalg^\vartheta_{B} e^0$,
is the deformed cyclotomic KLR algebra $\dalg^\lambda_m$ on $m$ strands, with the isomorphism being that of
\cite[Thm. \ref{n-O-isomorphism}]{WebBKnote} (which is a slight
modification of Brundan and Kleshchev's isomorphism from \cite{BKKL}).  
Thus, we see that:
\begin{lemma}\label{lem:Hecke-loadings}
  We have a commutative diagram of functors 
  \begin{equation*}
    \begin{tikzpicture}[->,very thick]{ \matrix[row sep=15mm,column sep=25mm,ampersand
    replacement=\&]{ \node (a) {$\PC^\vartheta_{\mathscr{D}}\mmod$}; \& \node (c)
      {$H_m(\pq,\PQ_\bullet)\mmod$}; \\
      \node (b) {$\dalg^\vartheta_{B}\mmod$}; \& \node (d) {$\dalg^\lambda_m\mmod$};\\
    }; \draw (a) -- (c)  node[above,midway]{$K$}; \draw (b)
    --(d)  node[below,midway]{$M\mapsto e^0M$}; \draw (a)--(b)
    node[right,midway]{$\sim$}; \draw (c)--(d)
    node[right,midway]{$\sim$}; }
    \end{tikzpicture}
  \end{equation*}
  In particular, the functor $M\mapsto e^0M$ is $0$-faithful by
  Corollary \ref{cor:0-faithful}, and the corresponding functor for
  $T^\vartheta$ is $-1$-faithful.  
\end{lemma}

We can combine this theorem with Theorem
\ref{theorem:diagrammatic-Cherednik} to compare category $\cO$ over a
Cherednik algebra to weighted KLR algebras.  Let $B^\circ_d$ be the
set of all loadings on sets in $\mathscr{D}^\circ_d$.  

\begin{theorem}\label{thm:cherednik-KLR}
  There is an equivalence of highest weight categories
  $T^{\vartheta^\sigma_{\Bs}}_{B^\circ_d}\umod\cong \cO^{\Bs}_d$,
  where $\sigma$ is the sign of $\kappa$.
\end{theorem}
Thus, considering the category $T^{\vartheta^\sigma_{\Bs}}_{B^\circ_d}\mmod$, we
obtain a graded lift of the category $\cO^{\Bs}_d$ as a highest weight category.

\subsection{Basis vectors}
Now, we use the combinatorics described above to give a cellular basis
of $\dalg^\vartheta_\nu$ and $T^\vartheta_\nu$, generalizing those of
\cite{HM,SWschur}.  This basis, and the variants of it we will
construct are the key to understanding the structure of these
quotients and their representation theory.

For each
$\Bi $-tableau $\sS$ of fixed shape $\nu$, we draw a diagram $B_\sS\in
e_{\nu}T^\vartheta e_{\Bi}$ which
has no dots, and connects the point connected to a box in $\Bi_\eta$
at $y=1$
to the point on the real line which labels it in $\sS$ at $y=0$;
put another way, the strands are in bijection with boxes, with each
strand ending just right of the $x$-coordinate of the box, and
starting at the real number labeling the box. Note the similarity to
the definition of $\hB_{\sS}$ given in Section \ref{sec:cellular-structure}.

The permutation $w_{\sS}$ traced out by the strands when read from the
top is the unique one which puts the Russian reading word of the
tableau into order.  As usual, letting $(-)^*$ be the
anti-automorphism which flips diagrams, let $\CST=B_{\sS}^*B_{\sT}$.
These vectors will be shown to be a cellular basis.  This will perhaps
be clarified a little by an example:

\begin{example}\label{big-example-2}
 Now, we consider the example where $k=\nicefrac{-9}{2}$, with $Q_0=0$
 and $Q_1=\nicefrac{1}{2}$, so $U=\{0,\nicefrac{1}{2}\}$.  Consider
 the algebra attached to $\mu= \omega_0+\omega_{\nicefrac{1}{2}}-\delta$.  We
 label the new edges so that $e_i$ connects to the node $i$.
The only resulting category, weighted order, and basis only depend on
the difference of the weights $\vartheta_1-\vartheta_2$.  In fact,
there are only 3 different possibilities; the category changes when
this value passes $\pm \nicefrac{9}{2}$.
 
There are 5 multipartitions of the right residue:
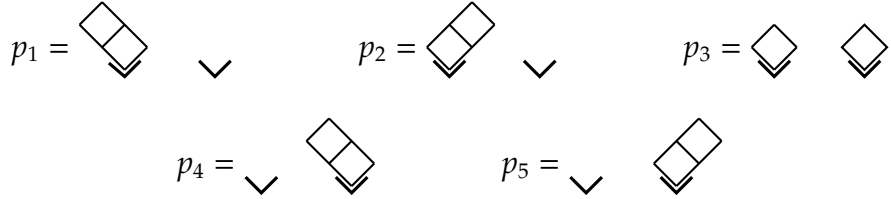
\[
p_1=\tikz[baseline=5pt,scale=.3,thick]{\draw (0,0) --(-2,2); \draw
  (0,0) --(1,1); \draw (-1,3) --(-2,2); \draw (-1,3) --(1,1); \draw
  (-1,1) -- (0,2); \draw[very thick] (-.7,.3) -- (0,-.3) -- (.7,.4); 
\draw[very thick] (3.3,.4) -- (4,-.3) -- (4.7,.4); 
}
\qquad \qquad 
p_2=\tikz[baseline=5pt,scale=.3,thick]{\draw (0,0) --(2,2); \draw
  (0,0) --(-1,1); \draw (1,3) --(2,2); \draw (1,3) --(-1,1); \draw
  (1,1) -- (0,2); \draw[very thick] (-.7,.4) -- (0,-.3) -- (.7,.4); 
\draw[very thick] (3.3,.4) -- (4,-.3) -- (4.7,.4); 
}
\qquad \qquad 
p_3=\tikz[baseline=5pt,scale=.3,thick]{\draw (0,0) --(1,1); \draw
  (0,0) --(-1,1); \draw (0,2) --(-1,1); \draw
  (1,1) -- (0,2); \draw[very thick] (-.7,.4) -- (0,-.3) -- (.7,.4); 
\draw[very thick] (3.3,.4) -- (4,-.3) -- (4.7,.4); \draw (4,0)
--(3,1)-- (4,2) -- (5,1)--cycle;
}
\]
\[
p_4=\tikz[baseline=5pt,scale=.3,thick]{\draw (0,0) --(-2,2); \draw
  (0,0) --(1,1); \draw (-1,3) --(-2,2); \draw (-1,3) --(1,1); \draw
  (-1,1) -- (0,2); \draw[very thick] (-.7,.4) -- (0,-.3) -- (.7,.4); 
\draw[very thick] (-3.3,.4) -- (-4,-.3) -- (-4.7,.4); 
}
\qquad \qquad 
p_5=\tikz[baseline=5pt,scale=.3,thick]{\draw (0,0) --(2,2); \draw
  (0,0) --(-1,1); \draw (1,3) --(2,2); \draw (1,3) --(-1,1); \draw
  (1,1) -- (0,2); \draw[very thick] (-.7,.4) -- (0,-.3) -- (.7,.4); 
\draw[very thick] (-3.3,.4) -- (-4,-.3) -- (-4.7,.4); 
}
\]
The basis vectors we draw will look exactly like those of Example
\ref{big-example-1}, except that now we draw black lines instead of
green and must label with the black strands with simple roots.
Since the pictures are so similar, let us specialize to the case with $\vartheta_1=0,\vartheta_2=9$.
In this case, our order is $p_1 > p_2> p_3>p_4>p_5$. 

A loading in this case is given by specifying the position the point $a$
labeled $0$ and the point $b$ labeled $\nicefrac{1}{2}$.  We denote this loading
$\Bi_{a,b}$.  As we'll see later, every projective is a summand of
that for one of $\Bi_{(1,-1)},\Bi_{(1,6)}, \Bi_{(1,10)},\Bi_{(8,10)},
\Bi_{(15,10)}$.  For these loadings, the tableaux with their
corresponding $B_\sS$'s are:
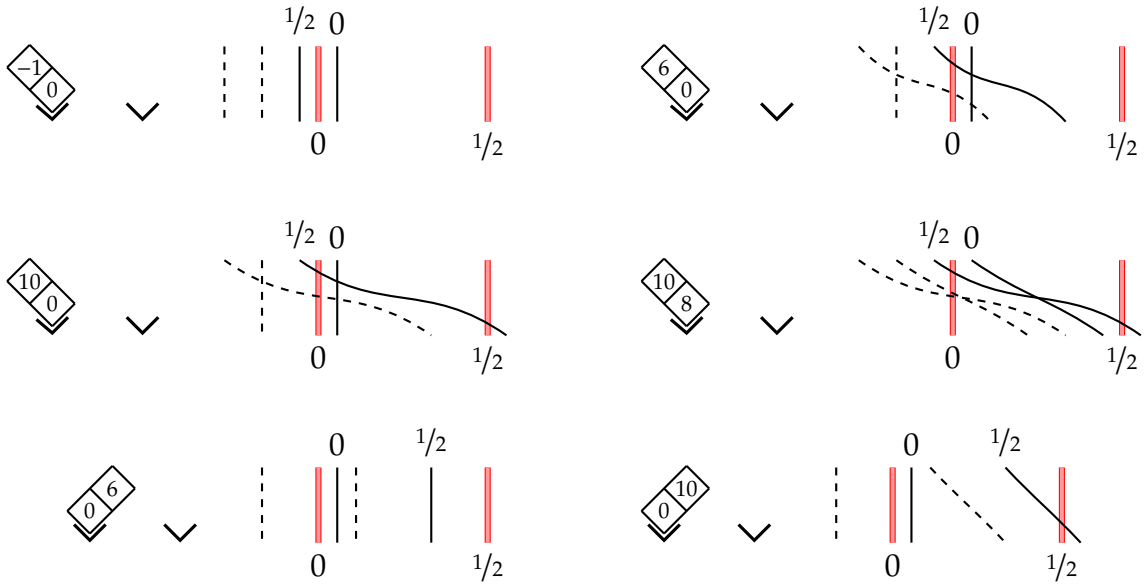
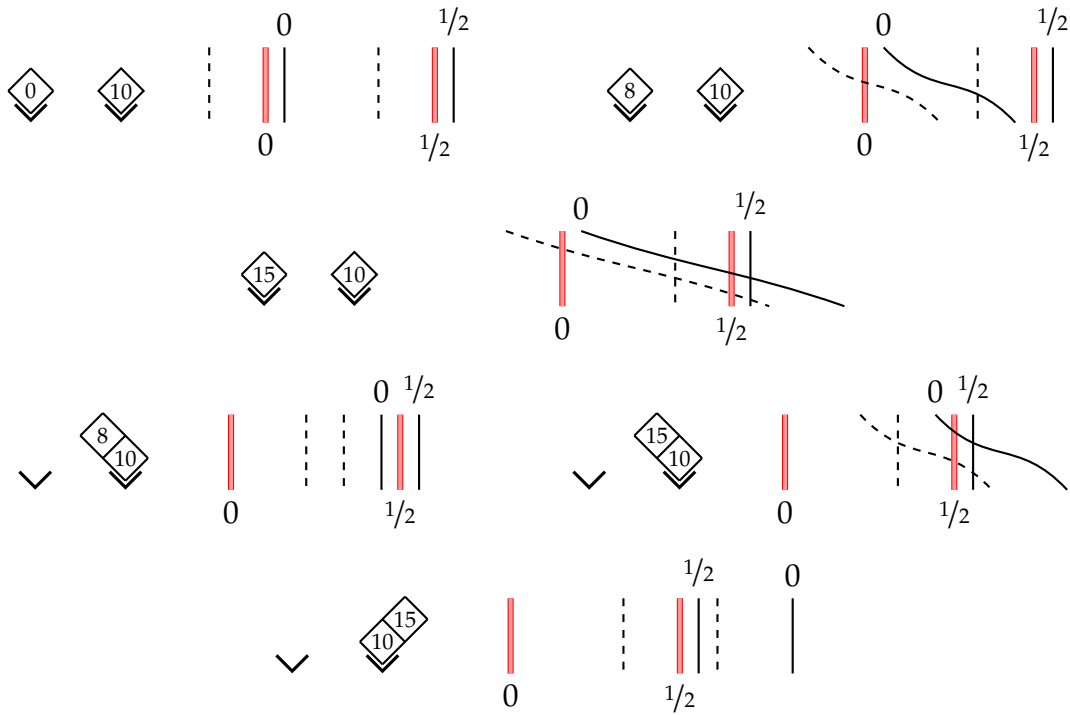
\[
\tikz[baseline=5pt,scale=.3,thick]{\draw (0,0) --(-2,2); \draw
  (0,0) --(1,1); \draw (-1,3) --(-2,2); \draw (-1,3) --(1,1); \draw
  (-1,1) -- (0,2); \draw[very thick] (-.7,.3) -- (0,-.3) -- (.7,.4); 
\draw[very thick] (3.3,.4) -- (4,-.3) -- (4.7,.4); \node[scale=.7] at (0,1)
{$0$}; \node[scale=.7] at (-1,2) {$-1$};
}
\qquad 
\tikz[baseline=-20pt,xscale=.25,yscale=-1, thick]{\draw[wei] (0,0) -- node[below,at end]{$0$} (0,1);
  \draw[wei] (9,0) -- (9,1) node[below, at end]{$\nicefrac{1}{2}$};\draw (-1,0) -- node[above, at start]{$\nicefrac{1}{2}$} (-1,1); \draw (1,0) --  node[above, at start]{$0$} (1,1);\draw[dashed] (-5,0) -- (-5,1); \draw[dashed] (-3,0) -- (-3,1);}
\qquad \qquad
\tikz[baseline=5pt,scale=.3,thick]{\draw (0,0) --(-2,2); \draw
  (0,0) --(1,1); \draw (-1,3) --(-2,2); \draw (-1,3) --(1,1); \draw
  (-1,1) -- (0,2); \draw[very thick] (-.7,.3) -- (0,-.3) -- (.7,.4); 
\draw[very thick] (3.3,.4) -- (4,-.3) -- (4.7,.4); \node[scale=.7] at (0,1)
{$0$}; \node[scale=.7] at (-1,2) {$6$};}\qquad 
\tikz[baseline=-20pt,xscale=.25,yscale=-1,thick]{\draw[wei] (0,0) -- node[below,at
  end]{$0$} (0,1); \draw[wei] (9,0) --  node[below, at end]{$\nicefrac{1}{2}$}
  (9,1);\draw (-1,0) to[out=15,in=-165]  node[above, at start]{$\nicefrac{1}{2}$} (6,1); \draw (1,0) -- node[above, at start]{$0$} (1,1);\draw[dashed] (-5,0) to[out=15,in=-165] (2,1); \draw[dashed] (-3,0) -- (-3,1);}\]
\[ \tikz[baseline=5pt,scale=.3,thick]{\draw (0,0) --(-2,2); \draw
  (0,0) --(1,1); \draw (-1,3) --(-2,2); \draw (-1,3) --(1,1); \draw
  (-1,1) -- (0,2); \draw[very thick] (-.7,.3) -- (0,-.3) -- (.7,.4); 
\draw[very thick] (3.3,.4) -- (4,-.3) -- (4.7,.4); \node[scale=.7] at (0,1)
{$0$}; \node[scale=.7] at (-1,2) {$10$};}
\qquad \tikz[baseline=-20pt,xscale=.25,yscale=-1,thick]{\draw[wei] (0,0) --
  node[below,at end]{$0$} (0,1); \draw[wei] (9,0) --  node[below, at
  end]{$\nicefrac{1}{2}$} (9,1);\draw (-1,0) to[out=10,in=-170]  node[above, at
  start]{$\nicefrac{1}{2}$} (10,1); \draw (1,0) --  node[above, at start]{$0$} (1,1);\draw[dashed] (-5,0) to[out=10,in=-170] (6,1); \draw[dashed] (-3,0) -- (-3,1);}
\qquad \qquad \tikz[baseline=5pt,scale=.3,thick]{\draw (0,0) --(-2,2); \draw
  (0,0) --(1,1); \draw (-1,3) --(-2,2); \draw (-1,3) --(1,1); \draw
  (-1,1) -- (0,2); \draw[very thick] (-.7,.3) -- (0,-.3) -- (.7,.4); 
\draw[very thick] (3.3,.4) -- (4,-.3) -- (4.7,.4); \node[scale=.7] at (0,1)
{$8$}; \node[scale=.7] at (-1,2) {$10$};}\qquad
\tikz[baseline=-20pt,xscale=.25,yscale=-1,thick]{\draw[wei] (0,0) -- node[below,at
  end]{$0$} (0,1); \draw[wei] (9,0) --  node[below, at end]{$\nicefrac{1}{2}$}
  (9,1);\draw (-1,0) to[out=10,in=-170]  node[above, at start]{$\nicefrac{1}{2}$}
  (10,1); \draw (1,0) to[out=10,in=-170] node[above, at start]{$0$} (8,1);\draw[dashed] (-5,0) to[out=10,in=-170] (6,1); \draw[dashed] (-3,0) to[out=10,in=-170] (4,1);}
\]
\[\tikz[baseline=5pt,scale=.3,thick]{\draw (0,0) --(2,2); \draw
  (0,0) --(-1,1); \draw (1,3) --(2,2); \draw (1,3) --(-1,1); \draw
  (1,1) -- (0,2); \draw[very thick] (-.7,.4) -- (0,-.3) -- (.7,.4); 
\draw[very thick] (3.3,.4) -- (4,-.3) -- (4.7,.4); \node[scale=.7] at (0,1)
{$0$}; \node[scale=.7] at (1,2) {$6$};
}\qquad 
\tikz[baseline=-20pt,xscale=.25,yscale=-1,thick]{\draw[wei] (0,0) -- node[below,at end]{$0$} (0,1); \draw[wei] (9,0) --  node[below, at end]{$\nicefrac{1}{2}$} (9,1);\draw (6,0) --  node[above, at start]{$\nicefrac{1}{2}$} (6,1); \draw (1,0) -- node[above, at start]{$0$} (1,1);\draw[dashed] (-3,0) -- (-3,1); \draw[dashed] (2,0) -- (2,1);}
\qquad \qquad 
\tikz[baseline=5pt,scale=.3,thick]{\draw (0,0) --(2,2); \draw
  (0,0) --(-1,1); \draw (1,3) --(2,2); \draw (1,3) --(-1,1); \draw
  (1,1) -- (0,2); \draw[very thick] (-.7,.4) -- (0,-.3) -- (.7,.4); 
\draw[very thick] (3.3,.4) -- (4,-.3) -- (4.7,.4); \node[scale=.7] at (0,1)
{$0$}; \node[scale=.7] at (1,2) {$10$};}\qquad 
\tikz[baseline=-20pt,xscale=.25,yscale=-1,thick]{\draw[wei] (0,0) -- node[below,at
  end]{$0$} (0,1); \draw[wei] (9,0) --   node[below,at
  end]{$\nicefrac{1}{2}$}  (9,1);\draw (6,0)
  to[out=15,in=-165]  node[above, at start]{$\nicefrac{1}{2}$} (10,1); \draw (1,0) -- node[above, at start]{$0$} (1,1);\draw[dashed] (2,0) to[out=15,in=-165] (6,1); \draw[dashed] (-3,0) -- (-3,1);}
\]
\[
\tikz[baseline=5pt,scale=.3,thick]{\draw (0,0) --(1,1); \draw
  (0,0) --(-1,1); \draw (0,2) --(-1,1); \draw
  (1,1) -- (0,2); \draw[very thick] (-.7,.4) -- (0,-.3) -- (.7,.4); 
\draw[very thick] (3.3,.4) -- (4,-.3) -- (4.7,.4); \draw (4,0)
--(3,1)-- (4,2) -- (5,1)--cycle;\node[scale=.7] at (0,1)
{$0$}; \node[scale=.7] at (4,1) {$10$};
}
\qquad 
\tikz[baseline=-20pt,xscale=.25,yscale=-1,thick]{\draw[wei] (0,0) -- node[below,at end]{$0$} (0,1); \draw[wei] (9,0) --  node[below, at end]{$\nicefrac{1}{2}$} (9,1);\draw (10,0) --  node[above, at start]{$\nicefrac{1}{2}$} (10,1); \draw (1,0) --node[above, at start]{$0$} (1,1);\draw[dashed] (-3,0) -- (-3,1); \draw[dashed] (6,0) -- (6,1);}
\qquad \qquad 
\tikz[baseline=5pt,scale=.3,thick]{\draw (0,0) --(1,1); \draw
  (0,0) --(-1,1); \draw (0,2) --(-1,1); \draw
  (1,1) -- (0,2); \draw[very thick] (-.7,.4) -- (0,-.3) -- (.7,.4); 
\draw[very thick] (3.3,.4) -- (4,-.3) -- (4.7,.4); \draw (4,0)
--(3,1)-- (4,2) -- (5,1)--cycle;\node[scale=.7] at (0,1)
{$8$}; \node[scale=.7] at (4,1) {$10$};
}
\qquad 
\tikz[baseline=-20pt,xscale=.25,yscale=-1,thick]{\draw[wei] (0,0) -- node[below,at
  end]{$0$} (0,1); \draw[wei] (9,0) -- node[below, at end]{$\nicefrac{1}{2}$}
  (9,1);\draw (10,0) --  node[above, at start]{$\nicefrac{1}{2}$} (10,1); \draw (1,0)
  to[out=15,in=-165] node[above, at start]{$0$} (8,1);\draw[dashed] (-3,0) to[out=15,in=-165] (4,1); \draw[dashed] (6,0) -- (6,1);}\]
\[
\tikz[baseline=5pt,scale=.3,thick]{\draw (0,0) --(1,1); \draw
  (0,0) --(-1,1); \draw (0,2) --(-1,1); \draw
  (1,1) -- (0,2); \draw[very thick] (-.7,.4) -- (0,-.3) -- (.7,.4); 
\draw[very thick] (3.3,.4) -- (4,-.3) -- (4.7,.4); \draw (4,0)
--(3,1)-- (4,2) -- (5,1)--cycle;\node[scale=.7] at (0,1)
{$15$}; \node[scale=.7] at (4,1) {$10$};
}\qquad \qquad
\tikz[baseline=-20pt,xscale=.25,yscale=-1,thick]{\draw[wei] (0,0) -- node[below,at end]{$0$} (0,1);
  \draw[wei] (9,0) --  node[below, at end]{$\nicefrac{1}{2}$}  (9,1);\draw (10,0)
  --  node[above, at start]{$\nicefrac{1}{2}$} (10,1); \draw (1,0) to[out=5,in=-175] node[above, at start]{$0$} (15,1);\draw[dashed] (-3,0) to[out=5,in=-175] (11,1); \draw[dashed] (6,0) -- (6,1);}
\]
\[
\tikz[baseline=5pt,scale=.3,thick]{\draw (0,0) --(-2,2); \draw
  (0,0) --(1,1); \draw (-1,3) --(-2,2); \draw (-1,3) --(1,1); \draw
  (-1,1) -- (0,2); \draw[very thick] (-.7,.4) -- (0,-.3) -- (.7,.4); 
\draw[very thick] (-3.3,.4) -- (-4,-.3) -- (-4.7,.4); 
\node[scale=.7] at (-1,2)
{$8$}; \node[scale=.7] at (0,1) {$10$};
} \qquad 
\tikz[baseline=-20pt,xscale=.25,yscale=-1,thick]{\draw[wei] (0,0) -- node[below,at end]{$0$} (0,1);
  \draw[wei] (9,0) --  node[below, at end]{$\nicefrac{1}{2}$} (9,1);\draw (10,0)
  --  node[above, at start]{$\nicefrac{1}{2}$} (10,1); \draw (8,0) to  node[above, at start]{$0$}(8,1);\draw[dashed] (4,0) to (4,1); \draw[dashed] (6,0) -- (6,1);}
\qquad \qquad \tikz[baseline=5pt,scale=.3,thick]{\draw (0,0) --(-2,2); \draw
  (0,0) --(1,1); \draw (-1,3) --(-2,2); \draw (-1,3) --(1,1); \draw
  (-1,1) -- (0,2); \draw[very thick] (-.7,.4) -- (0,-.3) -- (.7,.4); 
\draw[very thick] (-3.3,.4) -- (-4,-.3) -- (-4.7,.4); 
\node[scale=.7] at (-1,2)
{$15$}; \node[scale=.7] at (0,1) {$10$};
}
\qquad\tikz[baseline=-20pt,xscale=.25,yscale=-1,thick]{\draw[wei] (0,0) -- node[below,at end]{$0$} (0,1);
  \draw[wei] (9,0) --  node[below, at end]{$\nicefrac{1}{2}$} (9,1);\draw (10,0)
  --   node[above, at start]{$\nicefrac{1}{2}$} (10,1); \draw (8,0)
  to[out=15,in=-165]  node[above, at start]{$0$}(15,1);\draw[dashed] (4,0) to[out=15,in=-165]
  (11,1); \draw[dashed] (6,0) -- (6,1);}\]
\[
\tikz[baseline=5pt,scale=.3,thick]{\draw (0,0) --(2,2); \draw
  (0,0) --(-1,1); \draw (1,3) --(2,2); \draw (1,3) --(-1,1); \draw
  (1,1) -- (0,2); \draw[very thick] (-.7,.4) -- (0,-.3) -- (.7,.4); 
\draw[very thick] (-3.3,.4) -- (-4,-.3) -- (-4.7,.4); \node[scale=.7] at (1,2)
{$15$}; \node[scale=.7] at (0,1) {$10$};
}\qquad 
\tikz[baseline=-20pt,xscale=.25,yscale=-1,thick]{\draw[wei] (0,0) -- node[below,at end]{$0$} (0,1);
  \draw[wei] (9,0) --  node[below, at end]{$\nicefrac{1}{2}$} (9,1);\draw (10,0)
  --  node[above, at start]{$\nicefrac{1}{2}$} (10,1); \draw (15,0) to node[above, at start]{$0$} (15,1);\draw[dashed] (11,0) to (11,1); \draw[dashed] (6,0) -- (6,1);}
\]
\end{example}

\begin{proposition}
  The elements $\CST$ are homogeneous of degree $\deg(\sS)+\deg(\sT)$.
\end{proposition}
\begin{proof}
 These elements are defined as a product of homogeneous elements, and thus
 obviously homogeneous.  In order to determine the degree, we must
 count
 \begin{itemize}
 \item crossings of like-labelled black strands with degree -2: these correspond to pairs of boxes
 with the same residue which are not in the same column, such that the
 rightward one is filled with a smaller number than the leftward.  
\item crossings of like-labelled red and black strands with degree 1:
  these correspond to  pairs of boxes and nadirs of tableaux where the box is to
 the left of the nadir, but is filled with a higher number than the
 nadir's $x$-coordinate. 
\item  between strands and
   ghosts of adjacent strands with degree 1: these correspond to pairs of boxes with
 adjacent residue more than $\ck$ units apart, such that the
 rightward one is filled with a smaller number than the leftward.
 \end{itemize}

We organize counting these by the leftward box, whose residue we call
$i$; if the entry there is
$h$, we look at all boxes to the right of this one with the same or
adjacent residue.  These naturally form into strips around each
vertical line of residue  $i$. This isn't
quite true when $e=1,2$, but our argument goes through there as well,
simply noting that we double count every strip of residue $i\pm k$.  

In each such strip, there are 3 possibilities: relative to $h$ either there is an
addable box of residue $i$, a removable box of residue $i$ or
neither.  Assume for now that this strip does not lie above a nadir of
residue $i$. Then, if there is no removable or addable box, the number
of boxes with label $<h$ of residue $i$ is one less than those of
residue $i-k$ and one more than those residue $i+k$, or vice versa.
Thus, the degree contributions of the boxes of residue $i$ and those
of residue $i\pm k$ exactly cancel, and there is no total contribution
to the degree.  

If there is an addable box of residue $i$, then there is one more box
of adjacent residue than in the first case, and there is a total
contribution of 1 to the degree; if there is a removable box of
residue $i$, then  there is one fewer box
of adjacent residue than in the first case, and there is a total
contribution of -1 to the degree.

Finally,  if the strip we consider lies above a nadir of residue $i$,
then then we have one fewer adjacent box than expect, and so the
contribution to the degree is increased by 1, as we expected from the
red and black crossing.  This completes the proof.
\end{proof}

\subsection{Graded cellular structure}
\label{sec:cellular-basis}
 
Fix any set $B$ of loadings for the weighting $\vartheta$.  For a
multipartition $\xi$, let $M_B(\xi)$ be the set of all $\Bi$-tableaux on
$\xi$ for $\Bi\in B$.  The elements $C_{\sS,\sT}$ define a map
$C\colon M_B(\xi)\times M_B(\xi)\to T^\vartheta_B$, where
$T^\vartheta_B$ is the reduced steadied quotient of the weighted KLR algebra on the
  loadings $B$, and similarly for $\dalg^\vartheta_B$.
\begin{theorem}\label{th:graded-cellular}
  The algebra $\dalg^\vartheta_B$ has a cellular structure with data given by $(\cP_\ell,M_B,C,*)$.
\end{theorem} 
\begin{proof}
  Consider the axioms of a cellular algebra, as given in Definition 
 \ref{def:cellular}.  Condition (1) is manifest.

Condition (2) is that a basis is formed by the vectors $\CST$ where $\sS$ and $\sT$
range over tableaux for loadings from $B$ of the same shape.  First,
note that it suffices to prove this for any set of loadings containing
the original $B$, so we can always add new loadings.
By the graded Nakayama's lemma, it suffices to check this after base change
to $\K$.   In this case, we can essentially just transfer structure from the algebra
  $\PC^\vartheta$  using Theorem \ref{thm:Hecke-KLR}.   We have an
  isomorphism $\gamma\colon
  \PC^\vartheta_{\mathscr{D}}\otimes_{\mathscr{R}}\K\cong T^\vartheta_B$ where
  after possibly adding more loadings to $B$, we may assume that it is
  the set of all loadings on some collection of sets $\mathscr{D}$.  

Thus, any $D$-tableau for $D\in \mathscr{D}$ can be turned into a
tableau for a loading in $B$ by simply labeling points with the
content of the box they fill in the Young diagram.  This shows that
the number of $\CST$ is the same as the number of basis vectors
$\hCST$ from $\PC^\vartheta_{\mathscr{D}}\otimes_{\mathscr{R}}\K$.
Thus, it suffices to show that the $\CST$ span $T^\vartheta_B$.

First, note that when we consider $\PC^\vartheta_{\mathscr{D}}$ just
as a module over the squares, as we calculated in the proof of
\ref{lem:joint-spectrum}, action of a square is upper triangular in
the basis vectors $\hCST$: if as before $X_d$ denotes a square at
$d\in \R$, then $X_d\hCST=Q_pq^{\sigma(i-j)}\hCST+\cdots $ where
$(i,j,p)$ is the box of diagram containing $d$ and
as before, $\sigma$ denotes the sign of $\kappa$; the higher
order terms are either in higher cells, or have fewer crossings.  In
particular, replacing each $\hCST$ with its projection to this generalized
eigenspace $e_{\Bi_{\sS}}\hCST$ still yields a basis of
$\PC^\vartheta_{\mathscr{D}}\otimes_{\mathscr{R}}\K$.  Under the
isomorphism $\gamma$, this diagram is sent to a linear combination of
$\CST+\cdots$ where the other terms either have fewer crossings, or
lie in a higher cell.  This upper triangularity shows that the $\CST$
form a basis. 

Condition (3) is clear from the calculation \[\CST^*=(B_{\sS}^*B_{\sT})^*=B_{\sT}^*B_{\sS}=C_{\sT,\sS},\]

Thus, we need only check the final axiom, that for all $x$, we have an
equality
\begin{equation*}\label{eq:cell}
 x\CST\equiv \sum_{\sS'\in M_B(\xi)} r_x(\sS',\sS)C_{\sS',\sT} \tag{$\star$}
\end{equation*}
modulo the vectors associated to partitions higher in dominance
order.  The numbers $r_x(\sS',\sS)$ are just the structure
coefficients of $x^*$ acting on the basis of $S_\xi$ given by $B_\sS$.
Since we have that $ x B_{\sS}^*\equiv\sum_{\sS'\in M_B(\xi)}
r_x(\sS',\sS)B_{\sS'}$ modulo diagrams factoring through loadings that
are higher in weighted dominance order, the equation \eqref{eq:cell}
holds.  This completes the proof.
\end{proof}

It is a standard fact about cellular algebras that any projective
module over them has a cell filtration; a graded version of this is
proven by Hu and Mathas \cite[2.14]{HM}, showing that each projective $P$ has a cell
filtration where the graded multiplicity space of $S_\xi$ is $\dot{S}_\xi\otimes_{\alg^\vartheta}P$.
\begin{proposition}\label{prop:projective-filtration}
  The projective $P_\Bi$ has a standard filtration, where the graded
  multiplicity of $S_\xi$ is exactly the number of $\Bi$-tableaux on
  $\xi$, weighted by their degree.
\end{proposition}
\begin{proof}
  Since $\dot{S}_\xi\otimes_{\alg^\vartheta}P_\Bi\cong e_{\Bi}S_\xi$,
  this follows instantly from the result of Hu and Mathas mentioned above. 
\end{proof}

\begin{example}
  Let us return to the case of Example \ref{big-example-2}.   In this
  case, if we let $B$ be the collections of loadings given there,
  every simple module is 1-dimensional, and so $T^\vartheta_Be_{\Bi}$
  is already indecomposable.  Thus, the multiplicities of standard
  modules in the indecomposable projectives are easily calculated from
  the bases of standard modules given in  Example \ref{big-example-1}.

The decomposition matrix in the 3 cases are given by 
\begin{equation*}
\begin{bmatrix}
  1 & q^{-1} & q^{-2} & q^{-1} & 0\\
0 & 1 & q^{-1} & 0 & 0\\
0 & 0 & 1 & q^{-1} & q^{-2}\\
0 & 0 & 0 & 1 & q^{-1}\\
0 & 0 & 0 & 0 & 1
\end{bmatrix}
\qquad \begin{bmatrix}
  1 & q^{-2} & q^{-1} & 0 & 0\\
0 & 1 & 0 & 0 & 0\\
0 & q^{-1} & 1 & 0 & q^{-1}\\
0 & 0 & q^{-1} & 1 & q^{-2}\\
0 & 0 & 0 & 0 & 1
\end{bmatrix}
\qquad \begin{bmatrix}
  1 & q^{-1} & 0& 0& 0\\
0 & 1 & 0& 0 & 0\\
q^{-1} & q^{-2} & 1 & 0 & \\
q^{-1} & 0 & q^{-1}  & 1 & q^{-1}\\
0 & 0 & q^{-2} & 0 & 1
\end{bmatrix}
\end{equation*}
\end{example}

\subsection{Generalization to bimodules}
\label{sec:gener-bimod}

As defined in \cite{WebwKLR}, there are natural bimodules
$B^{\vartheta,\vartheta'}$ attached to each pair of weightings; these
bimodules have steadied quotients $\bra^{\vartheta,\vartheta'}$, which
are $T^{\vartheta}\operatorname{-}T^{\vartheta'}$-bimodules.  
We call the functors $\bra^{\vartheta,\vartheta'}\Lotimes-\colon
T^{\vartheta'}\mmod \to T^{\vartheta}\mmod$ {\bf
  change-of-charge functors}; these are quite interesting functors.
In particular, we will eventually show that they induce equivalences
of derived categories.

These bimodules are spanned by the KLR analogues of the WF
$\vartheta\operatorname{-}\vartheta'$ diagrams of Definition
\ref{WF-definition} with green strands replaced by black ones, and
squares by dots.  We let $\bra^{\vartheta,\vartheta'}$ be the quotient
of the span of these diagrams modulo the local relations
(\ref{dots-1}--\ref{dumb}) and the same steadying relation.  

Since this is a bimodule over a graded algebra, we expect it will be
graded.  The simplest possible choice of grading would be to simply
use the same local contributions for KLR diagrams, and assign a local
contribution of 0 for all new types of crossings.  This is actually
not the most natural choice, though. We'll instead assign a degree of
$|m|$ to a diagram where two red strands with labels $r$ and $r'$ satisfying $r=r'+mk \pmod
\Z$ have $x$-coordinates satisfying $\theta=\theta'+m\kappa$ at a
single $y$-coordinate and no other crossings (including of this type)
occur, and the strands are diverging as we read the diagram from top
to bottom, that is, we must have $|\theta-\theta'|<|m\kappa|$ at the bottom of the
diagram, and $|\theta-\theta'|>|m\kappa|$ at the top.

In this bimodule, we can construct analogues of the
elements $\CST$, which we will also denote $\CST$ by abuse of notation
(the original elements $\CST$ will be a special case of these where
$\vartheta=\vartheta'$). These are similar in form and structure to
the basis $\hC_{\sS,\sT}$ defined in Section \ref{sec:derived-equivalences}.

Let us first describe the basis which is cellular for the right module
structure.  Let $D_\sS$ be the element of the bimodule
$\bra^{\vartheta,\vartheta'}$ defined analogously with $B_{\sS}$. Its
bottom is given by $\Bi_\eta$ (for the weighting $\vartheta'$).  Its
top is given by the entries of $\sS$, with each entry determining the position
on the real line of a point in the top loading, labeled with the root
associated to that box.  The diagram proceeds by connecting the points
in the loading associated to the same box in the top and bottom, while
introducing the smallest number of crossings.  As usual, this diagram
is not unique; we choose any such diagram and fix it from now on.

\begin{definition}
  The right cellular basis for $e_{\Bi}\bra^{\vartheta,\vartheta'}e_{\Bj}$ is given
  by $D_\sS B_{\sT}^*$ for $\sS$ an $\Bi$-tableau for some
loading $\Bi$ and the weighting $\vartheta$ (upon which the definition of
$\Bi$-tableau depends),  and $\sT$ a $\Bj$-tableau for some loading
$\Bj$ and the weighting  $\vartheta'$.

The left cellular basis for
$e_{\Bj}\bra^{\vartheta',\vartheta}e_{\Bi}$ is given by the
reflections of these vectors, that is by $B_{\sT}D_{\sS}^*$.  
\end{definition}

\excise{
If both tableaux are the tautological tableau, then $\CST$ is an
element that essentially interpolates between them; it is just the diagram where a straight
line connects the points in the two loadings that correspond to the
same box in the diagram, with triple points resolved arbitrarily. Fix $d^{\vartheta,\vartheta'}_\xi=\deg(C_{\operatorname{taut}}^\xi)$

\begin{lemma}
 $\displaystyle\deg(\CST)=\deg_{\vartheta'}(\sS)+d^{\vartheta,\vartheta'}_\xi+\deg_{\vartheta}(\sT).$
 \qed
\end{lemma}}

\begin{lemma}\label{lem:between}
The vectors $D_\sS B_{\sT}^*$ are a basis for the  bimodule
$\bra^{\vartheta,\vartheta'}$.  Furthermore, the sum of vectors
attached to partitions $\leq \xi$ in $\vartheta'$-weighted order is a
right submodule. In particular, as a
right module, $\bra^{\vartheta,\vartheta'}$ is standard filtered.

Similarly, the left cellular basis shows that the bimodule
$\bra^{\vartheta,\vartheta'}$ is standard filtered as a left module.
\end{lemma}

\begin{proof}
  First, we wish to show that these elements are a basis.  This follows from
  Lemma \ref{lem:hbim-basis} by the same argument as the proof of
  Theorem \ref{th:graded-cellular}.  That they are standard filtered follows
  from the calculation 
\begin{equation*}
 B_{\sT}^*x\equiv \sum_{\sS'\in M_B(\xi)} B_{\sT'}^*r_x(\sT',\sT) +\cdots 
\end{equation*}
where the additional terms are in higher cells; multiplying on the
left by $D_{\sS}$, we obtain the desired result.
\end{proof}

\section{The structure of the categories}
\label{sec:structure-categories}

\subsection{Highest weight categorifications}
\label{sec:high-weight-categ}

We let $\mathcal{S}^\vartheta_\nu$ denote the category of
finite dimensional representations of the reduced steadied quotient
$T^\vartheta_\nu$; we let $\mathcal{S}^\vartheta$ denote the sum of
these over all $\nu$.  As shown in Corollary \ref{biadjoint}, if
$e\neq 1$, this category carries a
categorical $\gu$-action induced from that on projective modules.
We'll use $\eE_u$ and $\eF_u$ to denote the transport of these functors
to the algebras $\alg^\vartheta,\dalg^\vartheta$.

This categorical action also has a natural diagrammatic description,
given in \cite[Thm 3.1]{WebwKLR}.  Under the isomorphism of Theorem 
\ref{thm:Hecke-KLR}, the eigenvalues of the semi-simple part of $X$
translate into the labels on black strands.  Thus, $u$-induction/restriction
corresponds to the bimodule which adds/removes a strand at the far
right, which is fixed to have label $u$.  
That is, given a set of loadings $B$ with each of which has $m$
points, and set of loadings $C$, each of which has $m+1$ points, we
can define $B'$ to be the loadings where we take $\Bi$, and add a
point at with label $u$ at $x=s$, where $s$ is much greater than any other point
appearing in any of the loadings in $B$.  We let $e_B,e_{B'},e_C$ be
the sum of idempotents corresponding to these loadings.  Applying the
isomorphism to the definition of $\operatorname{ind}$ and
$\operatorname{res}$ in Section \ref{sec:relat-hecke-algebra}, we see that:
\begin{lemma}
  We have  isomorphisms of functors  $e_C
  \dalg^\vartheta e_{B'}\otimes_{\dalg^\vartheta_B} -\cong  \eF_u$ and $\Hom_{\dalg^\vartheta_C }(e_{C}
  \dalg^\vartheta e_{B'},-)\cong  \eE_u$
  where $e_C
  \dalg^\vartheta e_{B'}$ is made into a 
  $\dalg^\vartheta_C$-$\dalg^\vartheta_B$ bimodule by the obvious left
  action, and right action only on the leftmost $m$ strands.
\end{lemma}
If $B'\subset C$, then we can immediately see that
$\eE_u(M)=e_{B'}M$; of course, if $C$ is sufficiently large, its
Morita equivalence class will not be changed by adding any missing
elements of $B'$.

Since $e_C
  \dalg^\vartheta e_{B'}$ is a graded module, this allows us to define
  a graded lift of $\eF_i$ and $\eE_i$.  We'll use the obvious grading
  on $\eF_i$, and shift the obvious grading on $\eE_i$ acting on a
  module of weight $\mu$ downward by $\al_i^\vee(\mu)+1$.  The right
  adjoint to $\eF_i$ is $\eE_i(-\al_i^\vee(\mu)-1)$ (that is, the obvious grading above), and the
  left adjoint is $\eE_i(\al_i^\vee(\mu)+1)$.  This is a consequence of the main theorem of
  \cite{Brundandef}, in particular, of the form the adjunctions
  defined in \cite[(1.16-17)]{Brundandef}.

It follows immediately from Lemma \ref{lem:cell-hw} that:
\begin{proposition}\label{prop:highest-weight}
  The category $\mathcal{S}^\vartheta_\nu$ is highest weight with
  standards $S_\xi$ and 
  partial order given by weighted dominance order.  The category 
$\dalg^\vartheta_\nu\mmod$ is also highest weight, in the sense given
by Rouquier \cite[\S 4.1.3]{RouqSchur}.
\end{proposition}
 \excise{As shows, the same argument shows the
category 
\begin{proof}
The category $\mathcal{S}^\vartheta_\nu$ is the category of representations of a
  cellular algebra, with the partial order on cells given by weighted
  dominance order.  Thus we need only show that the natural inner
  product on each cell module is non-zero.  Obviously, the vector
  corresponding to the superstandard tableau has inner product 1 with
  itself.  Thus, the standard modules are the standard modules of a
  highest weight structure.
\end{proof}}

\begin{lemma}\label{lem:FS-filtration}
  The module $\eF_iS_\xi$ carries a filtration $M_1\subset M_2\subset \cdots
  \subset M_j$ indexed by addable boxes of residue $i$
  in $\xi$ from left to right.  The quotient $M_h/M_{h-1}$ is
  $S_{\xi(h)}(\deg(\sT_h))$, where $\xi(h)$ is $\xi$ with the $h$th addable box of
  residue $i$ added, and $\sT_h$ is obtained by putting the
  tautological tableau in $\xi$, and $s\gg 0$ in the corresponding addable box.
\end{lemma}
\begin{proof}
  We induct on the partial order; if $\xi$ is maximal, then $S_{\xi}=P_{\Bi_\xi}$
  and the only $\Bi_\xi$-tableau on $\xi$ is the tautological one.
  Thus, the result follows from Proposition
  \ref{prop:projective-filtration} in this case.

  Now, we induct.  The module $P_{\Bi_{\xi}}$ has a standard
  filtration, with multiplicity
  given by counting $\Bi_\xi$-tableaux of a given shape; those which
  are not tautological correspond to the kernel of the map
  $P_{\Bi_{\xi}}\to S_\xi$.  Since $\eF_i$ is exact, $\eF_iP_{\Bi_{\xi}}$ is
  filtered by the images under $\eF_i$ of these standards.  On the other
  hand, $\eF_iP_{\Bi_{\xi}}$ is still a projective module over a
  quasi-hereditary algebra, and thus has a canonical standard
  filtration, which has multiplicities given by the numbers of $\Bi_\xi\circ
  i$-tableaux of a given shape.   Thus by the inductive hypothesis, the kernel $K$ of the map $\eF_iP_{\Bi_\xi}\to
  \eF_iS_\xi$ has a standard filtration where the multiplicities of a
  given shape correspond to the  $\Bi_\xi\circ
  i$ tableaux which are not a tautological tableau on $\xi$ with a box with
  entry $s$ added.

The module $\eF_iP_{\Bi_{\xi}}$ also has a cellular basis; the basis
vectors are $\CST$, where $\sT$ is a $\Bi_\xi\circ
  i$-tableau.  If the entries from $\Bi_\xi$ fit into any shape $\xi'$
  other
  than $\xi$ (necessarily higher in dominance order), this basis
  vector is killed by the map $\eF_iP_{\Bi_\xi}\to
  \eF_iS_\xi$.  
Thus,   $\eF_iS_\xi$ is spanned by the remaining basis vectors where 
  $\sT$ is the tautological tableau of $\xi$ with a box added; the
  dimension count above shows that these are a basis.  Furthermore, we can define a filtration compatible
  with this basis given by the span $M_h$ of vectors
  where the new box on $\sT$ is equal to or left of the $h$th addable
  box; this is a submodule by the cellular multiplication property \ref{eq:cell}.

This defines the desired filtration, and we have an isomorphism
$S_{\xi(h)}(\deg(\sT_h))\to M_h/M_{h+1}$ sending the basis vector
$B_{\sS}$ to the 
basis vector $C_{\sS,\sT_h}$.
\end{proof}
For simplicity, we let $\delta_h$ denote $\deg(\sT_h)$; is precisely the
number of $i$-addable boxes right of the $h$th, minus the number of
such which are removable.  On the other hand, let $\delta^h$ denote the
number of $i$-removable boxes {\it left} of the $h$th, minus the
number of such which are removable.

Note that we have $\dot{S}_\eta\overset{L}\otimes (\eF_i) S_{\xi}\cong
\dot{(\eE_iS_\eta)}\overset{L}\otimes S_\xi$.  Combining this with
the usual criterion that $M$ is a standard filtered if and only if
$\operatorname{Tor}^i(\dot M, S_\xi)=0$ for all $\xi$, this shows that
$\eE_iS_\xi$ is standard filtered.

The functors $\eE_i$ and $\eF_i$ are biadjoint up to shift.  Thus they also commute
with duality.  The result above also implies that:
\begin{corollary}\label{cor:ES-filtration}
\hfill
  \begin{enumerate}
 \item The module $\eE_iS_\xi$ carries a filtration $N_m\subset N_{m-1}\subset \cdots
  \subset N_j$ indexed by removable boxes of residue $i$
  in $\xi$ from left to right.  The quotient $N_h/N_{h+1}$ is
  $S_{\xi\{h\}}(\delta^h)$, where $\xi\{h\}$ is $\xi$ with the $h$th removable box of
  residue $i$ removed, and $\sT_h$ is obtained by putting the
  tautological tableau in $\xi$, and $\infty$ in the new box.
  \item The module $\eF_iS_\xi^\star$ carries a filtration $O_j\subset O_{j-1}\subset \cdots
  \subset O_1$ indexed by addable boxes of residue $i$
  in $\xi$ from left to right.  The quotient $O_h/O_{h+1}$ is
  $S_{\xi(h)}^\star (-\delta_h)$.
\item  The module $\eE_iS_\xi^\star$ carries a filtration $Q_1\subset Q_{2}\subset \cdots
  \subset Q_m$ indexed by addable boxes of residue $i$
  in $\xi$ from left to right.  The quotient $Q_h/Q_{h-1}$ is
  $S_{\xi\{h\}}^\star (-\delta^h)$.
  \end{enumerate}
\end{corollary}

Losev has defined a notion of a {\bf highest-weight categorification} \cite[4.1]{LoHWCI};
this consists of the data of a 
\begin{enumerate}
\renewcommand{\labelenumi}{(\roman{enumi})}
\item a highest weight category $\mathcal{C}$ with index set $\Lambda$
  for its simples/standards/in\-de\-com\-posable projectives, together with
  a function $c\colon \Lambda\to \C$
\item a partition of $\Lambda$ into subsets $\Lambda_a$ with index set
  $\mathfrak{A}$ 
\item integers $n_a$ for each $a\in \mathfrak{A}$ and a
  function $d_a\colon \{1,\cdots,n_a\}\to \C$.
\item an isomorphism $\sigma_a\colon \{+,-\}^{n_a}\to \Lambda_a$,
  identifying $\Lambda_a$ with signed sequences of length $n_a$.
\end{enumerate}
Now, consider the highest weight category $\mathcal{S}^\vartheta$; we
aim to show that it is, in fact, a highest weight categorification in
the sense above.  The combinatorics of this structure are almost
exactly the same as those described by Losev for rational Cherednik
algebras \cite[\S 3.5]{LoHWCI}.
\begin{enumerate}
\renewcommand{\labelenumi}{(\roman{enumi})}
\item The indexing set $\Lambda=\cP_\ell$ is the set of
  $\ell$-multipartitions, and the function $c$ is the sum over all boxes of the partition of
  the $x$-coordinate of the box. 
\item The set $\mathfrak{A}$ is the set of partitions with no
  removable boxes of residue $i$, and  $\Lambda_a$ is the set of
  all partitions that contain $a$ with only boxes of residue $i$
  added.  
\item  The number of addable boxes of
  residue $i$ is $n_a$. The function $d_a$ records, from the left to
  right, the $x$-coordinates of the addable boxes.
\item The isomorphism $\Lambda_a\to \{+,-\}^{n_a}$ sends a partition $\xi$
to the sign vector where the first sign is $+$ if the leftmost addable
box of residue $i$ in $a$ is present in  $\xi$ and $-$ if it is not, and similarly
for the other addable boxes in order from left to right.  
\end{enumerate}
\begin{theorem}\label{thm:HWC}
When $e\neq 1$,  the categorical $\gu$-module  $\mathcal{S}^\vartheta$  is a
  highest weight categorification in the sense of Losev.
\end{theorem}

\begin{proof}
Let us consider the conditions from Losev's definition
\cite[4.1]{LoHWCI}:
\begin{enumerate}
\renewcommand{\theenumi}{HW\arabic{enumi}}
\setcounter{enumi}{-1}
\item We must show that $\eF_i$ and $\eE_i$ preserve the categories of
standard filtered objects;
by exactness, we need only check that the image of standards has a
standard filtration.  This follows from Lemma \ref{lem:FS-filtration}
and Corollary \ref{cor:ES-filtration}.
\item We must show that $\xi <\xi'$ implies that the sum of
  $x$-coordinates for $\xi$ is higher than that for $\xi'$.  This is
  clear from the definition of weighted dominance order, Definition
  \ref{def:dominance-order}.
\item We must show that the images $\eF_i S_\xi$ and $\eE_iS_\xi$ have
certain classes in the Gro\-then\-dieck group, which are exactly those
determined by  Lemma \ref{lem:FS-filtration}
and Corollary \ref{cor:ES-filtration}. 
\item We must show that changing the signs of the $k$th entry in $\{+,-\}^{n_a}$,
  which corresponds to adding or removing a box changes the sum of the
  $x$-coordinates by $\pm d_a(k)$.  Since $d_a(k)$ is the
  $x$-coordinate of the box added or removed, this is clear.  
\item  We must have $d_a(1)<d_a(2)<\cdots <d_a(n_a)$; this is simply a
  restatement of the fact that we read the boxes from left to right.
\end{enumerate}
This completes the proof.  
\end{proof}
\begin{remark}
  Annoyingly, Losev gives slightly different definitions of a highest
  weight categorification in the papers \cite{LoHWCI,LoHWCII}; we have
  used that of \cite{LoHWCI}.  In \cite{LoHWCII}, a stronger condition
  is imposed on the poset involved: it must carry a {\bf hierarchy}
  structure as defined in \cite[\S 3.1]{LoHWCII}.  The hierarchy structures on multipartitions discussed in
  \cite[\S 3.2]{LoHWCII} can easily be modified to apply in our
  situation as well, so Theorem \ref{thm:HWC} holds for either definition.
\end{remark}

Each simple module is the unique simple quotient of a unique standard module,
so we can index these by multipartitions as well; we denote the simple
quotient of $S_\xi$ by $L_\xi$, and its projective cover by
$P_\xi$. These simple modules (and also the projectives) carry a
natural crystal structure for $\gu$, induced by taking the unique
simple quotient of $\eF_iL_\xi$ or $\eE_iL_\xi$.  This gives a
crystal structure on multipartitions determined by the weighting
$\vartheta$.  

\begin{definition}
  The {\bf $\boldsymbol \vartheta$-weighted crystal structure} on the space of
  $\ell$-multipartitions is defined as follows:
  drawing the partitions in Russian style, one places a close
  parenthesis over each addable box of residue $i$, and an open
  parenthesis over each removable box of residue $i$.
  \begin{itemize}
  \item The Kashiwara operator $\tilde{e}_i$ removes the box under the leftmost
    uncancelled open parenthesis and sends the partition to 0 if there
    is no uncancelled open parenthesis.
\item The Kashiwara operator $\tilde{f}_i$ adds a box
    under the rightmost uncancelled closed parenthesis and sends the partition to 0 if there
    is no uncancelled closed parenthesis.
  \end{itemize}
\end{definition}
In the Uglov case, this crystal structure is precisely
that described by Tingley \cite[3.2]{Ting3} in terms of abaci; in
general, this crystal will coincide with that of the Uglovation. It
follows immediately from \cite[5.1]{LoHWCI} that:
\begin{corollary}
The map sending a multipartition to $L_\xi$ intertwines the  $
\vartheta$-weighted crystal structure with that defined by the
categorification functors.
\end{corollary}

\subsection{Decategorification}
\label{sec:decategorification}

With this cellular basis in hand, we can extend all the results
showing how quiver Schur algebras categorify Fock spaces to this more
general case.

For our purposes, the {\bf Fock space $\mathbf{F}_\vartheta$ of level $\ell$} is the
$\C[q,q^{-1}]$ module freely spanned by $\ell$-multipartitions.  For
each multi-partition $\xi$, we denote the corresponding vector
$u_\xi$.  When our weighting is Uglov, we will also use the notation
$\mathbf{F}_\Bs$.  
Now, we choose weighting for our partitions; as before, this
corresponds to choosing a weighting on $U_{\Bw}$, with all edges
in the cycle given weight $\ck$, and an ordering on the new edges (which
is arbitrary), to put them in bijection with the constituents of the
multipartition.

Let $A=\C[q,q^{-1}]$.  We can also realize $\mathbf{F}_\vartheta$ as a subspace of a
semi-infinite wedge product, by what is sometimes called the
{\bf boson-fermion correspondence}.  For simplicity of notation, we assume
that $\ck>0$; this suffices, since we can cover the case of $\ck<0$ using
a symmetry as in Proposition \ref{prop:symmetry}.  For each
$r\in \C/\Z$ and $\vartheta\in\R$ , we let $A^\infty_{r,\vartheta}$
denote the free $A$-module with basis $w_i$ indexed by $i\in \Z$.
While this space does not depend on $r$ or $\vartheta$, we'll define
some auxilliary structures which do.  Each vectors has an {\bf
  $x$-coordinate} given by $\vartheta+i\ck $, and a {\bf residue}
given by $r+ik\in \C/\Z$.  The space $A^\infty_{r,\vartheta}$ has a
natural action of $\fg_U$, with
\[E_j\cdot w_i=
\begin{cases}
  w_{i+1} & r+ik\equiv j \pmod \Z\\
  0 & \text{otherwise}\end{cases}
\qquad F_j\cdot w_i=
\begin{cases}
  w_{i-1} & r+(i-1)k\equiv j \pmod \Z\\
  0 & \text{otherwise}\end{cases}
\]
In the case where $k=1/e$ and $r\in \nicefrac{1}{e}\Z$, we typically
identify $A^\infty_{r,\vartheta}$ with $A^e[t,t^{-1}]$ by sending identifying
$w_0,w_{1/e},\dots, w_{(e-1)/e}$ with the usual basis of $A^e$, and
letting $w_{i\pm 1}=w_it^{\pm 1}$.  

We can construct a semi-infinite wedge space
$\bigwedge{}^{\!\!\!\!\nicefrac{\infty}2}A^\infty_{r,\vartheta}$,
spanned by wedges of the form $v_{\xi_1}\wedge v_{\xi_{2}-1}\wedge
v_{\xi_3-2}\wedge v_{\xi_4-3}\wedge \cdots $ for some partition
$\xi$.  Ordered wedges form a basis of this space, but we must
exercise some care about the meaning of unordered wedges.  These are
calculated using the straightening rules, which we only write here in
the case $\ck >0$.  If $m<n$, we let $g=\lfloor
      \frac{n-m}{e}\rfloor$.  We have $w_m\wedge w_n=-w_n\wedge w_m$
      if $k(m-n)\in \Z$, and if $k(m-n)\notin \Z$
\begin{equation}
  \label{eq:straightening}
    w_m\wedge w_n=-q^{-1}w_n\wedge w_m +(1-q^{-2})\Big(\sum_{p=1}^{g}
    q^{-2p+1}w_{n-ep}\wedge w_{m+ep}-q^{-2p}w_{m+e(p+g)}\wedge
    w_{n-e(p+g)}\Big).
\end{equation}

Given a weighting, we can consider the direct sum $
A^\infty_{\mathbf{r},\boldsymbol{\vartheta}}=\oplus_{i=1}^\ell
A^\infty_{r_i,\vartheta_i},$ and consider the semi-infinite wedge space
$\bigwedge{}^{\!\!\!\!\nicefrac{\infty}2}A^\infty_{\mathbf{r},\boldsymbol{\vartheta}}$,
where now we order wedges according to the $x$-coordinates of the
vectors.  These require more complicated straightening rules, based on
\cite[Prop. 3.16]{Uglov}; the only important fact about these rules is
that it replaces a pair of basis vectors with another pair with
$x$-coordinates between this pair.  We have a natural isomorphism
$\mathbf{F}_\vartheta\cong
\bigwedge{}^{\!\!\!\!\nicefrac{\infty}2}A^\infty_{\mathbf{r},\boldsymbol{\vartheta}}$
sending the basis vectors to the ordered wedges.  Note that
the module structure on this semi-infinite wedge does not depend on
$\vartheta$, but it will change the choice of preferred basis.

In the Uglov case, $\ell$ and $e$ actually play a symmetric role;
since $U$ is connected, all the representations $A^\infty_r$ are
isomorphic, and 
\[A^\infty_{\mathbf{r},\boldsymbol{\vartheta}}\cong
A^\infty_0\otimes A^{\ell}\cong A^e\otimes A^\ell[t,t^{-1}].\]  Thus,
in this case, we have a commuting action of $U_q(\sllhat)$ on
$A^\infty_{\mathbf{r},\boldsymbol{\vartheta}}$.  We can use this to
define commuting actions of $U_q(\gu\cong \slehat)$ and $U_q(\sllhat)$
on the wedge powers of this representation; note that as discussed in
\cite{Uglov}, Uglov uses different coproducts for the two algebras.
However, on the semi-infinite wedge space, the induced
action does not preserve the semi-infinite conditions we have fixed.
Rather, the operators $E_i$ and $F_i$ induce maps
$\mathbf{F}_\Bs\to \mathbf{F}_{\Bs\pm \al_i}$ where $\al_i=(0,\dots,
1,-1,0,\dots)$ with $1$ in the $i$th coordinate; in fact $\Bs$ is the
weight for the action of the derived subalgebra of $\sllhat$.  

In the general case, we can divide $U$ into its connected components,
with the weighting connecting each component of the multipartition to
a component of $U$.  This allows us to think of $\mathbf{F}_\vartheta$
as a tensor product
\begin{equation}
\mathbf{F}_\vartheta\cong \mathbf{F}_{\vartheta_1}\otimes
\mathbf{F}_{\vartheta_2}\otimes\cdots
\otimes\mathbf{F}_{\vartheta_p}\label{eq:Fock-tensor}
\end{equation}
over the components of $U$.

In the Uglov case, the Fock space $\mathbf{F}_\vartheta$ has an
anti-linear bar involution $u\mapsto \bar{u}$, defined in
\cite[(39)]{Uglov}; up to sign and factors of $q$, this is defined on each
wedge by reversing the order of the variables in the wedge product,
and then applying the straightening rule to return to the usual order
(as in \cite[Prop. 3.23]{Uglov}).  For a general Fock space, we extend
the bar involution on each tensor factor in \eqref{eq:Fock-tensor} to
the tensor product.

The affine Lie algebra $U_q(\gu)$ acts in a natural way on this
higher level Fock space.   We let 
\begin{equation}
F_iu_\xi =\sum_{\operatorname{res}(\eta/\xi)=i}q^{-m(\eta/\xi)}
u_{\eta} \qquad E_iu_\xi =\sum_{\operatorname{res}(\xi/\eta)=i}q^{n(\xi/\eta)}
u_{\eta}. \label{eq:F-E-action}
\end{equation}
As usual,
\begin{itemize}
\item the sums are over all ways of adding (resp. removing a box) of
  residue $i$, 
\item $m(\eta/\xi)$ is the number of addable boxes of residue
  $i$ {\it right} of the single box $\eta/\xi$ minus the number of such
  boxes which are removable, and 
\item $n(\xi/\eta)$ is the number of
  addable boxes of residue $i$ {\it left} of the single box $\xi/\eta$
  minus the number of such boxes which are removable.
\end{itemize}
Note that as long as the weights of the partitions are generic, no two
addable or removable boxes will be at the same horizontal position, so
for each pair, the first is left of the second, or {\it vice versa}.
Note also that if  $U$ is disconnected, then the tensor product
decomposition of \eqref{eq:Fock-tensor} gives the module structure as
an outer tensor product over the Kac-Moody algebras (each isomorphic
to $\slehat$) corresponding to the different components of $U$.

\begin{theorem}\label{Fock-space}
  The Grothendieck group of the category of representations of 
  $\alg^\vartheta_\nu$ is isomorphic as a $U_q(\gu)$
  representation to the corresponding level $\ell$ Fock space under
  the isomorphism $ [S_{\xi}]\mapsto u_\xi$.  In particular, we have
  that $[P_\Bi]$ maps to the sum over $\ell$-multipartitions of the 
  graded multiplicity of $\Bi$-tableaux.
\end{theorem}
\begin{proof}
 The classes $[S_\xi]$ are a basis of the Grothendieck group
  because $\cS^\vartheta$ is highest weight.  Thus, we need only
  check how categorification functors act on these classes, which
  follows immediately from comparing Lemma \ref{lem:FS-filtration} and Corollary
  \ref{cor:ES-filtration} with \eqref{eq:F-E-action}.
\end{proof}
  
The $q$-Fock space $\mathbf{F}_\vartheta$ has a natural symmetric bilinear
form
$(-,-)$ where the $u_\xi$ are an orthonormal basis.  Furthermore, it can be endowed with a
sesquilinear form by \[\langle  u, v\rangle:=(\bar u,v).\]

On the other hand, the Grothendieck group $K_q^0(T^\vartheta)$ also
carries canonical bilinear and sesquilinear forms: we
let \[([M],[N])=\dim_q(\dot{M}\Lotimes N)\qquad \langle M,N\rangle=\dim_q\R\Hom(M,N).\]

\begin{proposition}\label{prop:symmetric-forms}
  Under the isomorphism $\mathbf{F}_\vartheta\cong K_q^0(T^\vartheta)$, the forms $(-,-)$ match.
\end{proposition}
\begin{proof}
  We need only check that they are correct on standard modules; this
  follows from the orthonormality of the classes $S_\xi$. 
\end{proof}

While the notation suggests that the forms $\langle -,-\rangle$ will
coincide as well, this is not an easy statement to prove.  It is one
of the consequences of Proposition \ref{prop:bar}.

\subsection{Change-of-charge functors: KLR case}
\label{sec:twisting-functors}

The bimodules $\bra^{\vartheta,\vartheta'}$ induce functors between
the categories $\cS^{\vartheta}$ and $\cS^{\vartheta'}$.  We call the
groupoid of functors generated by these {\bf change-of-charge
  functors}. One should see these as 
analogous with the twisting functors on category $\cO$; this connection
can be made precise by realizing $\cS^{\vartheta}$ as a version of
``category $\cO$'' for an affine quiver variety.

These functors are particularly useful since they show that up to
derived equivalence, all the categories $\cS^{\vartheta}$ only depend
on $\la$, up to derived equivalence.  They thus allow us to transport
structure from one category to another.

\begin{lemma}\label{lem:far-right-dual}
  The functor $\bra^{\vartheta,-\vartheta}\Lotimes -$ sends projective
  modules to tilting modules and tilting modules to injective modules.
\end{lemma}
\begin{proof}
  We already know that  $\bra^{\vartheta,-\vartheta}e_{\Bi}$ is standard
  filtered as a left module by Lemma \ref{lem:between}, so if we prove
  it is self-dual, that will show it is tilting.  

For a fixed loading $\Bi$, 
choose a basepoint $b$ which is less that $b_i$ for all $i$, and a
real number $\gamma\gg 0$, sufficiently large so that the loading
$\Bi'$ where we move the points of the loading by the automorphism of $\R$ given by
$x\mapsto \gamma(x-b)+b$ is Hecke (it always will be for $\gamma$
sufficiently large).  There is a natural (generating) element 
$g_{\Bi}\in e_{\Bi}T^{-\vartheta}e_{\Bi'}$ and similarly for
$g_{\Bj}\in e_{\Bj}T^{\vartheta}e_{\Bj'}$.

Each vector in the basis $\CST$ for the bimodule
$e_{\Bj}\bra^{\vartheta,-\vartheta}e_{\Bi}$ factors as
$g_{\Bj}C_{\sS',\sT'}g_{\Bi}^*$ for $\sS',\sT'$ the obvious associated
tableaux of type $\Bj'$ and $\Bi'$.  Thus we have a surjective maps \[\pi \colon
e_{\Bj'}\bra^{\vartheta,-\vartheta}e_{\Bi'}\to
e_{\Bj}\bra^{\vartheta,-\vartheta}e_{\Bi}\qquad \pi(a)= g_{\Bj}a
g_{\Bi}^*. \]  Similarly, as we range over all $\sS,\sT$, the elements
$g_{\Bj}^*C_{\sS,\sT}g_{\Bi}$ are linearly independent, giving an
injective map
\[\iota \colon e_{\Bj}\bra^{\vartheta,-\vartheta}e_{\Bi}\to e_{\Bj'}\bra^{\vartheta,-\vartheta}e_{\Bi'}\qquad \iota(b)= g_{\Bj}^*b g_{\Bi}.\]

For two elements
$g_{\Bj}ag_{\Bi}^*$ and $g_{\Bj}bg_{\Bi}^*$, for $a,b\in
e_{\Bj'}\bra^{\vartheta,-\vartheta}e_{\Bi'}$, we define a pairing
\[\langle
g_{\Bj}ag_{\Bi}^*,g_{\Bj}bg_{\Bi}^*\rangle=\tau(a^*g_{\Bj}^*g_{\Bj}bg_{\Bi}^*g_{\Bi})=\tau(g_{\Bi}^*g_{\Bi}a^*g_{\Bj}^*g_{\Bj}b)\]
where $\tau \colon e_{\Bi'}T^{-\vartheta}e_{\Bi'}\cong e_{\Bi'}T^\la
e_{\Bi'}\to \K$ is the Frobenius trace of \cite[2.26]{Webmerged} if
$e\neq 1$ (we abuse
notation and also use $\Bi$ to denote the unloading of this loading).
 If $e=1$, then we can use an explicit trace on $\K[S_m]\wr
 \K[x]/(x^\ell)$.  
As noted in \cite[2.27]{Webmerged}, we can modify this trace to make it
symmetric; it is a bit more convenient to use this less-canonical
trace, but symmetric, trace.

This pairing is well defined, since if $b$ is in the kernel of $\pi$,
then \[a^*g_{\Bj}^*g_{\Bj}bg_{\Bi}^*g_{\Bi}=a^*g_{\Bj}^*\cdot 0\cdot
g_{\Bi}=0;\] the same statement for $a$ follows by a symmetrical
argument.  Now, assume that $\pi(b)\neq 0$; by injectivity,
$\iota\pi(b)=g_{\Bj}^*g_{\Bj}bg_{\Bi}^*g_{\Bi}\neq 0$ as well.  Thus,
by the non-degeneracy of $\tau$  \cite[2.26]{Webmerged} , there exists an $a$, such that $\langle
g_{\Bj}ag_{\Bi}^*,g_{\Bj}bg_{\Bi}^*\rangle\neq 0$; so this new pairing
is non-degenerate as well.

Note that furthermore, the adjoint under this action of right
multiplication by $c$ is left multiplication by $c^*$ since 
\[\langle
cg_{\Bj}ag_{\Bi}^*,g_{\Bj}bg_{\Bi}^*\rangle= \tau(a^*g_{\Bj}^*c^*g_{\Bj}bg_{\Bi}^*g_{\Bi})=\langle
g_{\Bj}ag_{\Bi}^*,c^*g_{\Bj}bg_{\Bi}^*\rangle\] and similarly for
right multiplication.  Since this is a non-degenerate invariant pairing, we
have proven the self-duality of this module.

The statement on tiltings and injectives is equivalent to the adjoint
$\mathbb{R}\operatorname{Hom}( \bra^{\vartheta,-\vartheta},-)$ sending
injectives to tiltings.  This functor sends the duals of projectives
to the duals of tiltings, so we are done. 
\end{proof}
\begin{corollary}\label{Ringel-dual}
  The Ringel dual of $\mathcal{S}^\vartheta_\nu$ is $\mathcal{S}^{-\vartheta}_\nu$.
\end{corollary}
Note that in our notation for Uglov weightings, this implies that
$\mathcal{S}^{\vartheta_{\Bs}^+}_\nu$ is Ringel dual to
$\mathcal{S}^{\vartheta_{\Bs}^-}_\nu$.  By Proposition
\ref{prop:symmetry}, this is in turn isomorphic to $\mathcal{S}^{\vartheta_{\Bs^\star}^+}_{\nu^\star}$ where $\nu^\star$
is the image of $\nu$ under the diagram automorphism induced by
$i\mapsto -i$ on $\C/\Z$.
\begin{lemma}\label{lem:c-o-c-equiv}
  The functor $\bra^{\vartheta,\vartheta'}\Lotimes -$ induces an equivalence of categories.
\end{lemma}
\begin{proof}
  We can reduce to the case where $\vartheta= -\vartheta'$;
  since any weighting is between this pair, functors of this form
  factor through $\bra^{\vartheta,\vartheta'}\Lotimes -$ on the right
  and the left.  Thus, if all the functors when
  $\vartheta'=-\vartheta$ are equivalences, the desired result will
  follow.

  Since $\bra^{\vartheta,-\vartheta}e_{\Bi}$ is a tilting
  module by Lemma \ref{lem:far-right-dual}, its Ext algebra is
  concentrated in homological degree 0 (i.e. there are no higher
  Ext's).  The functor $\bra^{\vartheta,-\vartheta}\Lotimes -$
  induces a map
  \begin{equation}
e_{\Bi}T^\vartheta e_{\Bj}\to
  \Hom(\bra^{\vartheta,-\vartheta}e_{\Bj},\bra^{\vartheta,-\vartheta}e_{\Bi}).\label{eq:bra-action}
\end{equation}
By the vanishing of higher Exts, it suffices to
  prove that this map is an isomorphism.  

We already know that the dimension of the left hand side is \[\dim (e_{\Bi}T^\vartheta e_{\Bj})=\sum_\xi
[T^\vartheta e_{\Bi} :S_\xi][T^\vartheta e_{\Bj} :S_\xi]\] by BGG
reciprocity.  On the other hand, the dimension of the right hand side
is \[\dim \Hom(\bra^{\vartheta,-\vartheta}e_{\Bj},\bra^{\vartheta,-\vartheta}e_{\Bi})=\sum_\xi
[S_\xi':\bra^{\vartheta,-\vartheta}
e_{\Bi}][S_\xi':\bra^{\vartheta,-\vartheta} e_{\Bj}]\] since the
multiplicities of the standard and costandard filtrations on a tilting coincide.
Thus, the equality of dimensions follows immediately from
  the fact that the (co)standard multiplicities of
  $\bra^{\vartheta,-\vartheta}e_{\Bj}$ coincide with those of
  $T^\vartheta e_{\Bj}$ by Corollary \ref{cor:B-multiplicities}.

  Thus, we need only show that this map is injective.  It's a
  consequence of the $-1$-faithfulness of Lemma
  \ref{lem:Hecke-loadings} that any $b\neq 0\in T^\vartheta$, thought of as an endomorphism by right multiplication, must still
  act non-trivially on $e^0 T^\vartheta$; that is, there must exist
  $a$ such that $e^0ab\neq 0$.  Since $T^\vartheta$ is self-opposite,
  we can apply the same result on the right to $e^0ab$, and see that there is a
  $c$ with $e^0abce^0\neq 0$.  

Thus, if $b$ is an element of the kernel of the map
\eqref{eq:bra-action}, then $e^0abce^0$ will be as well, and we can assume
without loss of generality that $e_{\Bi}e^0\neq 0$, and
$e^0e_{\Bj}\neq 0$, which is the same as to say that $\Bi$ and $\Bj$ are
  Hecke loadings.  But, in this case
  $\bra^{\vartheta,-\vartheta}e_{\Bj}\cong T^{-\vartheta}e_{\Bj}$, so
  we just obtain the induced isomorphism \[e_{\Bi}T^\vartheta e_{\Bj}\cong e_{\Bi}T^\lambda e_{\Bj}\cong e_{\Bi}T^{-\vartheta} e_{\Bj}\cong
  \Hom(\bra^{\vartheta,-\vartheta}e_{\Bj},\bra^{\vartheta,-\vartheta}e_{\Bi}).\]
  This completes the proof.
\end{proof}
This shows that the derived category of $T^\vartheta\mmod$ only
depends on the highest weight $\la$ and not on $\vartheta$ itself
(though these different categories are not {\it canonically equivalent}).
Combining Lemma \ref{lem:c-o-c-equiv} and Theorem
\ref{thm:cherednik-KLR} implies that:
\begin{corollary}
  If the charges $\Bs$ and $\Bs'$ are in the same orbit of
  $\widehat{B}_\ell$, i.e. their KZ functors land in the same block of
  the Hecke algebra, then the categories $D^b(\cO^{\Bs})$ and
  $D^b(\cO^{\Bs'})$ are equivalent.
\end{corollary}

Recall that if we have an exceptional collection $(\Delta,\leq )$, and we
choose a new order $\leq'$ on the the collection, there is a unique
new exceptional collection $(\Delta',\leq')$ with the same indexing set,
such that $\Delta_i'$ lies in the triangulated category generated by
$\{\Delta_j\}_{j\geq' i}$ and $\Delta_i'\equiv \Delta_i$ modulo the
triangulated category generated by $\{\Delta_j\}_{j >' i}$. We call
this the {\bf mutation} of the exceptional collection by this change
of partial order.  Let $d^{\vartheta,\vartheta'}_\xi$ be the degree of
the basis vector $D_{\sT}$ for the tautological tableau on the
multipartition $\xi$.
\begin{lemma}\label{lem:cat-mutate}
  The image of the standard exceptional collection in $\mathcal{S}^{\vartheta'}$ under
  $\bra^{\vartheta',\vartheta}\Lotimes -$ is the mutation of the
  shifted standard collection $S_\xi(-d^{\vartheta,\vartheta'}_\xi)$ in $\mathcal{S}^{\vartheta}$ for the induced
  change of partial order. 
\end{lemma}
\begin{proof}
  We prove this by induction on the partial order for $\vartheta'$,
  which we denote $\leq'$ (matching the role it plays in the
  definition of mutation above).  If $\xi$ is maximal, then $S_\xi'$ is
  projective, and $\bra^{\vartheta',\vartheta}\Lotimes S_\xi'=S_\xi(-d^{\vartheta,\vartheta'}_\xi)$.  

For $\xi$ arbitrary, we have that by induction, the image of the
category generated by $S_\eta'$ with $\eta>' \xi$ is the same that
generated by  $S_\eta$ with $\eta>' \xi$.  Since  $P_\xi'\equiv
S_\xi'$ modulo the subcategory generated by  $S_\eta'$ with $\eta\geq'
\xi$, we have that  $\bra^{\vartheta',\vartheta}\Lotimes P_\xi'\equiv
\bra^{\vartheta',\vartheta}\Lotimes S_\xi'$ modulo  $S_\eta$ with
$\eta\geq' \xi$.  

On the other hand, the standard filtration on
$\bra^{\vartheta',\vartheta}\Lotimes P_\xi'$ makes it clear that it
lies in the subcategory generated by $S_\eta$ with $\eta\geq' \xi$ and
is equivalent to $S_\xi(-d^{\vartheta,\vartheta'}_\xi)$ modulo  $S_\eta$ with $\eta>' \xi$.  Thus,
the same statements hold for $\bra^{\vartheta',\vartheta}\Lotimes S_\xi'$, and we are done.
\end{proof}

Following Bezrukavnikov \cite[Prop. 1]{Bez03}, we can reconstruct the entire $t$-structure
of $D^b(T^\vartheta\mmod)$ just from the exceptional collections $S_*$
and $S^\star_*$; there is a unique $t$-structure containing both of
these sets of modules in its heart.  This gives us a description of
the image of the standard $t$-structure on $D^b(T^{\vartheta'}\mmod)$
under $\bra^{\vartheta',\vartheta}\Lotimes-$.

\begin{proposition}
  The equivalence $\bra^{\vartheta',\vartheta}\Lotimes-$ sends the
  standard $t$-structure on $D^b(T^{\vartheta'}\mmod)$ to the unique
  $t$-structure whose heart contains the mutation of $S_\xi$ and inverse
  mutation of $S_\xi^*$ for the new ordering $>'$.
\end{proposition}

Note that if we replace a weighting by its Uglovation, no boxes with
the same residue switch order,  so weighted partial order does
not change.  Thus, we have that:
\begin{corollary}\label{cor:Uglovation-Morita}
  If $\vartheta_{\Bs}$ is the Uglovation of $\vartheta$, the bimodule
  $\bra^{\vartheta_{\Bs},\vartheta}$ induces a Morita equivalence.
\end{corollary}

Let $V$ be a free $\Z[q,q^{-1}]$-module of finite rank, equipped with a
sesquilinear form $\langle -, -\rangle$ and an antilinear bar
involution such that $\langle \bar u, \bar v\rangle=\langle
  v,u\rangle$.  A {\bf semi-orthonormal basis} of $V$ is a partially
ordered $\Z[q,q^{-1}]$-module basis $\{v_i\}_{i\in(I,\leq)}$ such
that $\langle v_i,v_j\rangle=0$ if $j\nleq i$, and $\langle v_i,v_i\rangle=1$.  

If $\leq'$ is another partial order on $I$, then $v_*$ possesses a unique {\bf mutation} 
to another semi-orthogonal basis $\{v_i'\}$ indexed by $I$, this time
endowed with $\leq'$, such that
\begin{equation}
v_i'\in
\operatorname{span}\{v_j\}_{j\geq' i} \qquad\qquad v_i'\equiv v_i
\pmod{\operatorname{span}\{v_j\}_{j>' i}}.\label{eq:vector-mutate}
\end{equation}

It follows from \cite[Prop. 4.11]{Uglov} that the
standard basis $u_\xi$ is semi-orthogonal for the weighted dominance
order for $\vartheta$.

\begin{proposition}\label{prop:classes-mutate}
  The classes $[\bra^{\vartheta',\vartheta}\Lotimes S_\xi']$ are the
  mutation of the semi-orthogonal basis
  $q^{-d^{\vartheta,\vartheta'}_\xi}[S_\xi]$ from $\vartheta$- to
  $\vartheta'$-weighted dominance order. 
\end{proposition}
\begin{proof}
  The properties of a highest weight category guarantee that
  $[\bra^{\vartheta',\vartheta}\Lotimes S_\xi']$ and
  $q^{-d^{\vartheta,\vartheta'}_\xi}[S_\xi]$ form semi-orthogonal
  bases for the appropriate orders.  The definition of mutation of
  exceptional collections directly corresponds to the conditions
  \eqref{eq:vector-mutate}, so the result follows by uniqueness of mutations.
\end{proof}

\subsection{The affine braid action}
\label{sec:mutat-affine-braid}

In this section, for the sake of simplicity, we'll only consider the
case of an 
Uglov weighting, that is, we'll assume that $U$ is
connected.  
The affine Weyl group $\widehat{W}_\ell$ of rank $\ell$ acts on the
set of Uglov weightings.  In terms of the variables $s_*$, we have  
\[\sigma_i\cdot ( s_1,\dots, s_\ell)=( s_1,\dots,
s_{i+1}, s_i,\dots,
s_\ell)\]
\[
\sigma_0\cdot ( s_\ell+ e,s_2,\dots,s_{\ell-1}, s_1-e).\]
The effect on weightings is that it leaves $\ck$ unchanged, and acts
on the weights $(\vartheta_1,\dots, \vartheta_\ell)$ of the new edges
by
\[\sigma_i\cdot ( \vartheta_1,\dots, \vartheta_\ell)=( \vartheta_1,\dots,
\vartheta_{i+1}-\ck e/\ell, \vartheta_i+\ck e/\ell,\dots,
\vartheta_\ell)\]
\[
\sigma_0\cdot (
\vartheta_\ell-\ck e(-1+1/\ell),\vartheta_2,\dots,\vartheta_{\ell-1},
\vartheta_1+\ck e(-1+1/\ell))\]

For each element of the affine Weyl group, we have an induced
bijection between the sets of weighted multi-partitions, permuting
them in the obvious manner.

We can lift this action of the affine Weyl group to the Fock spaces,
at the cost of making it an action of the affine  braid group
$\widehat{B}_\ell$; as discussed above, the
sum $\oplus_\Bs \mathbf{F}_\Bs$ where the sum is over all
Uglov weightings carries an action of the quantum group
$U_q(\mathfrak{\widehat{sl}}_\ell)$ which commutes with the $U_q(\gu)$
action.  There is a natural map of the affine braid group
$\hat{B}_\ell \to \dot U_q(\mathfrak{\widehat{sl}}_\ell)$ called the {\bf
  quantum Weyl group}.  This map sends $\sigma_i$ to the element $t_i$ that acts
on an element $v$ of weight $\mu$ by 
\[t_i\cdot v=\sum_{\substack{a,b\geq 0\\ a-b=\mu^i}} q^{-a}F_i^{(b)}E_i^{(a)}v.\]

\begin{lemma}\label{lem:quantum-weyl}
  Each of the generators $t_i^{-1}$ for $i=0,\dots, \ell-1$ induces an isometry
  $\mathbf{F}_{\vartheta}\to \mathbf{F}_{\sigma_i\cdot \vartheta}$ that
  sends the standard basis of $\mathbf{F}_{\vartheta}$ to the
   mutation of the shifted standard basis $\{q^{-d^{\vartheta,\sigma_i\cdot
       \vartheta}} u_\xi\}$ of $\mathbf{F}_{\sigma_i\cdot \vartheta}$
  for the order change from $\sigma_i\cdot \vartheta$-weighted dominance
  order to $\vartheta$-weighted dominance order.  
\end{lemma}
\begin{proof}
  Since $t_i$ lies inside the completion of the quantum universal
  enveloping algebra of the root $\mathfrak{sl}_2$ for $i$, we need
  only study the action of this subalgebra on $\oplus_\Bs
  \mathbf{F}_\Bs$.   Let $t_i^\vee$ be the
adjoint of $t_i$.  The map  $t_i^\vee t_i$ is an endomorphism by
the commutation relations of \cite[Thm. 8.1.2]{CP}.  

To show that $t_i$ is an isometry, it suffices to show that  on 
any simple submodule with highest weight $q$ for this root
$\mathfrak{sl}_2$, the endomorphism $t_i^\vee t_i$ is trivial.  On this submodule, the form $\langle-,-\rangle$ must restrict to a
multiple of $q$-Shapovalov form.  We have 
\begin{equation*}
  \langle v_h,t_i^\vee t_i v_h\rangle=\langle t_iv_h, t_i
  v_h\rangle=\langle F_i^{(q)}v_h, F_i^{(q)} v_h\rangle=\langle v_h,
  E_i^{(q)}F_i^{(q)} v_h\rangle=\langle v_h, v_h\rangle.
\end{equation*}
Since $t_i^\vee t_i v_h$ is also a highest weight vector, and the
space of such vectors is 1-dimensional, this shows that $t_i^\vee t_i
v_h=v_h$, so $t_i^\vee t_i$ must be the identity map.

Now, we turn to showing the required triangularity.  Recall that we think of this sum as a semi-infinite wedge power of
$A^e\otimes A^\ell[t,t^{-1}]$; restricted to $U_q(\mathfrak{sl}_2)$, this representation breaks up as a sum
of infinitely many copies of the trivial representation, and the
standard representation on $A^2$.  The semi-infinite wedge space is
just a sum of terms where we take an infinite number of trivial
factors (either from trivial summands of $A^e\otimes
A^\ell[t,t^{-1}]$, or from $\bigwedge{}^{\!\!2}A^2$) and a finite number
of factors of the form $A^2$.  That is, $\oplus_\Bs \mathbf{F}_\Bs$ is
a sum of infinitely many summands, each of which is a tensor product
of finitely many copies of $A^2$, with the standard basis matching the
usual pure tensor basis of the tensor product.  If we use rank-level duality to
index the basis of $\oplus_\Bs \mathbf{F}_\Bs$ with charged
$e$-multipartitions (so $\sllhat$ acts as in \eqref{eq:F-E-action}),
then these summands correspond to partitions with no removable boxes
of residue $i$, and the factors correspond to the addable boxes with
these residues (this is what Losev refers to as a {\bf family}
structure for this $U_q(\mathfrak{sl}_2)$ action).  

Thus, we need only show the required triangularity in this case.  This
follows from the multiplicative formula for quasi-R-matrix.  In our
notation, we have that the quasi-R-matrix of $U_q(\mathfrak{sl}_2)$
embedded in the root subalgebra is given that
$\theta_i=\Delta(t_i^{-1})(t_i\otimes t_i)$ by \cite{KR90}; note
that the precise formula here depends on the chosen coproduct.
Generalizing to the $m$-fold case, and moving terms, we have that 
\begin{equation*}
  \Delta^{(m)}(t_i)=(t_i\otimes \cdots \otimes t_i)(\theta_i^{(m)})^{-1}.
\end{equation*}
Since $\theta_i= 1+\theta^{(1)}_i+\theta^{(2)}_i+\cdots $ where
$\theta^{(k)}_i$ is lies in the tensor product of the $k\al_i$ and
$-k\al_i$ weight spaces,
\begin{equation}
(\theta_i^{(m)})^{-1}u_\xi=
u_{\xi}+\sum_{\xi<\xi'} a_{\xi'} u_{ \xi'}\label{eq:theta-action}
\end{equation}
for some $a_{\xi'}$.  
The action of $t_i$ on the standard representation is easily
calculated: it sends $v_h$, a highest weight vector $v_\ell:=\dot f_i v_h$, and $t_i\cdot v_\ell=q^{-1} v_h$.  Thus,
\begin{equation}
(t_i\otimes
\cdots \otimes t_i)\cdot u_{\xi'}=q^{\frac{m-\mu^i}{2}} u_{\sigma_i\cdot
  \xi'} \label{eq:t-action}
\end{equation} 
Combining \eqref{eq:theta-action} and \eqref{eq:t-action}, we see that 
\begin{equation}
  \Delta^{(m)}(t_i)u_\xi=
q^{\frac{m-\mu^i}{2}} u_{\xi}+\sum_{\xi<\xi'} q^{\frac{\mu^i-m}{2}} a_{\xi'} u_{ \xi'}\label{eq:delta-action}
\end{equation}
To complete the proof, we must show that $\frac{m-\mu^i}{2}$ is equal
to $d^{\vartheta,\sigma_i\cdot
       \vartheta}$.  
The diagram $\mathsf{D}_{\sT}$ given by the tautological tableau
crosses the strands corresponding to columns with the same $x$-value in the $i$th and
$(i+1)$st components of the multitableau. All the strands that cross
have the same residue, so the contribution of this crossing is  $-2$
times the product of the number of boxes in the two columns, plus this
number in the $i$th tableau, times the number in the column $\ck$
units to the
left in the $(i+1)$st component, plus this number with components reversed.

First, we verify that if the two partitions are empty, then the result
is correct.  In this case, we either have $\mu^i=\pm m$, depending on
whether the red lines are converging or diverging, and our convention
for the degree of $\mathsf{D}_{\sT}$ assures we have the right answer.
When we add a box in a column in $\xi^{(i+1)}$, we add in the corresponding strand, and we get a
contribution that depends on the corresponding column in $\xi^{(i)}$:
we get the sum 
of
\begin{itemize}
\item the number of boxes in adjacent columns to the left and right,
  because of the crossings of strands and ghosts
\item $-2$ times the number in the column itself, because of the
  crosssings with strands of the same label
\item an
  additional one if we are at the central column of the component
  containing $(1,1,i)$,
  from the red strand.
\end{itemize}
This shows that, we get 1
if the corresponding column in $ \xi^{(i)}$ has an addable box, -1 if
it has a removable box, and 0 otherwise.
\begin{itemize}
\item In the case of an addable box in $ \xi^{(i)}$, adding a box in
  $\xi^{(i+1)}$ increases $m$ by 2 and leaves $\mu^i$ invariant.
\item In the case of an removable box in $ \xi^{(i)}$, adding a box in
  $\xi^{(i+1)}$ decreases $m$ by 2 and leaves $\mu^i$ invariant.
\item In the case whether there is neither, adding a box in
  $\xi^{(i+1)}$ leaves both $m$ and $\mu^i$ invariant.
\end{itemize}
Thus, we have the desired change of degree as we add boxes in
$\xi^{(i+1)}$.  Our argument is symmetric in $\xi^{(i+1)}$ and
$\xi^{(i)}$, so we can also add boxes in $\xi^{(i)}$, until we have
arrived at the desired components. 

Since the element $t_i$ acts as an isometry in the pairing $\langle
-,-\rangle$, this is again an semi-orthonormal basis, and thus agrees
with the mutation of $q^{\frac{\mu^i-m}{2}} u_{\sigma_i\cdot \xi} =q^{-d^{\vartheta,\sigma_i\cdot
       \vartheta}} u_{\sigma_i\cdot \xi}$.  
\end{proof}

\begin{theorem}\label{thm:braid-action}
The functors $\mathbb{B}_{\sigma_i}=\bra^{\vartheta,\sigma_i\cdot
  \vartheta}\Lotimes - $
  define a strong action of the affine braid group on the categories
  $D(\mathcal{S}_\vartheta)$ where $\vartheta$ is summed over all
  Uglov weightings, categorifying the action of the quantum Weyl group
  of $\mathfrak{\widehat{sl}}_\ell$.
\end{theorem}
\begin{proof}
We apply Lemma \ref{lem:between} in order to check the braid
relations. For any positive lift $w$ of an element of the affine symmetric
group, and any factorization $w=w'w''$ into positive elements, we have
that $w''\vartheta$ is between $w\vartheta$
and $\vartheta$.  Thus, by   Lemma \ref{lem:between}, we have that
$\mathbb{B}_{w'}\mathbb{B}_{w''}\cong \bra^{\vartheta,w\cdot
  \vartheta}\Lotimes -.$ This implies the braid relations and the
associativity of these isomorphisms shows that this action is strong.

Thus, we need only check the action on the Grothendieck group is
correct.  This follows immediately from comparing Lemma
\ref{lem:quantum-weyl} and Lemma \ref{lem:cat-mutate}; the action of
$\mathbb{B}_{\sigma_i}$ and of the quantum Weyl group both send the
standard basis to its mutant by the same change of order, so they coincide.
\end{proof}

\begin{remark}
  The same tensor product also induces actions on the categories
  $D^b(T^\vartheta\umod)$ of ungraded modules (by forgetting the
  grading) and on $D^b(T^\vartheta\dgmod)$ by considering all graded
  algebras and modules as complexes with trivial differential.  In
  both these cases, the conclusions of Theorem \ref{thm:braid-action}
  still hold.
\end{remark}

Note that we have a sort of dual braid group action, that arising from
Rickard complexes for $\slehat$, as in the work of Chuang and Rouquier
\cite[6.1]{CR04}. We'll use the inverse of Chuang and Rouquier's
functors, which act on objects of weight $n$ by the complex 
\[\cdots \to \eF^{(r+n-1)}\eE^{(r-1)}(r-1) \to \eF^{(r+n)}\eE^{(r)}(r) \to
    \eF^{(r+n+1)}\eE^{(r+1)}(r+1)\to \cdots\]
with $ \eF^{(r+n)}\eE^{(r)} $ the $r$th term in homological degree.  
This is an action of the affine braid group $\widehat{B}_e$
categorifying the quantum Weyl group action from $\gu$.  We denote
the functor associated to $\sigma\in \widehat{B}_e$ by
$\Theta_\sigma$. 

Recall that in a highest-weight categorification, the set of simple
objects is divided into {\bf families} (as discussed in \cite[\S
3.1]{LoHWCII}).  The simples in each family are in bijection with sign
vectors $\{+,-\}^m$ for some $m$ depending on the family.  In the
Grothendieck group, the corresponding standard modules must span a
copy of $(\mathbb{C}^2)^{\otimes m}$.    We let $m_\xi$ be this
statistic for the family containing $\xi$.  
\begin{lemma} \label{CR-mutate}
Consider any highest weight categorical $\mathfrak{sl}_2$-action.
Then $\Theta_s$ for the unique simple reflection $s$ sends the
exceptional collection of standard modules $S_{\xi}$ of weight $n$ to the mutation of the
exceptional collection $S_{\xi}[\frac{m_\xi-n}{2}]$ where we reverse order on
each family.
\end{lemma}
\begin{proof}
We need only check this for the unique highest weight categorification
of $(\C^2)^{\otimes m}$, since every highest weight categorification
has a filtration (compatible with standards and categorification
functors) with these as subquotients by \cite[Prop. 5.9]{LoHWCII}.   
There are many concrete models for such a categorification, for
example, as representations of the algebra $T^{\boldsymbol{\lambda}}$
for $\mathfrak{sl}_2$ introduced in \cite[\S 4]{Webmerged}, or as a
sum of 
singular blocks of category $\cO$ for $\mathfrak{sl}(m)$ 
corresponding to the subgroups $S_{k}\times S_{m-k}\subset S_m$.  All
of these are equivalent by the main result of \cite{LoWe}.

The standards in this case are
naturally indexed by sign sequences.  We let $\bar \xi$ of a sign
sequence denote the same sequence with $+$ and $-$ switched. 
By \cite[Prop. 7.3]{LoHWCII}, in this case, $\Theta$ sends projective
modules to tilting modules, up to shift.  Our definition of $\Theta$
is the inverse of a shift of
Losev's, so we obtain a shift where he has none.
By standard properties of Ringel duality, this means that $\Theta $
also sends standards to shifted
costandards.  By considering the
effect on the Grothendieck group, we see that $\Theta(S_\xi)=S_{\bar{\xi}}^*[\frac{m_\xi-n}{2}]$. In particular, it has the same
composition factors as $S_{\bar{\xi}}[\frac{m_\xi-n}{2}]$, and thus is in the
subcategory generated by $S_{\bar{\eta}}$ for $\eta\geq \xi$, and
equivalent to $S_{\bar{\xi}}[\frac{m_\xi-n}{2}]$ modulo that generated by
$S_{\bar{\eta}}$ for $\eta> \xi$. Since obviously the image of the
standards is an exceptional collection, these properties show it must
be the mutated one.
\end{proof}

\subsection{Canonical bases}
\label{sec:canonical-bases}
There is a natural duality $\psi$ on projective objects in $S^\vartheta$,
given by the anti-automorphism $*$.  More categorically, we can think
of this as $\Hom(-,T^\vartheta)$, which is naturally a right module,
given a left module structure via $*$.  We can extend this to derived
categories in the obvious way.
\begin{samepage}
  \begin{proposition}\label{prop:bar}\hfill
    \begin{enumerate}
    \item The functor $\psi$ categorifies the bar involution of Fock
      space.
    \item The sesquilinear inner products denoted $\langle -,-\rangle$
      on Fock space and the Grothendieck group coincide.
    \item The affine braid group action of Theorem
      \ref{thm:braid-action} categorifies the quantum Weyl group
      action.
    \end{enumerate}
  \end{proposition}
\end{samepage}

\begin{proof}
First, note that if $U$ is disconnected, then all of these structures
are induced by the tensor product decomposition of
\eqref{eq:Fock-tensor}, so we can immediately reduced to the connected
case, and assume that our weighting is Uglov.

Since \[\langle [M],[N]\rangle=([\psi M],[N])\qquad \langle
u,v\rangle=(\bar u,v)\]
and we already know that the forms $(-,-)$ coincide by Proposition
\ref{prop:symmetric-forms}, the statements $(1)$ and $(2)$ are equivalent.

For each $\vartheta$ and $\nu$, there exists some element of the
affine braid group $\pi$, such that $\pi\cdot \vartheta$ is
well-separated (in the sense of \cite[\S 3.3]{WebwKLR}) for $\nu$;
that is, the weights $\vartheta_i$ are sufficiently far apart that the
weighted dominance order on multipartitions of weight $\nu$, and thus the
category $\mathcal{S}^{\vartheta}_\nu$, will not change as we separate them further.  As proven in
\cite[3.6]{WebwKLR}, this algebra is Morita equivalent to the {\bf
  quiver Schur algebra} of \cite{SWschur}. We let $\ell(\vartheta,\nu)$ be the minimal
length of such an element.  We will prove the statements above by
induction on $\ell(\vartheta, \nu)$.  More precisely, our inductive
hypothesis will be
\begin{itemize}
\item [$(h_n)$] the inner products $\langle -,-\rangle$ agree for all
  $\vartheta$ and $\nu$ such that $\ell(\vartheta,\nu)\leq n$, and for
  any
  generator $\mathbb{B}_i$, the action when both $
  \ell(\vartheta,\nu)\leq n$
  and $\ell(\sigma_i\vartheta,\nu)\leq n$ agrees with the quantum Weyl group action.
\end{itemize}

When $n=0$, the category $\cS^\vartheta_\nu$ agrees with the
representations of a quiver Schur algebra as in \cite{SWschur};  thus,
statement $(1)$ and thus $(2)$ hold by \cite[7.19]{SWschur}.  Since we
have checked that the sesquilinear forms coincide, Proposition
\ref{prop:classes-mutate} and Lemma \ref{lem:quantum-weyl} describe
the effect of the change-of-charge functor and the quantum Weyl group action
in terms of the same mutations, so they coincide.  This principle is
the key of the proof: once we know that the forms $\langle -,-\rangle$
coincide on the image category, we know that the action of
$\mathbb{B}_i$ agrees with the quantum Weyl group.

Thus, we move to the inductive step $(h_{n-1}) \Rightarrow (h_n)$.  We
consider $\vartheta,\nu$ with $\ell(\vartheta,\nu)=n$.  Then, for some
generator $\sigma_i$, we have $\ell(\sigma_i\vartheta,\nu)=n-1$.  
We know that $t_iu_\xi=[\mathbb{B}_i
S_\xi]$ by Theorem \ref{thm:braid-action}. 
Thus, we have that 
\begin{equation}
\langle u_\xi,u_{\eta}\rangle \overset{\text{(i)}}=\langle
t_iu_\xi,t_iu_{\eta}\rangle\overset{\text{(ii)}}=\langle \mathbb{B}_i
S_\xi,\mathbb{B}_iS_{\eta}\rangle\overset{\text{(iii)}}=\langle
S_\xi,S_{\eta}\rangle,\label{eq:forms-match}
\end{equation}
where we use in turn (i) that $t_i$ is an isometry, (ii) the inductive
hypothesis establishes the coincidence of forms for
$\cS^{\sigma_i\vartheta}_\nu$, and (iii) that $\mathbb{B}_i$ induces an isometry
on Grothendieck groups.

This establishes claims $(1\operatorname{-}2)$, and claim $(3)$ for reflections that
decrease $\ell(\vartheta,\nu)$;  however, the cases where $\sigma_i$
increases or keeps $\ell(\vartheta,\nu)$ unchanged follow immediately
by the same argument.  We already know that the forms $\langle
-,-\rangle$ coincide in the target, so we may use the same argument as
\eqref{eq:forms-match}.  This establishes the theorem.
\end{proof}
The structure of $q$-Fock spaces together with their bar involution
leads to the definition of a canonical basis, as defined in \cite[\S
4.4]{Uglov}; this also fits in the framework for canonical bases
discussed in \cite{WebCB}.
\begin{definition}
Let $\{b_\xi\}$ be the unique bar invariant
  basis such that
  $b_\xi\in u_\xi+\sum_{\xi'<\xi}q^{-1}\Z[q^{-1}]u_\xi$; in the
  notation of \cite{Uglov}, this is $\mathcal{G}^-$.
\end{definition}

\begin{theorem}\label{th:canonical}
  The basis in $K^0(\mathcal{S}^\vartheta)$ given by the
  indecomposable projectives $P_\xi$ is identified under the
  isomorphism to twisted Fock space with Uglov's canonical basis $\{b_\xi\}$, and
  thus the basis of simples with the dual canonical basis.
\end{theorem}
\begin{proof}
  The projectives $P_\xi$ are invariant under $\psi$: the
  modules 
  $T^\vartheta e_\Bi$ are,  and when $\Bi=\Bi_\xi$, the indecomposable
  $P_\xi$ appears as a summand exactly once.  The highest weight
  structure shows that $b_\xi\in
  u_\xi+\sum_{\xi'<\xi}\Z[q,q^{-1}]u_\xi$.  Thus, we need only
  establish these coefficients are all polynomial in $q^{-1}$ with no
  constant term.  That is, that only
  positive shifts of standard modules appear in the standard
  filtration.  

  For this, it suffices to check that $\Hom(P_{\xi'},P_{\xi})$ is
  positively graded for $\xi'\neq \xi$.  We first note that the
  corresponding result holds for $\tilde{T}^{\vartheta}$, the algebra
  without the violating relation.  Let $P_{\xi'} $ be the projective
  cover of $P_{\xi'} $ as a module over this algebra.  The algebra
  $\tilde{T}^{\vartheta}$ is isomorphic to the Ext-algebra of a sum of
  shifts of semi-simple perverse sheaves on a version of the Lusztig
  quiver variety.  Thus, by
  \cite[Cor. 4.4]{WebwKLR}, the sum $\oplus_\xi \tilde{P}_\xi$ is a
  summand of a graded projective generator whose endomorphisms are
  positively graded, and so $\Hom(\tilde P_{\xi'},\tilde P_{\xi})$ is
  positively graded.  Since
  $\Hom(\tilde P_{\xi'},\tilde P_{\xi}) \to \Hom(P_{\xi'},P_{\xi})$ is
  a surjection by the lifting property of projectives, the latter is positively graded as well.
\end{proof}

This shows a diagrammatic analogue of Rouquier's conjecture.  By BGG
reciprocity, we have that the multiplicities
$[S_\xi:P_\eta]=[L_{\eta}:S_\xi]$ agree;  thus, it follows that 
have that:
\begin{corollary}\label{cor:Fock-decom}
  The graded decomposition numbers for $T^\vartheta$ agree with the
  coefficients of Uglov's canonical basis of Fock space $\mathbf{F}_\vartheta$ in terms of
  standard modules.  That is, for all $\eta $, we have that  \[b_\eta=\sum_{\xi}[S_{\xi}:P_\eta]u_{\xi}=\sum_{\xi}[L_\eta:S_{\xi}]u_{\xi}.\]
\end{corollary}

Transferring structure via the equivalence of Theorem
\ref{thm:cherednik-KLR}, we find that Corollary \ref{cor:Fock-decom} implies that:
\begin{corollary}[Rouquier's conjecture]
  The multiplicities of standard modules in projectives in
  $\cO^{\Bs}$, and thus by BGG reciprocity, the multiplicities of
  simples in standards, are the same as the coefficients of Uglov's
  canonical basis of a Fock space, specialized at $q=1$.
\end{corollary}

\section{Koszul duality}
\label{sec:koszul-duality}

Unlike the earlier sections, the results in this section depend on the ``categorical dimension conjecture'' of Vasserot and
Varagnolo  \cite[8.8]{MR2673739} that $\cO^{\vartheta}$ is equivalent to a truncated
parabolic category $\cO$ for an affine Lie algebra; as mentioned in
the introduction, this is proven in \cite{RSVV} and \cite{LoVV}. 
This conjecture shows, amongst other things, 
that $\cO^{\vartheta}$ and thus $T^\vartheta$ possess a Koszul grading.

\begin{remark}
  Papers of the author \cite{Webqui,WebSD} which have appeared since this article first
  became available as a preprint give an alternate approach to proving
  the Koszulity of $ T^\vartheta$, and the Koszul duality result of
  Theorem \ref{Koszul-dual}.  These use the realizations of these algebras
  in terms of category $\cO$ on affine quiver varieties, instead of
  the truncated affine category $\cO$ of \cite{MR2673739}.
\end{remark}

{\it A priori}, it is not clear that this Koszul grading is Morita
equivalent to the one that we've already defined; in fact, a
general uniqueness property of Koszul gradings shows this.  To clarify:
\begin{definition}
  We call a finitely dimensional graded algebra $A$ {\bf Koszul} if it
  is graded Morita equivalent to a positively graded algebra $A'$
  which is Koszul in the usual sense; we call a graded abelian
  category {\bf Koszul} if it is equivalent to the category of graded
  modules over a Koszul algebra.
\end{definition}

\begin{theorem}\label{Koszul-grading}
The usual grading on $T^\vartheta$ is Koszul, and the equivalence of
Theorem~\ref{thm:cherednik-KLR} induces an
equivalent graded lift of $\cO^{\vartheta}$ to the grading on
category $\cO$.  In particular, $\cS^\vartheta$ is standard Koszul and balanced.
\end{theorem}
\begin{proof}
  By the numerical criterion of Koszulity \cite[2.11.1]{BGS96}, if an algebra has one
  Koszul grading, then any other grading with the same graded Cartan
  matrix is again Koszul, and in fact graded Morita equivalent to the
  first Koszul grading.  Thus, any grading on $T^\vartheta$ whose
  Cartan matrix is the matrix expressing Uglov's canonical basis in
  terms of its dual is a Koszul grading, since the grading induced
  from the truncated parabolic category $\cO$ has this property by \cite[8.2]{MR2673739}.  By Corollary
  \ref{cor:Fock-decom}, this is the case for the diagrammatic grading
  on $T^\vartheta$ as well.  Thus, the grading on $T^\vartheta$ is Koszul.  Similarly,
  $T^\vartheta$ is balanced and standard Koszul (the latter being part
  of the definition of the former) by \cite[4.3]{SVV}.
\end{proof}

In the case where $T^{\vartheta}$ is Morita equivalent to a quiver
Schur algebra (as shown in \cite[Th. A]{WebwKLR}), this Koszulity has
been established independently by Maksimau in forthcoming work
\cite{Maks}.  As mentioned before, we give an independent and
different proof in \cite{Webqui}.

Now we turn to describing the Koszul dual of
$T^\vartheta$; for simplicity, we only do this in the case where
$U$ is a $e$-cycle, so $\gu=\slehat$.
Consider an $\ell\times e$ matrix of integers $U=\{u_{ij}\}$, and
let $s_i=\sum_{j=1}^e u_{ij}$ and $t_j=\sum_{i=1}^\ell u_{ij}$ and an
integer $w$.  We wish to consider the former as an Uglov weighting for
$\slehat$, and the latter for $\sllhat$.

Associated
to each row of $U$, we have a charged $e$-core partition; we fill an
abacus with beads at the positions $(u_{ij}-a)e +j$ for $j=1,\dots,
e$ and all $a\in \Z_{\geq 0}$, and take the partition described by this abacus.
Let $v_i$ be
the unique integer such that $v_i-w$ is the
total
number of boxes of residue $i$ in all these partitions.  We
wish to consider the algebra $T^{\vartheta_{\Bs}^+}_{\Bt,w}:=T^{\vartheta_{\Bs}^+}_{\mu}$  
and
$\cS^{\vartheta}_{\Bt, w}:=\cS^\vartheta_\mu$ with weight
$\mu:=\la-\sum v_i\al_i$.  We note that by Proposition
\ref{prop:symmetry}, we have an equivalence
$\cS^{\vartheta_\Bs^\pm}_{\Bt, w}\cong \cS^{\vartheta_{\Bs^\star}^\mp}_{\Bt^\star, w}$.

\begin{theorem}\label{Koszul-dual}
  The Koszul dual of $\cS^{\vartheta_{\Bs}^\pm}_{\Bt,w}$ is
  $\cS^{\vartheta_{\Bt}^\mp}_{\Bs,w}\cong \cS^{\vartheta_{\Bt^\star}^\pm}_{\Bs^\star, w}$.
\end{theorem}
\begin{proof}
The result \cite[7.4]{RSVV}
 implies that $\cO^{\Bs}_{\Bt,w}$ and $\cO^{\Bt^\star}_{\Bs^\star,w}$
  are Koszul dual.  Translating to diagrammatic algebras, this implies that
  $\cS^{\vartheta_\Bs^+}_{\Bt,w}$ and $\cS^{\vartheta_{\Bt^\star}^+}_{\Bs^\star,w}$
  are Koszul dual.  
\end{proof}

We can visualize the combinatorial bijection between simple modules in
these two Koszul dual categories.  To a simple in
$\cS^{\vartheta_{\Bs}^+}$, we can associate a charged $\ell$-multipartition,
and thus an $\ell$-runner abacus.  We place the runner for the new
edge $e^1$ at the bottom, and the list them in ascending order.  The
duality map works by cutting this abacus into rectangles along the
vertical lines between $ae$ and $ae+1$ for $a\in \Z$, and then
flipping along the SW/NE diagonal. That is, the runner corresponding
to $e^j$ becomes the beads in the positions $ae+j$.  This reverses the
roles of $\ell$ and $e$.

\begin{figure}
  \centering
  \begin{tikzpicture}
\node (a) at (-3,0){
    \begin{tikzpicture}[very thick,scale=.5]
      \draw[thick, densely dashed] (-.5,1.6) -- (-.5,-.6);
      \draw[thick, densely dashed] (-3.5,1.6) -- (-3.5,-.6);
      \draw[thick, densely dashed] (2.5,1.6) -- (2.5,-.6);
      \node[circle, draw, inner sep=3pt] at (0,0){}; \node[circle,
      draw, inner sep=3pt,fill=black] at (0,1) {}; \node[circle, draw, inner
      sep=3pt,fill=black] at (1,0) {}; \node[circle, draw, inner sep=3pt] at
      (1,1) {}; \node[circle, draw, inner sep=3pt] at (2,0) {};
      \node[circle, draw, inner sep=3pt,fill=black] at (2,1) {}; \node[circle,
      draw, inner sep=3pt] at (3,0) {}; \node[circle, draw, inner
      sep=3pt] at (3,1) {}; \node[circle, draw, inner sep=3pt,fill=black] at
      (-1,0) {}; \node[circle, draw, inner sep=3pt,fill=black] at (-1,1) {};
      \node[circle, draw, inner sep=3pt] at (-2,0) {}; \node[circle,
      draw, inner sep=3pt,fill=black] at (-2,1) {}; \node[circle, draw, inner
      sep=3pt,fill=black] at (-3,0) {}; \node[circle, draw, inner sep=3pt] at
      (-3,1) {}; \node[circle, draw, inner sep=3pt,fill=black] at (-4,0) {};
      \node[circle, draw, inner sep=3pt,fill=black] at (-4,1) {}; \node at
      (-5.2,.5) {$\cdots$}; \node at (4.2,.5) {$\cdots$};
    \end{tikzpicture}};
\node (b) at (3,0){    \begin{tikzpicture}[very thick,scale=.5]
      \draw[thick, densely dashed] (-.5,1.6) -- (-.5,-1.6);
      \draw[thick, densely dashed] (-2.5,1.6) -- (-2.5,-1.6);
      \draw[thick, densely dashed] (1.5,1.6) -- (1.5,-1.6);
      \node[circle, draw, inner sep=3pt,fill=black] at (0,0){}; 
      \node[circle, draw, inner sep=3pt] at (0,1) {}; 
      \node[circle, draw, inner sep=3pt] at (1,0) {}; 
      \node[circle, draw, inner sep=3pt,fill=black] at (1,1) {}; 
      \node[circle, draw, inner sep=3pt] at (2,0) {};
      \node[circle, draw, inner sep=3pt] at (2,1) {}; 
      \node[circle, draw, inner sep=3pt] at (-1,-1) {}; 
      \node[circle, draw, inner sep=3pt,fill=black] at (-2,-1) {}; 
      \node[circle, draw, inner sep=3pt,fill=black] at (-1,0) {}; 
      \node[circle, draw, inner sep=3pt,fill=black] at (-1,1) {};
      \node[circle, draw, inner sep=3pt] at (-2,0) {}; 
      \node[circle, draw, inner sep=3pt,fill=black] at (-2,1) {}; 
      \node[circle, draw, inner sep=3pt,fill=black] at (-3,0) {}; 
      \node[circle, draw, inner sep=3pt,fill=black] at (-3,1) {}; 
      \node[circle, draw, inner sep=3pt] at (0,-1) {};
      \node[circle, draw, inner sep=3pt,fill=black] at (1,-1) {}; 
      \node[circle, draw, inner sep=3pt,fill=black] at (-3,-1) {}; 
      \node[circle, draw, inner sep=3pt] at (2,-1) {}; 
      \node at (-4.2,0) {$\cdots$}; 
      \node at (3.2,0) {$\cdots$};
    \end{tikzpicture}};
\draw[<->,very thick] (a)--(b);
  \end{tikzpicture}

  \caption{The Koszul duality bijection}
  \label{fig:koszul}
\end{figure}
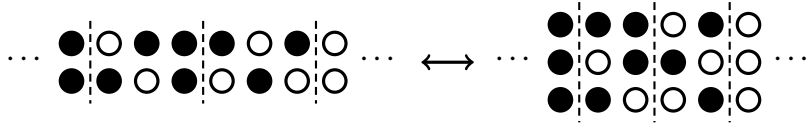
\begin{proposition}\label{koszul-bijection}
  This bijection between multipartitions matches that on simples
  induced by Koszul duality. 
\end{proposition}
\begin{proof}
In order to understand this duality, we must give the correspondence
between our combinatorics and that for affine Lie algebras as in the
work of Vasserot-Varagnolo \cite{MR2673739}.

  We associate a weight of an affine Lie algebra to an abacus diagram
  as follows: we cut off the diagram at some point to the far left of all boxes of
  the partition (i.e. left of which the abacus is solid).  We can 
  simultaneously shift all the $s_i$, so we can assume that we
  cut off all the beads at negative positions, so we have exactly $s_i$
  dots remaining on the $i$th runner, and $N=\sum s_i$ total dots. We read the $x$-coordinates of
  the dots on each runner in turn (all on the first, then all on the
  second, etc.), which
  gives us an $N$-tuple which we denote $(a_1,\dots,
  a_N)$ (this matches the notation in \cite[\S 2.3]{LoVV}).  The
  affine Weyl group $\widehat{S}_N$ acts on this set with the level
  $e$-action (i.e. the ``translation'' adding $e$ to one coordinate and subtracting $e$
  from another is an element of the Weyl group).  We let $y$ be the
  unique minimal length element of this group that sends the sequence  $(a_1,\dots,
  a_N)$ to an element of the fundamental alcove (all entries are
  increasing and 
  between $1$ and $e$).   Visually, we can think of $y$ the element
  that switches
  \begin{itemize}
  \item from the order induced on dots by reading leftward on each
    runner in order
\item to
    that induced by reading across the runners from the first to the
    $\ell$th, first reading all dots in position $\dots,
    2e+1,e+1,1,-e+1,-2e+1,\cdots$ starting at the greatest $x$
    position that appears, then at $x$-coordinates congruent to
    $2\pmod e$, etc.
  \end{itemize}

  By \cite[2.16]{SVV}, the weight of the Koszul dual simple is obtained by applying the
  element $y$ in the level $\ell$ action to the element of the
  fundamental alcove given by $s_1$ instances of $1$, then $s_{2}$
  instances of $2$, etc.  This is given by the flip map we have
  described, since this switches the reading down runners used to
  obtain $(a_1,\dots,
  a_N)$ with the reading across runners that gives $y$, and preserves
  how shifted from the fundamental alcove a dot is (this matches with
  taking the inverse since we have gone from level $e$ to level
  $\ell$).  
\end{proof}


Alternatively, we can describe this map by decomposing this abacus
further into one with runners corresponding to each entry of an
$\ell\times e$ matrix; the runners of our previous description
correspond to the rows, and the runner for the $j$th column is gotten
by taking the beads (or lack of beads) at positions $ae+j$.  In this
case, the duality map is gotten by transposing the matrix of runners.

\excise{ Recall that the categories $D^b(\cS^{\vartheta_{\Bs}}_{\Bt,w})$ carry
commuting actions of $\widehat{B}_e$ by the Chuang-Rouquier braid
functors and of $\widehat{B}_\ell$ by change-of-charge functors.  The
former acts preserving $\Bs$ and on $\Bt$ by the usual action, and the
latter preserves $\Bt$ and acts on $\Bs$ by the usual action.

\begin{conjecture}
  The Koszul duality between $\cS^{\vartheta_{\Bs}}_{\Bt,w}\mmod$ is
  $\cS^{-\vartheta_{\Bt}}_{\Bs,w}\mmod$ interchanges the
  Chuang-Rouquier braid functors and change-of-charge functors.
\end{conjecture}
\begin{proof}
 Lemmas \ref{lem:cat-mutate} and \ref{CR-mutate} show that these
  induce the same $t$-structure, since Koszul duality switches
  standards, and intertwines grading shift with shift in the
  triangulated category.  So, the Koszul twist of the Chuang-Rouquier
  functors send each projective to the corresponding projective in the
  heart of the $t$-structure;  thus, the composition of this functor
  with its corresponding change-of-charge functor sends projectives to
  projectives.  That is, it is thus induced by a graded algebra
  automorphism of $T^{\vartheta_{\Bs}}_{\Bt}$.  The proof of this
will appear in \cite{Webqui}.
\end{proof}}

 \bibliography{../gen}
\bibliographystyle{amsalpha}

\end{document}